\documentclass[11pt]{amsart}
\usepackage[top=1.0in,bottom=1.0in,left=1.0in,right=1.0in]{geometry}
\usepackage{graphicx}
\usepackage{amssymb, amsthm,xcolor,mathrsfs, mathtools,stmaryrd}
\usepackage[T1]{fontenc}
\usepackage[utf8]{inputenc}
\usepackage{microtype}
\usepackage{epstopdf}
\usepackage{enumitem} 
\usepackage{tikz}
\usepackage{subfigure}
\newcommand{\be}{\begin{equation}}
\newcommand{\ee}{\end{equation}}

\newcommand{\bse}{\begin{subequations}}
\newcommand{\ese}{\end{subequations}}

\newcommand{\jbracket}[1]{\langle{#1}\rangle}

\DeclarePairedDelimiter{\parn}{\lparen}{\rparen}	

\newcommand{\augV}{{V}_c^{\mathrm{aug}}}

\newcommand{\R}{\mathbb{R}}

\newcommand{\Xspace}{\mathbb{X}}

\newcommand{\Wspace}{\mathbb{W}}
\newcommand{\Vspace}{\mathbb{V}}

\newcommand{\jump}[1]{\left\llbracket{#1}\right\rrbracket}

\newcommand{\DpsiS}{\mathcal{S}}
\newcommand{\DpsiT}{\mathcal{T}}
\newcommand{\eng}{E}
\newcommand{\Hc}{H_c}
\newcommand{\augHam}{E_c}
\newcommand{\mom}{P}
\newcommand{\potE}{V}
\newcommand{\kinE}{K}
\newcommand{\mi}{d}

\newcommand{\DN}{G}
\newcommand{\HE}{\mathcal{H}}
\newcommand{\Qform}{Q_c}
\newcommand{\NLop}{\widetilde{\mathcal{M}}^\pm_\varepsilon}
\newcommand{\NLopc}{\widetilde{\mathcal{M}}^\pm_c}
\newcommand{\symbm}{\mathfrak{m}}
\newcommand{\symbq}{\mathfrak{q}}
\newcommand{\symbr}{\mathfrak{r}}
\newcommand{\etaerror}{r}

\newcommand{\Lin}{\mathrm{Lin}}

\newcommand{\nbhdO}{\mathcal{O}}
\newcommand{\tube}{\mathcal{U}}

\newcommand{\gammaC}{\gamma}

\newcommand{\placeholder}{\,\cdot\,}
\newcommand{\ident}{\mathrm{Id}}
\newcommand{\Dom}[1]{\operatorname{Dom}{#1}}
\newcommand{\Rng}[1]{\operatorname{Rng}{#1}}

\newcommand{\realpart}[1]{\operatorname{Re}{#1}}
\newcommand{\imagpart}[1]{\operatorname{Im}{#1}}
\newcommand{\signum}[1]{\operatorname{sgn}{#1}}
\newcommand{\spectrum}[1]{\operatorname{spec}{#1}}
\newcommand{\essspectrum}[1]{\operatorname{ess\,spec}{#1}}

\newcommand{\sechsq}[1]{\operatorname{sech}^2{#1}}
\DeclareMathOperator{\sech}{sech}

\newcommand{\diffx}{\textup{d}x}
\newcommand{\diff}{\textup{d}}
\newcommand{\diffy}{\textup{d}y}
\newcommand{\Diff}{\textup{D}}

\numberwithin{equation}{section}

\theoremstyle{plain} 
\newtheorem{theorem}{Theorem}[section] 
\newtheorem{corollary}[theorem]{Corollary}

\newtheorem{lemma}[theorem]{Lemma}
\theoremstyle{remark}
\newtheorem{remark}[theorem]{Remark}

\title[Internal wave stability]{Orbital stability of internal waves}

\author[R. M. Chen]{Robin Ming Chen}
\address{Department of Mathematics, University of Pittsburgh, Pittsburgh, PA 15260} 
\email{mingchen@pitt.edu}  

\author[S. Walsh]{Samuel Walsh}
\address{Department of Mathematics, University of Missouri, Columbia, MO 65211, USA} 
\email{walshsa@missouri.edu} 

\date{\today}

\usepackage{hyperref}

\begin{document}

\begin{abstract}
This paper studies the nonlinear stability of capillary-gravity waves propagating along the interface dividing two immiscible fluid layers of finite depth.  The motion in both regions is governed by the incompressible and irrotational Euler equations, with the density of each fluid being constant but distinct.  A diverse collection of small-amplitude solitary wave solutions for this system have been constructed by several authors in the case of strong surface tension (as measured by the Bond number) and slightly subcritical Froude number.  We prove that all of these waves are (conditionally) orbitally stable in the natural energy space.  Moreover, the trivial solution is shown to be conditionally stable when the Bond and Froude numbers lie in a certain unbounded parameter region.  For the near critical surface tension regime, we prove that one can infer conditional orbital stability or orbital instability of small-amplitude traveling waves solutions to the full Euler system from considerations of a dispersive PDE model equation.

These results are obtained by reformulating the problem as an infinite-dimensional Hamiltonian system, then applying a version of the Grillakis--Shatah--Strauss method recently introduced in \cite{varholm2020stability}.  A key part of the analysis consists of computing the spectrum of the linearized augmented Hamiltonian at a shear flow or small-amplitude wave.  For this, we generalize an idea used by Mielke \cite{mielke2002energetic} to treat capillary-gravity water waves beneath vacuum. 
\end{abstract}

\maketitle

\setcounter{tocdepth}{1}
\tableofcontents

\section{Introduction} \label{introduction section}

We consider the classical problem of determining the evolution of a free boundary dividing two superposed incompressible, inviscid, and immiscible fluids under the influence of gravity.  This situation arises in countless applications, with a particularly important example being internal waves propagating along a pycnocline or thermocline in the ocean. Recent years have seen enormous progress made in understanding the Cauchy problem for this system, and there is now a robust (local) well-posedness theory.  In parallel, a large body of work has established the existence of myriad traveling wave solutions.  Far less is known about the stability of these waves.  While many authors have addressed the spectral or linear stability of interfacial waves, nonlinear results are mostly limited to dispersive model equations such as Kortweg--de Vries (KdV).   In this paper, we prove a number of theorems on the (conditional) orbital stability of small-amplitude traveling wave solutions to the full system when the surface tension is strong in a sense to be quantified shortly. 

\begin{figure} \label{internal wave system figure}
\includegraphics{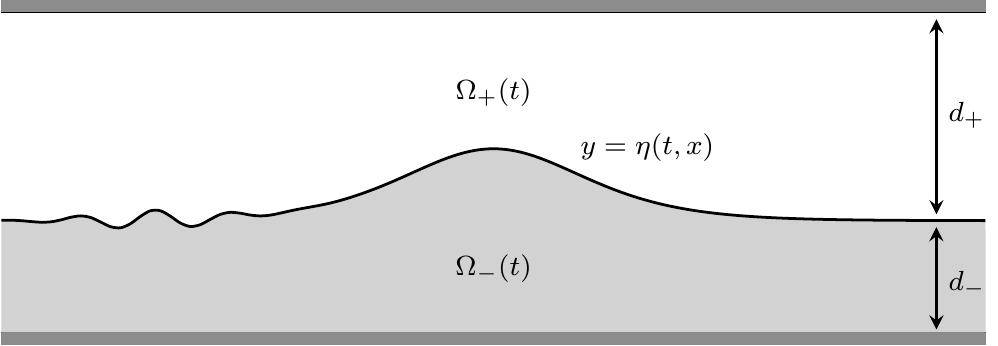}
\caption{Configuration of the internal wave system.  The unshaded fluid region $\Omega_+(t)$ has density $\rho_+$ while the darker shaded region $\Omega_-(t)$ below is of density $\rho_- \geq \rho_+$.  Their interface $\mathscr{S}(t)$ is a free boundary given by the graph of $\eta = \eta(t,x)$.  In the far field, the widths of the upper and lower layer limit to $d_+$ and $d_-$, respectively.}
\end{figure}

Mathematically, the problem is formulated as follows.  Fix Cartesian coordinates $(x,y)$ so that the wave propagates in the $x$-direction with gravity acting in the negative $y$-direction.  Because we are most interested in the motion of the boundary, we suppose that the fluid domain is confined to a channel with rigid walls at heights $y = \pm d_\pm$ for fixed $d_\pm \in (0, \infty)$.  At each time $t \geq 0$, the interface $\mathscr{S} = \mathscr{S}(t)$ is taken to be the graph of an unknown smooth function $\eta = \eta(t,x)$. For small-amplitude waves, this choice incurs no loss of generality.  Then, the upper layer inhabits the (time-dependent) set
\[ 
\Omega_+ = \Omega_+(t) := \left\{ (x,y) \in \mathbb{R}^2 : \eta(t,x) < y < d_+ \right\},
\]
while the lower layer is given by
\[ 
\Omega_- = \Omega_-(t) := \left\{ (x,y) \in \mathbb{R}^2 : -d_- < y <  \eta(t,x) \right\}.
\]
We write $\Omega(t) := \Omega_+(t) \cup \Omega_-(t)$ to denote the fluid domain.  Our focus will be on spatially localized waves for which $\eta(t,\placeholder)$ decays at infinity.  See Figure~\ref{internal wave system figure} for an illustration.

Assuming that the flow in each region is irrotational and incompressible, the velocity field in $\Omega_\pm(t)$ is then given by $\nabla\Phi_\pm$, for some function $\Phi_\pm = \Phi_\pm(t,x)$ called the velocity potential.  We take the density in $\Omega_\pm(t)$ to be constant and denote it by $\rho_\pm > 0$.  In order to ensure that heavier fluid elements do not lie above lighter elements, it is required that $\rho_+ \leq \rho_-$.  The case $\rho_+ = 0$ formally corresponds to a single fluid beneath vacuum.  All of our analysis extends to this regime with only superficial modifications to the arguments.  

The evolution of the system is governed by the incompressible irrotational Euler equations with a free boundary.  In the bulk, the conservation of momentum has the simple expression
\bse \label{Eulerian water wave problem}
\be 
\Delta \Phi_\pm = 0 \qquad \textrm{in } \Omega_\pm(t).  \label{Eulerian Laplace problem} 
\ee
On both the rigid and moving boundary components, we have the kinematic condition
\be 
\left\{
\begin{aligned} 
\partial_t \eta & = -\eta^\prime \partial_x \Phi_- + \partial_y \Phi_- = -\eta^\prime \partial_x\Phi_+ + \partial_y \Phi_+ & \qquad &\textrm{on } \{ y = \eta(t,x) \} \\
\partial_y \Phi_\pm & = 0 & \qquad & \textrm{on } \{ y = \pm d_\pm \},
\end{aligned}
\right.
 \label{Eulerian kinematic} \ee
while on $\mathscr{S}(t)$ the dynamic or Bernoulli condition is imposed:
\be
 \jump{ \rho \partial_t \Phi  +  \frac{1}{2} \rho |\nabla \Phi|^2 + g \rho \eta} + \sigma \left( \frac{\eta^\prime}{\sqrt{1+(\eta^\prime)^2}} \right)^\prime = 0  \qquad \textrm{on } \{ y = \eta(t,x) \}.  
\label{Eulerian dynamic} \ee 
\ese
Here $\jump{\placeholder} := (\placeholder)_+ - (\placeholder)_-$ denotes the jump of a quantity over the interface, $g > 0$ is the gravitational constant, and $\sigma > 0$ is the coefficient of surface tension.  The last term on the right-hand side above is the signed curvature of the interface and represents the influence of capillary effects.  In \eqref{Eulerian kinematic}, we are enforcing the continuity of the normal velocity across the interface, while \eqref{Eulerian dynamic} arises from the Young--Laplace law for the pressure jump.   Also, here and in what follows we will mostly adhere to the convention that primes denote $x$-derivatives of functions depending on $(t,x)$, while $\partial_x$ is reserved for functions of $(t,x,y)$ or in defining operators.  

Rather than work with the full velocity potential $\Phi_\pm$, which is defined on a moving domain, it is advantageous to consider its restriction to the free boundary:
\[ \varphi_\pm = \varphi_\pm(t,x) := \Phi_\pm(t,x,\eta(t,x)).\]
Through the use of nonlocal operators, it is possible to reformulate \eqref{Eulerian water wave problem} in terms of the surface variables $(\eta, \varphi_+,\varphi_-)$; see Section~\ref{nonlocal operators section}.

\subsection{Informal statement of results}

 \emph{Traveling} or \emph{steady} solutions of \eqref{Eulerian water wave problem} are waves of permanent configuration that appear independent of time when viewed in a moving reference frame.  Specifically, they exhibit the ansatz
\[ \eta(t,x) = \eta^c(x-ct), \qquad \varphi_\pm(t,x) = \varphi_\pm^c(x-ct),\]
for some traveling wave profile $(\eta^c, \varphi_+^c,\varphi_-^c)$ and wave speed $c \in \mathbb{R}$.
\begin{figure} 
\begin{center}
\includegraphics[scale=1.2]{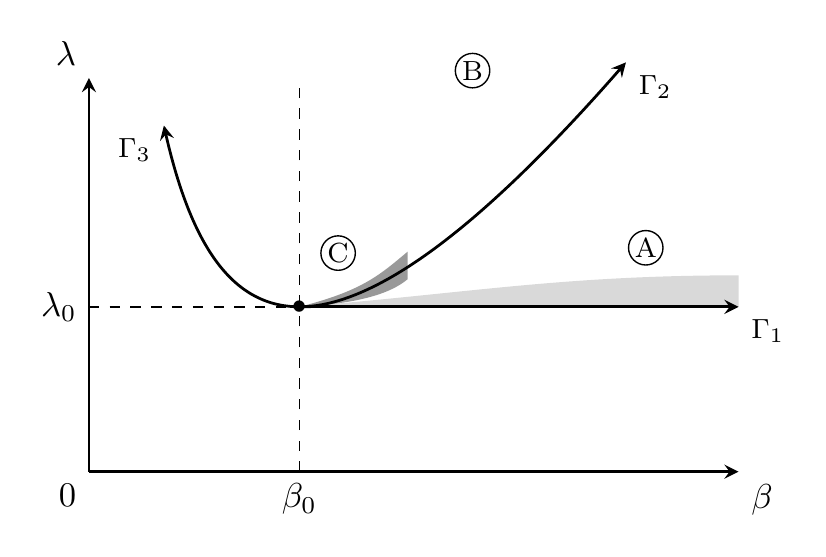}
\end{center}
\caption{Bifurcation diagram for internal capillary-gravity waves.  Region A is the lighter shaded area that lies above $\Gamma_1$ and to the right of $\Gamma_2$; this is where one has monotone solitary waves.  Region B consists of all $(\beta,\lambda)$ lying above $\Gamma_1$ and to the right of $\Gamma_3$. Finally, Region C is the darker shaded set neighboring $\Gamma_2$. Explicit parameterizations for these curves can be found in \eqref{parameterization Gamma} and \eqref{Gamma 3 parameterization}}
\label{dispersion figure}
\end{figure}
In the gravity wave case $\sigma = 0$, it is known that there exist \emph{solitary waves} \cite{bona1983finite,amick1986global,mielke1995homoclinic,james2001internal}, for which $\eta^c$ decays as $|x| \to \infty$; \emph{periodic waves} \cite{amick1986global,amick1989small}, for which $\eta^c$ is periodic in $x$; and \emph{fronts} \cite{amick1989small,makarenko1992bore,mielke1995homoclinic,chen2019center,chen2020global}, for which $\eta^c$ has distinct limits upstream and downstream.  Without surface tension, however, the dynamical problem is ill-posed \cite{lannes2013stability}, so to study stability we always take $\sigma > 0$.    Rigorous existence results for small-amplitude periodic waves (including those with vorticity) were obtained in this regime by Le \cite{le2018transmission}.  Solitary internal capillary-gravity waves were constructed by Kirrmann \cite{kirrmann1991reduction} and Nilsson \cite{nilsson2017internal}; the stability of these solutions is the main subject of the present paper.  We also note that analytical and numerical investigations of this regime have been performed by Laget and Dias \cite{laget1997interfacial}.

The existence and qualitative properties of traveling internal waves are determined by four dimensionless parameters.  The primary two are the Bond number $\beta$ and inverse square Froude number $\lambda$ given by  
\be\label{dimensionless parameters}
\beta := \frac{\sigma}{d_+ \rho_- c^2}, \qquad \lambda := -\frac{g \jump{\rho} d_+}{\rho_- c^2}.
\ee
The Bond number measures the strength of the surface tension, while $\lambda$ describes the balance between kinetic and potential energy.  One can think of the Froude number $1/\sqrt{\lambda}$ as a non-dimensionalized wave speed, hence large $\lambda$ corresponds roughly to slow moving waves.  

The dispersion relation for internal capillary-gravity waves (rescaled to dimensionless variables) is given by
\be \label{dispersion relation}
 \sum_{\pm}  \frac{\rho_\pm}{\rho_-} \xi \coth{\left(\frac{d_\pm}{d_+} \xi\right)} = \lambda + \beta \xi^2. 
\ee
This results from linearizing the problem at the trivial solution $(\eta, \varphi_+, \varphi_-) = (0,0,0)$, then looking for eigenvalues of the form $i \xi$.  If $\xi$ is a root to \eqref{dispersion relation}, the linearized problem admits a plane wave solution with $\eta = \exp{(i\xi(x-ct))}$.  After some algebra, it can be shown that there are three bifurcation curves $\Gamma_1$, $\Gamma_2$, $\Gamma_3$ that organize the $(\beta,\lambda)$-plane into regions where the configuration of the spectrum near the imaginary axis is qualitatively the same; see Figure~\ref{dispersion figure}.  They meet at the point $(\beta_0,\lambda_0)$, which is given by
\be\label{critical parameters}
\beta_0 := \frac13 \left( \frac{\rho_+}{\rho_-} + \frac{d_-}{d_+} \right), \qquad \lambda_0 := \frac{\rho_+}{\rho_-} + \frac{d_+}{d_-},
\ee
and there we find that $\xi = 0$ is a root of \eqref{dispersion relation} with multiplicity $4$.  We say that $\beta_0$ is the critical Bond number separating the weak and strong surface tension regimes.

In this regard, the internal wave system is quite similar to that of water waves beneath vacuum; see, for instance, \cite{amick1989solitary,iooss1990bifurcation,dias1993capillary,iooss1993perturbed,buffoni1996bifurcation,buffoni1996plethora,groves2015existence}.  However, there are two additional parameters to consider: the ratios of the fluid densities $\varrho$ and far-field layer heights $h$, defined by
\be \label{definition ratios}
\varrho := \frac{\rho_+}{\rho_-}, \qquad h := \frac{d_-}{d_+}.
\ee
These are specific to the two-fluid problem and allow for a surprisingly rich variety of traveling waves.  For example, it has been proved by Nilsson \cite{nilsson2017internal} and Kirrmann \cite{kirrmann1991reduction} that for $(\beta,\lambda)$ in the Region A illustrated in Figure~\ref{dispersion figure}, there exists six qualitatively distinct types of small-amplitude waves.  When $\varrho - 1/h^2$ is negative and $O(1)$ as $\lambda \searrow \lambda_0$, they find waves of depression (that is, $\eta < 0$) that are to leading order KdV solitons.  These are the only kind of wave possible in the corresponding parameter regime for the one-fluid case, which is consistent with simply taking $\rho_+ = 0$.  On the other hand, when $\varrho - 1/h^2 > 0$, there are internal waves of elevation ($\eta > 0$) whose interface is a perturbed KdV soliton.  Moreover, in the regime $| \varrho - 1/h^2| \eqsim |\lambda-\lambda_0|^{1/2} \ll 1$, they construct traveling waves that are Gardner solitons to leading order.  This furnishes four types of solutions, with waves of depression and elevation for both signs of $\varrho - 1/h^2$.  A fuller account is given in Section~\ref{traveling wave section}.

Our first theorem, stated informally for the time being, establishes the nonlinear stability of all these waves in the orbital sense.

\begin{theorem}[Strong surface tension] \label{large beta theorem} 
Every sufficiently small-amplitude solitary internal wave $(\eta^c, \varphi_+^c,\varphi_-^c)$ with $(\beta,\lambda)$ in Region~A and $0 < \lambda - \lambda_0 \ll 1$ is conditionally orbitally stable in the following sense.  For all $R > 0$ and $r > 0$, there exists $r_0 > 0$ such that, if $(\eta, \varphi_+, \varphi_-)$ is any solution defined on a time interval $[0,t_0)$ that obeys the bound
\be \label{a priori bound}
\sup_{t \in [0,t_0)} \left( \| \eta(t) \|_{H^{3+}} + \| \varphi_+(t) \|_{\dot H^{\frac{5}{2}+} \cap \dot H^{\frac{1}{2}}} + \| \varphi_-(t) \|_{\dot H^{\frac{5}{2}+} \cap \dot H^{\frac{1}{2}}} \right) < R,
\ee
and for which the initial data satisfies 
\be \label{initial data bound}
\| \eta(0) - \eta^c \|_{H^1} + \| \varphi_+(0) - \varphi_+^c \|_{\dot H^{\frac{1}{2}}} +  \| \varphi_-(0) - \varphi_-^c \|_{\dot H^{\frac{1}{2}}} < r_0,
\ee
then 
\be \label{conditional stability}
\sup_{t \in [0,t_0)} \inf_{s \in \mathbb{R}} \left( \| \eta(t, \placeholder - s) - \eta^c \|_{H^1} + \| \varphi_+(t, \placeholder - s) - \varphi_+^c \|_{\dot H^{\frac{1}{2}}}   +  \| \varphi_-(t, \placeholder -s) - \varphi_-^c \|_{\dot H^{\frac{1}{2}}} \right) < r.
\ee
\end{theorem}

\begin{remark} \label{strong surface tension remark}
The bound in \eqref{conditional stability} controls the distance between $(\eta,\varphi_+,\varphi_-)$ and the family of translates of the steady wave.  This is natural given that the underlying system \eqref{Eulerian water wave problem} is translation invariant, and indeed it is necessary even for model equations such as KdV.  Local well-posedness for the Cauchy problem at the level of regularity represented by the norm in \eqref{a priori bound} has been proved by Shatah and Zeng \cite{shatah2008interface,shatah2011interface}.  On the other hand, we will show in Section~\ref{Hamiltonian structure section} that the lower regularity norm in \eqref{initial data bound} and \eqref{conditional stability} is equivalent to the physical energy.  We also emphasize that because $r$ is independent of $t_0$, this result is much stronger than continuity of the data-to-solution map.  For a global-in-time solution, it gives orbital stability in the classical sense.
\end{remark}

Our next result concerns uniform flows for which the interface is perfectly flat and the velocity is purely horizontal with the same constant value $c$ in both layers.  In a reference frame moving with the wave, it therefore appears quiescent.  
While linear stability criteria for this regime are classical (see, for example, \cite{drazin2004book}), as far as we are aware, this is the first nonlinear stability result.  

\begin{theorem}[Uniform flow] \label{parallel flow theorem} 
The laminar solution $(\eta^c, \varphi_+^c, \varphi_-^c) = (0,0,0)$ is conditionally stable in the sense of Theorem~\ref{large beta theorem} provided that $(\beta,\lambda)$ lies in Region~B.
\end{theorem}

Lastly, we consider the critical surface tension case where $(\beta,\lambda)$ lies in Region~C near $(\beta_0,\lambda_0)$; see Figure~\ref{dispersion figure}.   It is well-established that in this regime, the dynamics of sufficiently shallow waves are captured by a fifth-order nonlinear dispersive PDE similar to the Kawahara equation \cite{kawahara1972oscillatory,benney1976general}.   For spatially localized traveling waves, one can then integrate to obtain a fourth-order ODE
\be \label{intro kawahara}
Z^{\prime\prime\prime\prime} - 2(1+\delta) Z^{\prime\prime} + Z - Z^2 = 0,
\ee
where we have scaled out all but the non-dimensional parameter $\delta = \delta_c$, which is determined explicitly by the wave speed via \eqref{definition epsilon C region}.   The ODE \eqref{intro kawahara} boasts an extraordinarily large variety of solutions that are homoclinic to $0$ (see, for example, \cite{buffoni1996plethora}). For this paper, we focus on the family $\{ Z_\delta \}$ of ``primary homoclinic'' orbits that are even, unimodal, and exponentially localized.  They have been rigorously constructed for $\delta \geq 0$ and  $-1 \ll \delta < 0$,  and numerically observed to persist as $\delta \searrow -2$.  Nilsson \cite{nilsson2017internal} shows that for every $|\delta_c| \ll 1$, there exists a traveling wave $(\eta^c, \varphi_+^c, \varphi_-^c)$ solution to \eqref{Eulerian water wave problem} with $\eta^c$ given to leading order by a rescaling of $Z_{\delta_c}$.  The next result states that the orbital stability or instability of these solutions to the full internal wave problem can be determined by  considerations of the far simpler model equation \eqref{intro kawahara}.

\begin{theorem}[Critical surface tension] \label{critical beta theorem}
Let $\{ Z_\delta \}$ be the family of primary homoclinic solutions to \eqref{intro kawahara} and suppose that $(\beta,\lambda)$ lies in Region~C with $|\delta_c| \ll 1$.  Then the corresponding traveling wave solution $(\eta^{c_*}, \varphi_+^{c_*}, \varphi_-^{c_*})$ to \eqref{Eulerian water wave problem} is conditionally orbitally stable provided that the function
\be \label{intro d prime}
c \mapsto  \signum{c} \int_\R  Z_{\delta_c}^2 \,\diffx
\ee
is strictly increasing at $c_*$, and it is orbitally unstable if this function is strictly decreasing there.
\end{theorem}

We remark that this theorem is new even for the one-fluid case.  Physically, the integral in \eqref{intro d prime} represents the momentum carried by the wave; whether it is increasing or decreasing as a function of $\delta$ has been investigated by many authors but remains open in the present case.  Under conditions analogous to Theorem~\ref{critical beta theorem}, Levandosky \cite{levandosky1999stability,levandosky2007stability} proves a nonlinear stability/instability result for ground state solutions to a family of fifth-order dispersive PDEs that includes the Kawahara equation.   On the other hand, for $\delta = 1/6$, the primary homoclinic solution to \eqref{intro kawahara} has the explicit formula 
\[
Z_{\tfrac{1}{6}} = \frac{35}{24} \sech^4{\left(\frac{\sqrt{6}}{12} \placeholder \right)},
\]
and by exploiting this, \eqref{intro d prime} can be evaluated directly for various choices of the dimensional parameters \cite{albert1992positivity,dey1996stationary,molinet2018}.  Numerical evidence in \cite{Ilichev1992stability} suggests that stability holds for the Kawahara equation with $\delta > 0$, but analytical results are not currently available.  Through Theorem~\ref{critical beta theorem}, progress on this question for the model equation can immediately be translated to \eqref{Eulerian water wave problem}.

\subsection{Idea of the proof}

It is well known that the internal wave problem \eqref{Eulerian water wave problem} can be formulated as an abstract Hamiltonian system of the general form 
\[ \partial_t u = J \Diff \eng(u),\]
where $u = u(t)$ is an unknown related to $(\eta, \varphi_+,\varphi_-)$, the Poisson map $J$ is a skew-adjoint operator, and $\eng$ is a conserved energy functional.  The translation invariance of the system gives rise to a second conserved quantity, the momentum $\mom$.  A traveling wave solution with wave speed $c$ is in fact a critical points of the augmented Hamiltonian $\augHam := \eng - c \mom$.  

It is therefore natural to adopt a constrained variational viewpoint, attempting to show that the waves are minimizers of the energy on level sets of the momentum.  A serious challenge that arises in many applications, including the present one, is that $\Diff^2 \augHam$ has an unstable direction as well as a $0$ eigenvalue due to translation invariance.  This situation can lead to either stability or instability, and a deft use of the conserved quantities is necessary to discern which occurs for the waves in question.  Benjamin \cite{benjamin1972stability} pioneered this approach in his study of the orbital stability of KdV solitons.  A systematic and greatly expanded version was later developed by Grillakis, Shatah, and Strauss \cite{grillakis1987stability1}.  Now called the GSS method, it is one of the primary tools in nonlinear stability theory for Hamiltonian systems.

Historically, though, GSS has not been especially successful in treating the full water wave problem. Indeed, \eqref{Eulerian water wave problem} exhibits a host of features that make it highly resistant to na\"ive applications of systematic methods. For example, the theory in \cite{grillakis1987stability1} requires that $J$ be an isomorphism, which does not hold here as we show in Section~\ref{Hamiltonian structure section}.   It is also formulated under the hypothesis that the Cauchy problem is globally well-posed in the natural energy space.  At present, \eqref{Eulerian water wave problem} is only known to be locally well-posed and this assumes considerably more smoothness.  Because the water wave problem is quasilinear, it is not expected to generate a flow on the energy space.  Worse still, the corresponding functional $\eng$ is not even differentiable at this level or regularity.

Seeking to address these issues, Varholm, Wahl\'en, and Walsh \cite{varholm2020stability} obtained a variant of the GSS method that weakens the above hypotheses.  In place of the bijectivity of $J$, it essentially requires only that $J$ is injective with dense range.  The functional analytic framework is also designed to accommodate the gap in regularity between the energy space and the smoothness needed for local well-posedness.  In this paper, we use the relaxed GSS method to attack the water wave problem directly and prove Theorem~\ref{large beta theorem} and Theorem~\ref{critical beta theorem}.  A simpler, self-contained argument suffices for Theorem~\ref{parallel flow theorem} as the augmented linearized Hamiltonian has no unstable directions in that case.  

The most challenging step in this procedure is computing the spectrum of the linearized augmented Hamiltonian at a traveling wave.  For this, we generalize a technique introduced by Mielke \cite{mielke2002energetic} in his work on solitary capillary-gravity waves in a single finite-depth fluid and with strong surface tension.  Briefly, this involves using the kinematic condition to eliminate $\varphi_\pm$ and obtain an auxiliary functional acting only on $\eta$.  Conjugating by a rescaling operator, a delicate argument shows that for sufficiently small-amplitude waves, the spectrum coincides to leading order with the linearization of a dispersive model equation (steady KdV or Gardner in the setting of Theorem~\ref{large beta theorem} and steady Kawahara for Theorem~\ref{critical beta theorem}).  Here it is important to note that these calculations are substantially more difficult in the internal wave setting than for a single fluid: the nonlocal operators introduced in the Hamiltonian reformulation are more complicated, and they must be expanded to higher order.  On the other hand, Mielke proves conditional orbital stability using an ad hoc modification of the GSS method.  Because we have at our disposal the general theory from \cite{varholm2020stability}, we are able to streamline this part of the argument.

Let us also mention an alternative variational approach to proving nonlinear stability of water waves due to Buffoni.  Roughly speaking, this consists of a penalization scheme followed by a concentration compactness argument to directly construct traveling waves as constrained minimizers of the energy with fixed momentum.  In some circumstances, one can then apply a soft analysis argument of Cazenave and Lions \cite{cazenave1982orbital} to infer so-called (conditional) \emph{energetic stability}, meaning that the \emph{set} of constrained minimizers is stable in the energy norm.  This differs from the orbital stability we obtain unless one also has uniqueness of the minimizer up to translation, which is typically not available.  Through this variational method, Buffoni proved the existence and stability (in the above sense) of solitary waves in the single-fluid case with strong surface tension \cite{buffoni2004stability}.  He also gave partial results concerning waves with weak surface tension and in infinite depth \cite{buffoni2005conditional,buffoni2009gravity}.  Pushing significantly further the technique, Groves and Wahl\'en \cite{groves2010existence,groves2011conditional} subsequently obtained complete versions of these theorems, and also treated the case of constant vorticity \cite{groves2015existence}.

\subsection{Plan of the article}

In Section~\ref{formulation section}, we begin by reformulating the internal wave problem \eqref{Eulerian water wave problem} as an abstract Hamiltonian system in the style of Benjamin and Bridges \cite{benjamin1997reappraisal1}.  A  number of hypotheses necessary to apply the general theory in \cite{varholm2020stability} are then be verified.  We also recall the existence theory due to Nilsson \cite{nilsson2017internal}, recasting it within the Hamiltonian framework of the present paper.  

Section~\ref{spectrum section} is devoted to computing the spectrum of the linearized augmented Hamiltonian at a uniform flow or small-amplitude traveling wave.  As mentioned above, our calculation is patterned on the basic approach of Mielke \cite{mielke2002energetic}, but with many additional challenges owing to the more complicated physical setting.  

The main results are then proved in Section~\ref{proof section}.  Thanks to the general theory, this requires us only to determine whether the so-called moment of instability, a scalar-valued function of the wave speed, is strictly convex or concave.  This is accomplished by exploiting a long-wave rescaling and the leading-order form of the waves known from the existence theory.

Finally, Appendix~\ref{identities appendix} contains some elementary calculations that plan an essential part in the spectral computation.
 
\section{Hamiltonian formulation for internal waves} \label{formulation section}

\subsection{Nonlocal operators and surface variables} \label{nonlocal operators section} 

Following the classical Zakharov--Craig--Sulem idea, we will reformulate the interface Euler equations \eqref{Eulerian water wave problem} as a nonlocal problem in terms of quantities restricted to the free boundary $\mathscr{S}(t)$.  A similar approach was taken by Benjamin and Bridges \cite{benjamin1997reappraisal1} and Craig and Groves \cite{craig2000normal} in their treatments of this system.

Recall that we have defined 
\[ \varphi_\pm(t,x) := \Phi_\pm(t, x, \eta(t,x)),\]
to be the traces of the velocity potentials for the upper and lower regions.  The velocity field can then be recovered by means of the Dirichlet--Neumann operator in $\Omega_\pm(t)$. For a fixed $\eta$, this is the mapping given by
\be \begin{split}
 \DN_\pm(\eta) f_\pm & := \jbracket{\eta^\prime} \big( N_\pm \cdot \nabla {\HE}_\pm(\eta) f \big)|_{\mathscr{S}} 
\end{split}\label{definition DN} \ee
where $N_\pm$ is the unit outward normal to $\Omega_\pm$ along $\mathscr{S}$, we are making use of the Japanese bracket notation $\jbracket{\placeholder} := \sqrt{1+|\placeholder|^2}$, and ${\HE}_\pm(\eta) f$ is the harmonic extension of $f$ to $\Omega_\pm$.  Specifically, in view of the kinematic conditions \eqref{Eulerian kinematic} on the rigid boundaries, we take $\HE_\pm(\eta) f$ to be the unique solution to
\be \label{definition HE}
 \left\{ \begin{aligned} 
\Delta {\HE}_\pm(\eta) f & = 0 \qquad \textrm{in } \Omega_\pm \\ 
{\HE}_\pm(\eta) f & = f \qquad \textrm{on } \{y = \eta\} \\
\partial_y {\HE}_\pm(\eta) f & = 0 \qquad \textrm{on } \{y = \pm d_\pm\}. \end{aligned} \right.
\ee

Dirichlet--Neumann operators are a standard tool in the study of water waves; for a general reference, see \cite{lannes2013book} or \cite{shatah2008geometry}.  In particular, for any real numbers $k_0 > 1/2$ and $k \in [1/2-k_0,1/2+k_0]$, and profile $\eta \in H^{k_0+1/2}(\mathbb{R})$ with $-d_- < \eta < d_+$, we have that $\DN_\pm(\eta)$ is an isomorphism $\dot H^k(\mathbb{R}) \to \dot H^{k-1}(\mathbb{R})$, where $\dot H^k$ denotes the usual homogeneous Sobolev space of order $k$.  Similarly, $\HE_\pm(\eta)$ is bounded as a mapping $H^k(\mathbb{R}) \to H^{k+1/2}(\Omega_\pm)$ and $\dot H^k(\mathbb{R}) \to \dot H^{k+1/2}(\Omega_\pm)$.  Our analysis relies on the fact that the Dirichlet--Neumann operator depends smoothly on $\eta$.  Indeed, $\eta \mapsto \DN_\pm(\eta)$ is real analytic and at $\eta = 0$, it is the Fourier multiplier $\DN_\pm(0) = |\partial_x|\tanh{(d_\pm |\partial_x|)}$. Note also that $\DN_\pm(\eta)$ is self-adjoint $\dot H^{1/2}(\mathbb{R}) \to \dot H^{-1/2}(\mathbb{R})$ and  positive definite.  

Because $N_+ + N_- = 0$, the continuity of the normal velocity over the interface is equivalent to 
\be 
\DN_+(\eta) \varphi_+ + \DN_-(\eta)\varphi_- = 0.  
\label{continuity normal velocity} \ee
Thus the kinematic condition \eqref{Eulerian kinematic} on $\mathscr{S}(t)$ can be expressed as 
\be 
\partial_t \eta = \mp G_\pm(\eta) \varphi_\pm.
\label{kinematic phipm} \ee
Note that the kinematic condition on $\{ y =\pm d_\pm\}$ is encoded in the definition of $\HE$. 

Rather than work with $\varphi_\pm$, we consider the quantity
\be\label{def psi} 
\psi := -\jump{\rho \Phi} = \rho_- \varphi_- - \rho_+ \varphi_+.
\ee
Using \eqref{continuity normal velocity}, we can recover both $\varphi_+$ and $\varphi_-$ from $\psi$.  Indeed, we compute that 
\begin{align*} 
-\DN_-(\eta) \psi &= \rho_+ \DN_-(\eta) \varphi_+ - \rho_- \DN_-(\eta) \varphi_- \\
& = \rho_+ \DN_-(\eta) \varphi_+ + \rho_- \DN_+(\eta) \varphi_+ = B(\eta) \varphi_+,
\end{align*}
where 
\be 
B(\eta) := \rho_+ \DN_-(\eta) + \rho_- \DN_+(\eta). 
\label{definition B}\ee
By the above discussion, we have that $B(\eta)$ is bounded and linear $H^k(\mathbb{R}) \to H^{k-1}(\mathbb{R})$ and $\dot H^k(\mathbb{R}) \to \dot H^{k-1}(\mathbb{R})$, for all $\eta \in H^{k_0+1/2}(\mathbb{R})$ and with $k_0$, $k$ given as before.  One can readily confirm, moreover, that $B(\eta)$ is an isomorphism $\dot H^k(\mathbb{R}) \to \dot H^{k-1}(\mathbb{R})$.   Thus, repeating the same computation with signs reversed leads to the identity
\be 
\varphi_\pm = \mp B(\eta)^{-1} \DN_\mp(\eta) \psi.  
\label{psi phipm identity} \ee
The kinematic condition \eqref{kinematic phipm} can then be recast as
\be \label{Hamiltonian kinematic}
 \partial_t \eta = A(\eta) \psi
 \ee
for the operator  
\be 
A(\eta) := \DN_-(\eta) B(\eta)^{-1} \DN_+(\eta).
\label{definition A} \ee
It is simple to show that these operators commute, and hence we can alternatively write 
\[ 
A(\eta) = \DN_+(\eta) B(\eta)^{-1} \DN_-(\eta).
\]

To reformulate the Bernoulli condition \eqref{Eulerian dynamic} requires being able to reconstruct the full gradient $\nabla \Phi_\pm$ restricted to the interface from the surface variables.  For this, we simply observe that
\[ 
\varphi_\pm^\prime = (\partial_x \Phi_\pm)|_{y = \eta}  + \eta^\prime (\partial_y \Phi_\pm)|_{y = \eta},
\]
which together with the definition of $\DN_\pm(\eta)$ in \eqref{definition DN} leads to the useful identities
\be \label{grad Phi identities} 
\begin{pmatrix}  \varphi_\pm^\prime \\ \DN_\pm(\eta) \varphi_\pm \end{pmatrix} = 
\begin{pmatrix} 1 &  \eta^\prime \\ 
\pm  \eta^\prime & \mp 1 \end{pmatrix} (\nabla \Phi_\pm)|_{\mathscr{S}}, 
\quad 
(\nabla \Phi_\pm)|_{\mathscr{S}} 
= \frac{1}{1+(\eta^\prime)^2}\begin{pmatrix}  1 & \pm  \eta^\prime \\
 \eta^\prime & \mp 1\end{pmatrix} \begin{pmatrix}  \varphi_\pm^\prime \\ \DN_\pm(\eta) \varphi_\pm \end{pmatrix}.  
\ee
 
 Now, observe that simply by definition
\[ 
-\partial_t \psi = \rho_+ \partial_t \varphi_+ - \rho_- \partial_t \varphi_- = \jump{\rho \partial_t \Phi} + (\partial_t \eta) \jump{\rho \partial_y \Phi}.
\]
Thus \eqref{Eulerian dynamic} can be rewritten as 
\be \label{rewritten Bernoulli}
 \partial_t \psi = \frac{1}{2} \jump{\rho |\nabla \Phi|^2} - (\partial_t \eta) \jump{\rho \partial_y \Phi} + g\jump{\rho} \eta + \sigma \left( \frac{\eta^\prime}{\jbracket{\eta^\prime}} \right)^\prime.
 \ee
In view of \eqref{grad Phi identities}, this gives a formulation of the Bernoulli condition involving only the surface variables $\eta$ and $\psi$.

\subsection{Functional analytic setting} 

Let us now define the function spaces in which the internal wave problem will be posed.  Following the approach outlined above, we wish to recast the system in terms of the unknown $u := (\eta, \psi)$.  It is convenient to introduce a scale of spaces describing the spatial regularity of $u$:  for each $k \geq 1/2$, let
\be \label{definition Xkspace}
\Xspace^k = \Xspace_1^k \times \Xspace_2^k := H^{k+\frac{1}{2}}(\mathbb{R}) \times \left( \dot H^k(\mathbb{R}) \cap \dot H^{\frac{1}{2}}(\mathbb{R}) \right).
\ee 
In what follows, we will frequently use the shorthand $\Xspace^{k+}$ (and likewise $H^{k+}$) to denote $\Xspace^{k+\varepsilon}$ for any $0 < \varepsilon \ll 1$ that is fixed and then suppressed.  

\begin{remark} \label{density remark}
 Observe that $H^r(\mathbb{R}) \cap \dot H^s(\mathbb{R})$ is dense in both $H^r(\mathbb{R})$ and $\dot H^s(\mathbb{R})$  for all $r, s \in \mathbb{R}$; see, for example, \cite[Lemma A.1]{varholm2020stability}.
\end{remark}

We will work in a trio of nested Banach spaces $\Wspace \hookrightarrow \Vspace \hookrightarrow \Xspace$.  The largest, $\Xspace$, we call the \emph{energy space}.  Specifically, we take
\be 
\Xspace := \Xspace^{\frac{1}{2}} = H^1(\mathbb{R}) \times \dot H^{\frac{1}{2}}(\mathbb{R}).
\label{definition Xspace} \ee
Its dual is 
\[ 
\Xspace^* = H^{-1}(\mathbb{R}) \times \dot{H}^{-\frac{1}{2}}(\mathbb{R}),
\]
and we let $I  = ( 1-\partial_x^2, |\partial_x|)$  denote the natural isomorphism $\Xspace \to \Xspace^*$.  In particular, when $u \in \Xspace$, the velocity field $\nabla \Phi_\pm \in L^2(\Omega_\pm)$.  As we will see below, this ensures that the kinetic energy is indeed finite.  Likewise, the $H^1$ norm of $\eta$ is equivalent to the excess potential energy relative to the undisturbed state.
 
However, observe that $u \mapsto \DN_\pm(\eta)$ is not smooth with domain $\Xspace$, since we must have that $\eta$ is at least Lipschitz continuous and also bounded away from the rigid boundaries at $y =\pm d_\pm$.  This leads us to introduce the space 
\be \label{definition Vspace}
\Vspace :=  \Xspace^{1+} = H^{\frac{3}{2}+}(\mathbb{R}) \times \left( \dot{H}^{1+}(\mathbb{R}) \cap  \dot{H}^{\frac{1}{2}}(\mathbb{R}) \right),
\ee
and neighborhood
\[ 
\mathcal{O} := \left\{ (\eta, \psi) \in \Xspace : -d_- < \eta < d_+ \right\}.
\]
Note that $H^{3/2+}(\mathbb{R}) \hookrightarrow W^{1,\infty}(\mathbb{R})$, so $u \in \Vspace$ does indeed imply that $\eta$ has the requisite Lipschitz continuity.

Lastly, because the Cauchy problem is not likely to be well-posed in $\Vspace$, we consider the even smoother space 
\be 
\Wspace := \Xspace^{\frac{5}{2}+} = H^{3+}(\mathbb{R}) \times \left( \dot{H}^{\frac{5}{2}+}(\mathbb{R}) \cap  \dot{H}^{\frac{1}{2}}(\mathbb{R}) \right).
\label{definition Wspace} \ee
Local well-posedness at this level of regularity was proved by Shatah and Zeng \cite{shatah2011interface}, for example. 

Before continuing, we record the fact these spaces have the following embedding property that corresponds to \cite[Assumption 1]{varholm2020stability}.

\begin{lemma}[Spaces] \label{spaces lemma} 
Let the spaces $\Xspace$, $\Vspace$, and $\Wspace$ be defined by \eqref{definition Xspace}, \eqref{definition Vspace}, and \eqref{definition Wspace}, respectively.  There exists a constant $C > 0$ and $\theta \in (0,\tfrac{1}{4})$ such that 
\[ 
\| u \|_{\Vspace}^3 \leq C \| u \|_{\Xspace}^{2+\theta} \| u \|_{\Wspace}^{1-\theta} \qquad \textrm{for all } u \in \Wspace.
\]
\end{lemma}
\begin{proof}
This can be quickly verified using from the definitions of $\Xspace$, $\Vspace$, and $\Wspace$  and the Gagliardo--Nirenberg interpolation inequality. 
\end{proof}

Observe that this inequality ensures that small cubic terms in $\Vspace$ are dominated by quadratic terms in $\Xspace$ on bounded sets in $\Wspace$, which is needed in the general theory when Taylor expanding functionals that are smooth with domain $\Vspace \cap \mathcal{O}$.  A similar argument appears in the proof of Theorem~\ref{precise parallel theorem}.

\subsection{Hamiltonian structure} \label{Hamiltonian structure section}
Benjamin and Bridges \cite{benjamin1997reappraisal1} established that the internal wave problem \eqref{Eulerian water wave problem} has a (canonical) Hamiltonian formulation in terms of the state variable $u$ by adapting the well-known Zakharov--Craig--Sulem formulation for the single-fluid case.   In this section, we will recall the system obtained in \cite{benjamin1997reappraisal1} while verifying that it satisfies a number of the hypotheses of the general theory.

 The kinetic energy carried by the wave is given by
\begin{align*} 
\kinE &= \frac{1}{2} \int_{\Omega_+(t)} \rho_+ |\nabla \Phi_+|^2 \,\diffx \,\diffy + \frac{1}{2} \int_{\Omega_-(t)} \rho_- |\nabla \Phi_-|^2 \,\diffx \,\diffy \\
& = \frac{1}{2} \int_{\mathbb{R}} \rho_+ \varphi_+ \DN_+(\eta) \varphi_+ \,\diffx + \frac{1}{2} \int_{\mathbb{R}}  \rho_- \varphi_- \DN_-(\eta) \varphi_-  \,\diffx. \end{align*}
Using \eqref{continuity normal velocity} and \eqref{psi phipm identity}, this can be rewritten as 
\[ 
K = \frac{1}{2} \int_{\mathbb{R}} \psi \DN_-(\eta) B(\eta)^{-1} \DN_+(\eta) \psi \,\diffx.
\]
Thus, we can view $K$ as the $C^\infty(\nbhdO \cap \Vspace, \mathbb{R})$ functional acting on $u$ given by
\be \label{definition kinE}
\kinE(u) := \frac{1}{2} \int_{\mathbb{R}} \psi A(\eta) \, \psi \,\diffx,  
\ee
where recall $A(\eta)$ was defined in \eqref{definition A}.    Likewise, the potential energy for the system is described by the functional
\[ 
\potE(u):= -\frac{1}{2} \int_{\mathbb{R}} g \jump{\rho} \eta^2 \,\diffx + \sigma \int_{\mathbb{R}} \left(\sqrt{1+(\eta^\prime)^2}  -1 \right) \,\diffx.
\]
The total energy is thus 
\be \begin{split} 
\eng(u) &:= \kinE(u) + \potE(u) \\
& = \frac{1}{2} \int_{\mathbb{R}} \psi A(\eta) \, \psi \,\diffx -\frac{1}{2} \int_{\mathbb{R}} g \jump{\rho} \eta^2 \,\diffx + \sigma \int_{\mathbb{R}} \left(\sqrt{1+(\eta^\prime)^2}  -1 \right) \,\diffx.
\end{split}\label{definition energy} \ee
By our choice of spaces, $\eng \in C^\infty(\nbhdO \cap \Vspace; \mathbb{R})$.   We claim, moreover, that $\Diff\eng(u)$ can be extended to a mapping defined on the entire dual space $\Xspace^*$.  This rather technical fact is necessary in order to reformulate the problem as a Hamiltonian system.  

Before addressing this question, we pause to record the following crucial formulas for the Fr\'echet derivatives of the nonlocal operators $\DN_\pm(\eta)$ and $A(\eta)$.

\begin{lemma}[First derivatives] \label{DG formula lemma}  
Let $(\eta,\psi) \in \nbhdO \cap \Vspace$, $\dot\eta \in \Vspace_1$, and $\xi \in \Vspace_2$ be given.  
\begin{enumerate}[label=\rm(\alph*)]
\item The Fr\'echet derivative of $\DN_\pm(\eta)$ admits the representation formula
\be \begin{split}
\int_{\mathbb{R}}  \xi \left \langle \Diff\DN_\pm(\eta) \dot \eta, \, \psi \right\rangle \,\diffx & = \int_{\mathbb{R}}  \left( a_1^\pm(\eta, \psi) \xi^\prime + a_2^\pm(\eta, \psi) \DN_\pm(\eta) \xi \right) \dot\eta \,\diffx,
\end{split}  \label{first derivative DN formula} \ee
with  
\be \label{def a1 a2}  \begin{split}
 a_{1}^\pm(\eta, \psi) & := \frac{1}{1+(\eta^\prime)^2} \left( \mp \psi^\prime - \eta^\prime \DN_\pm(\eta) \psi \right) \\
 a_{2}^\pm(\eta, \psi) & :=\frac{1}{1+(\eta^\prime)^2} \left( \pm \DN_\pm(\eta) \psi - \eta^\prime \psi^\prime \right) .
 \end{split}\ee
 \item The Fr\'echet derivative of $A(\eta)$ admits the representation formula 
 \be
\begin{split} 
\int_{\mathbb{R}} \xi \left\langle \Diff A(\eta) \dot\eta, \psi \right\rangle \,\diffx & =  \sum_\pm \rho_\pm \int_{\mathbb{R}} \left( a_1^\pm(\eta, A(\eta) \DN_\pm(\eta)^{-1} \psi) \left( A(\eta) \DN_\pm(\eta)^{-1} \xi \right)^\prime  \right) \dot\eta \,\diffx \\
& \qquad + \sum_\pm \rho_\pm \int_{\mathbb{R}} \left( a_2^\pm(\eta, A(\eta) \DN_\pm(\eta)^{-1} \psi)  A(\eta)  \xi    \right) \dot\eta \,\diffx.
\end{split} \label{first derivative A formula} \ee
 \end{enumerate}
\end{lemma}

\begin{remark} \label{a1 a2 remark}
Observe that by \eqref{grad Phi identities}, $a_1^\pm(\eta, \psi) = \mp(\partial_x {\HE}_\pm(\eta) \psi)|_{\mathscr{S}}$ while $a_2^\pm(\eta, \psi) = -(\partial_y {\HE}_\pm(\eta) \psi)|_{\mathscr{S}}$.  In particular, this means that both are linear in $\psi$.
\end{remark}

\begin{proof}[Proof of Lemma~\ref{DG formula lemma}]
The formula \eqref{first derivative DN formula} for $\Diff\DN_\pm(\eta)$ can be derived using the same method as the standard one-fluid case. To obtain \eqref{first derivative A formula}, it is easier to first consider the derivative of 
\be \label{A inverse formula} 
A(\eta)^{-1} = \DN_+(\eta)^{-1} B(\eta) \DN_-(\eta)^{-1} = \rho_+ \DN_+(\eta)^{-1} + \rho_- \DN_-(\eta)^{-1}.
\ee
Then,
\begin{align*}
\left\langle \Diff A(\eta) \dot\eta, \, \psi \right\rangle & = -A(\eta) \left\langle \Diff(A(\eta)^{-1}) \dot\eta, A(\eta) \psi \right\rangle \\
& =  \sum_\pm \rho_\pm A(\eta) \DN_\pm(\eta)^{-1} \left\langle \Diff\DN_\pm(\eta) \dot \eta, \DN_\pm(\eta)^{-1} A(\eta) \psi \right\rangle. 
\end{align*} 
Using the self-adjointness of $\DN_\pm(\eta)$ and the formula \eqref{first derivative DN formula} for $\Diff\DN_\pm(\eta)$, this leads immediately to \eqref{first derivative A formula}.
\end{proof}

We are now able to prove that $\Diff\eng(u)$ extends to $\Xspace^*$ when the base point $u$ has sufficient regularity.

\begin{lemma}[Energy extension] \label{energy extension lemma}  There exists a mapping $\nabla \eng \in C^\infty(\mathcal{O} \cap \Vspace; \Xspace^*)$ such that
\[ 
\langle \nabla \eng(u), v \rangle_{\Xspace^* \times \Xspace} = \Diff\eng(u) v \qquad \textrm{for all } u \in \nbhdO \cap \Vspace, ~v \in \Vspace.
\]
\end{lemma}
\begin{proof}  Let $u = (\eta, \psi) \in \mathcal{O} \cap \Vspace$ and $\dot u = (\dot \eta, \dot \psi) \in \Vspace$ be given.  Then from the definition of $\eng$ in \eqref{definition energy} and the self-adjointness of $A(\eta)$, we compute that
\begin{align*}
\Diff\eng(u) \dot u & = \frac{1}{2} \int_{\mathbb{R}} \psi \langle \Diff A(\eta) \dot \eta, \, \psi \rangle \,\diffx + \int_{\mathbb{R}} \dot \psi A(\eta) \psi \,\diffx  -\int_{\mathbb{R}} \left( g \jump{\rho} \eta + \sigma \left( \frac{\eta^\prime}{\jbracket{\eta^\prime}} \right)^\prime \right) \dot \eta \,\diffx.
\end{align*}
The latter two terms on the right-hand side certainly correspond to an element of $\Xspace^*$ acting on $\dot u$.  To see the same is true for the first term, we make use of the representation formula \eqref{first derivative A formula} to write
\begin{align*}
\int_{\mathbb{R}} \psi \langle \Diff A(\eta) \dot \eta, \, \psi \rangle \,\diffx & =  \sum_\pm \rho_\pm \int_{\mathbb{R}}  a_1^\pm(\eta, \theta_\pm) \theta_\pm^\prime   \dot\eta \,\diffx  + \sum_\pm \rho_\pm \int_{\mathbb{R}} \left( a_2^\pm(\eta, \theta_\pm)  A(\eta)  \psi    \right) \dot\eta \,\diffx,
\end{align*}
for $a_1^\pm$ and $a_2^\pm$ given by \eqref{def a1 a2} and $\theta_\pm := A(\eta) \DN_\pm(\eta)^{-1} \psi$.   Since $u \in \mathcal{O} \cap \Vspace$, it is easy to check that 
\[ 
A(\eta) \psi,~a_1^\pm(\eta,\theta_\pm),~a_2^\pm(\eta,\theta_\pm) \in L^2(\mathbb{R}), \quad\theta_\pm \in H^1(\mathbb{R}),
\]
and hence the extension $\nabla\eng(u)$ can be defined explicitly as
\[ 
\langle \nabla\eng(u), \, v \rangle_{\Xspace^* \times \Xspace} = (\eng^\prime(u), v)_{L^2}, 
\]
where the $L^2$ gradient $\eng^\prime(u) = (\eng_\eta^\prime(u), \eng_\psi^\prime(u))$ takes the form
\be \label{definition E prime}
\begin{split}
\eng_\eta^\prime(u) & := \frac{1}{2} \sum_{\pm} \rho_\pm \left( a_1^\pm(\eta, \theta_\pm) \theta_\pm^\prime + a_2^\pm(\eta, \theta_\pm)A(\eta)\psi \right) - g\jump{\rho} \eta - \sigma \left( \frac{\eta^\prime}{\jbracket{\eta^\prime}} \right)^\prime , \\
\eng_\psi^\prime(u) & := A(\eta) \psi. 
\end{split}
\ee
This completes the proof.
\end{proof}

\begin{remark} 
Throughout the paper, we use the notational convention that, for a $C^1$ functional $F(\Vspace; \mathbb{R})$ and $u \in \Vspace$,  $\Diff F(u) \in \Vspace^*$ is the Fr\'echet derivative at $u$, $F^\prime(u)$ is the $L^2$ gradient, and $\nabla F(u)$ is an extension of $\Diff F(u)$ to $\Xspace^*$ (should such an extension exist).
\end{remark}

The energy space $\Xspace$ will be endowed with symplectic structure through the prescription of the Poisson map
\be \label{definition J}
J := \begin{pmatrix} 0 & 1 \\ -1 & 0 \end{pmatrix} : \Dom{J} \subset \Xspace^* \to \Xspace
\ee
with domain 
\be \label{definition domain J}
\Dom{J} := \left( H^{-1}(\mathbb{R}) \cap \dot H^{\frac{1}{2}}(\mathbb{R}) \right) \times \left( H^{1}(\mathbb{R}) \cap \dot H^{-\frac{1}{2}}(\mathbb{R}) \right).
\ee
While $J$ appears relatively anodyne at first glance, the difference in regularity and homogeneity between $\Xspace_1$ and $\Xspace_2$ means that it is not bijective.  This unpleasant fact is one of the major barriers to applying the classical GSS method \cite{grillakis1987stability1} to the system.   The next lemma shows, however, that $J$ satisfies the weaker requirements of \cite[Assumption 2]{varholm2020stability}.  

\begin{lemma}[Poisson map] \label{poisson map lemma} 
The Poisson map $J$ defined by \eqref{definition J} satisfies the following.
\begin{enumerate}[label=\rm(\alph*)]
\item \label{domain J dense part} $\Dom{J}$ is dense in $\Xspace^*$; 
\item \label{J injective part} $J$ is injective; and 
\item \label{J skew-adjoint part} $J$ is skew-adjoint in the sense that 
\[ \langle Ju, v \rangle_{\Xspace^* \times \Xspace} = -\langle u, J v \rangle_{\Xspace^{**} \times \Xspace^*} \qquad \textrm{for all } u,\, v \in \Dom{J}.\]
\end{enumerate}
\end{lemma}
\begin{proof}
Part~\ref{domain J dense part} is a consequence of Remark~\ref{density remark}, while \ref{J injective part} and \ref{J skew-adjoint part} are obvious by definition.
\end{proof}

\begin{theorem}[Hamiltonian formulation] 
Consider the abstract Hamiltonian system 
\be \label{Hamiltonian equation}
\partial_t u = J \Diff\eng(u), \quad u|_{t=0} = u_0
\ee
where $u_0 \in \mathcal{O} \cap \Wspace$ is the initial data, $J$ is the canonical symplectic matrix \eqref{definition J}, and the energy $\eng$ is defined in \eqref{definition energy}.  We say $u \in C^0([0,t_0); \mathcal{O} \cap \Wspace)$ is a (weak) solution to \eqref{Hamiltonian equation} provided 
\[ 
\frac{\textup{d}}{\textup{d}t} \langle u(t), w \rangle = -\langle \nabla\eng(u(t)), \, Jw \rangle \qquad \textrm{for all } w \in \Dom{J}
\]
in the distributional sense on the time interval $t \in [0,t_0)$.  This holds if and only if the corresponding $(\eta, \Phi_\pm)$ solves the Eulerian internal wave problem  \eqref{Eulerian water wave problem}. 
\end{theorem}
\begin{proof}
At this formulation of the problem was previously obtained by Benjamin and Bridges \cite{benjamin1997reappraisal1}, we provide a sketch of the argument for completeness.  Suppose that $u(t) = (\eta(t), \psi(t)) \in C^0([0,t_0); \mathcal{O} \cap \Wspace)$ is a weak solution to the Hamiltonian system \eqref{Hamiltonian equation}.  Recalling \eqref{psi phipm identity}, we have that $\Phi_\pm := \mp\HE(\eta) \DN_\pm(\eta)^{-1} A(\eta) \psi \in \dot H^{3+} \cap \dot H^1$ is the velocity potential in $\Omega_\pm$ and satisfies \eqref{Eulerian Laplace problem}.  The definition of the harmonic extension operator $\HE(\eta)$  in \eqref{definition HE} also ensures the kinematic condition holds on $\{ y = \pm d_\pm\}$.  Moreover, from the expression for $\eng^\prime(u)$ obtained in  \eqref{definition E prime}, we see that 
\[ 
\partial_t \eta = \eng_\psi^\prime(u) = A(\eta) \psi,
\]
in the distributional sense.  This is precisely \eqref{Hamiltonian kinematic} and hence corresponds to the kinematic condition on the internal interface \eqref{Eulerian kinematic}.  

We claim that the Bernoulli condition \eqref{Eulerian dynamic} is equivalent to 
\[ 
\partial_t \psi = -\eng_\eta^\prime(u).
\]
interpreted again in the distributional sense.   Observe that, due to Remark~\ref{a1 a2 remark} and the identity \eqref{psi phipm identity}, many of the quantities occurring in $\eng_\eta^\prime(u)$ have physical significance:
\[ 
\theta_\pm = A(\eta) \DN_\pm(\eta)^{-1} \psi = \mp \varphi_\pm, \quad a_1^\pm(\eta, \theta_\pm) = (\partial_x\Phi_\pm)|_{\mathscr{S}}, \quad a_2^\pm(\eta, \theta_\pm) = \pm (\partial_y \Phi_\pm)|_{\mathscr{S}}.
\]
Hence, 
\begin{align*} 
\eng_\eta^\prime(u) &= \frac{1}{2} \sum_\pm \rho_\pm \left( \mp (\partial_x \Phi_\pm)|_{\mathscr{S}}  \varphi_\pm^\prime \pm (\partial_y \Phi_\pm)|_{\mathscr{S}} A(\eta) \psi \right) - g\jump{\rho} \eta - \sigma \left( \frac{\eta^\prime}{\jbracket{\eta^\prime}} \right)^\prime \\
& = -\frac{1}{2} \jump{\rho |\nabla \Phi|^2} + (\partial_t \eta) \jump{\rho \partial_y \Phi} - g\jump{\rho} \eta - \sigma \left( \frac{\eta^\prime}{\jbracket{\eta^\prime}} \right)^\prime,
\end{align*}
where in the second line we have used the kinematic condition \eqref{Hamiltonian kinematic} and the identities \eqref{grad Phi identities}.  Comparing this to equivalent statement of the Bernoulli condition in \eqref{rewritten Bernoulli}, we see that the proof is indeed complete. 
\end{proof}

\subsection{The symmetry group and the momentum}

The internal wave problem is invariant under translations in the $x$-direction, which formally should be associated to the conservation of (horizontal linear) momentum; see, for example, \cite{benjamin1982hamiltonian}.  To put this on firmer ground, we introduce the one-parameter symmetry group
\be \label{definition T}
T(s) u := u(\placeholder - s) \qquad \textrm{for all } u \in \Xspace.
\ee
In the next lemma, we verify that $T$ exhibits the necessary properties for the abstract theory in \cite{varholm2020stability}.

\begin{lemma}[Symmetry] \label{symmetry lemma} 
The translation symmetry group $T$ given by \eqref{definition T} satisfies the following.
\begin{enumerate}[label=\rm(\alph*)]
\item The neighborhood $\mathcal{O}$, $\Xspace^k$ for any $k$, and $I^{-1} \Dom{J}$ are invariant under $T(s)$ for all $s \in \mathbb{R}$.
\item $T$ comprises a flow on $\Xspace$ in the sense that $T(0) = \ident_{\Xspace}$ and $T(s+r) = T(s) T(r)$ for all $s, r \in \mathbb{R}$.  Moreover, $T(s)$ is unitary on $\Xspace$ and an isometry on $\Vspace$ and $\Wspace$ for all $s \in \mathbb{R}$.
\item The symmetry group commutes with the Poisson map in the sense that 
\be \label{J and T commute} 
J I T(\placeholder) = T(\placeholder) J I.
\ee
\item \label{generator part} The infinitesimal generator of $T|_{\Xspace^k}$ is the unbounded linear operator 
\be \label{definition generator} 
T^\prime(0)|_{\Xspace^k} : \Dom{T^\prime(0)} \subset \Xspace^k \to \Xspace^k \qquad u \mapsto -\partial_x u
\ee
with (dense) domain $\Dom{T^\prime(0)|_{\Xspace^k}} := \Xspace^{k+1}.$  In particular,
\[ 
\Dom{T^\prime(0)} = \Dom{T^\prime(0)|_\Xspace} = \Xspace^{\frac{3}{2}}, \quad 
\Dom{T^\prime(0)|_{\Vspace}} = \Xspace^{\frac{5}{2}+} , \quad 
\Dom{T^\prime(0)|_{\Wspace}} = \Xspace^{\frac{7}{2}+}. 
\]
\item \label{J dense range part} The subspace $\Rng{J}  \cap \Dom{T^\prime(0)|_{\Wspace}}$ is dense in $\Xspace$.
\item \label{energy conserved part} We have $E(T(s) u) = E(u)$ for all $s \in \mathbb{R}$ and $u \in \mathcal{O} \cap \Vspace$.
\end{enumerate}
\end{lemma}

\begin{proof}
Most of these facts are simple to confirm, so we omit the details.  However, part~\ref{J dense range part} merits closer consideration since its conclusion is the key assumption in \cite{varholm2020stability} that replaces the hypothesis that $J$ is bijective in the standard GSS approach.  First note that 
\[ 
\Rng{J} = \left( H^1(\mathbb{R}) \cap \dot H^{-\frac{1}{2}}(\mathbb{R}) \right) \times \left( H^{-1}(\mathbb{R}) \cap \dot H^{\frac{1}{2}}(\mathbb{R}) \right),
\]
and hence by part~\ref{generator part} we have that
\[ 
\Dom{T^\prime(0)|_\Wspace} \cap \Rng{J} = \left( H^{4+}(\mathbb{R}) \cap \dot H^{-\frac{1}{2}}(\mathbb{R}) \right) \times \left( H^{-1}(\mathbb{R}) \cap \dot H^{\frac{1}{2}}(\mathbb{R}) \cap \dot H^{\frac{7}{2}+}(\mathbb{R})  \right).
\]
This is indeed dense in $\Xspace$ due to Remark~\ref{density remark}.
\end{proof}

Now, letting
\[  
\mom_\pm := \pm \int_{\mathbb{R}} \rho_\pm \eta^\prime \varphi_\pm \,\diffx
\]
represent the momentum in $\Omega_\pm$, we have that the total momentum carried by the wave is 
\be \mom(u) := \mom_+(u) + \mom_-(u) = -\int_{\mathbb{R}} \eta^\prime \psi \,\diffx,\label{definition momentum} \ee
which defines a $C^\infty(\mathcal{O} \cap \Vspace; \mathbb{R})$ functional.  The next lemma establishes that $\mom$ is indeed generated by the translation invariance in the sense that \eqref{symmetry generates waves} holds.  In particular, together with Lemmas~\ref{energy extension lemma} and \ref{symmetry lemma}, this completes the proof that \cite[Assumption 3 and Assumption 4]{varholm2020stability} hold.

\begin{lemma}[Momentum] \label{momentum lemma} 
The momentum functional $\mom$ given by \eqref{definition momentum} satisfies the following.  
\begin{enumerate}[label=\rm(\alph*)]
\item \label{momentum extension part} There exists a mapping $\nabla \mom \in C^0(\mathcal{O} \cap \Vspace; \Xspace^*)$ such that, for all $u \in \mathcal{O} \cap \Vspace$,   $\nabla\mom(u)$ is an extensions of the Fr\'echet derivative $\Diff\mom(u)$.  
\item \label{T generates P part} For all such $u \in \mathcal{O} \cap \Vspace$ it holds that $\nabla P(u) \in \Dom{J}$ and, moreover,
\be \label{symmetry generates waves} 
T^\prime(0) u = J \nabla P(u)  \qquad \textrm{for all } u \in \mathcal{O} \cap \Dom{T^\prime(0)}.
\ee
\end{enumerate}
\end{lemma}
\begin{proof}
The existence of the extension $\nabla \mom$ in part~\ref{momentum extension part} is obvious from the formulas for the derivative $\Diff\mom$.  In particular, for $u = (\eta, \psi) \in \mathcal{O} \cap \Vspace$ and $\dot u = (\dot\eta, \dot\psi) \in \Vspace$, we have  
\be \label{DP}
 \Diff\mom(u) \dot u = \int_{\mathbb{R}} \psi^\prime \dot \eta \,\diffx - \int_{\mathbb{R}} \eta^\prime \dot \psi \,\diffx =: \langle \nabla\mom(u), \, \dot u\rangle_{\Xspace^* \times \Xspace}.
 \ee
The right-hand side above clearly defines an element of $\Xspace^*$ that depends continuously on $u$.  In particular, it has the explicit $L^2$ gradient
\be \label{definition P prime}
\mom^\prime(u) = (\mom_\eta^\prime(u), \mom_\psi^\prime(u)), \qquad \mom_\eta^\prime(u) := \psi^\prime, \quad \mom_\psi^\prime(u) := -\eta^\prime.
\ee

From this it is also clear that $\nabla \mom(u) \in \Dom{J}$ for $u \in \mathcal{O} \cap \Vspace$.  Noting that $\Dom{T^\prime(0)} = \Xspace^{3/2} \subset \Vspace$, the identity \eqref{symmetry generates waves} now follows from  the definitions of $J$ in \eqref{definition J} and $T^\prime(0)$ in \eqref{definition generator}.
\end{proof}

\subsection{Traveling waves} \label{traveling wave section}

In Hamiltonian language, a traveling internal wave is a solution to \eqref{Hamiltonian equation} taking the form 
\be \label{definition traveling wave} 
u(t) = T(ct) U,
\ee
for some wave speed $c \in \mathbb{R}$ and time-independent bound state $U \in \mathcal{O} \cap \Wspace$.  Let us now discuss in somewhat finer detail the existence theory obtained by Nilsson in \cite{nilsson2017internal}.  

Recall that we have defined the dimensionless parameters $\beta$, $\lambda$, $\varrho$, and $h$ in \eqref{dimensionless parameters} and \eqref{definition ratios}.  Let $\mathscr{T} := \{ z \in \mathbb{C} : \realpart{z} \in (-r,r)\}$ be a thin slab centered on the imaginary axis.  For $r > 0$ sufficiently small, we have by the dispersion relation \eqref{dispersion relation} that there exist three curves in the $(\beta,\lambda)$-plane along which the spectrum of the linearized problem in $\mathscr{T}$ crosses the real or imaginary axis.  

Consider first the curve $\Gamma_1$, which is simply the line $\lambda = \lambda_0$.  Immediately below it and to the right of $\beta = \beta_0$, the spectrum in $\mathscr{T}$ consists of a pair of oppositely signed real eigenvalues and a complex conjugate pair on the imaginary axis.    Passing through $\Gamma_1$, the imaginary eigenvalues collide at the origin then move along the real axis.  This same $0^2$ resonance is associated with transition from periodic solutions to solitons in the steady KdV equation, for example.  On the curve 
\[ 
\Gamma_2 := \left\{ (\beta(\xi), \lambda(\xi)) : \xi \in [0,\infty) \right\},
\]
where 
\be \label{parameterization Gamma}
\begin{split}
\beta(\xi) & := \sum_\pm \frac{\rho_\pm}{\rho_-} \frac{d_+}{d_\pm} \left( \frac{- \sin{(\frac{d_\pm}{d_+} \xi) \cos{(\frac{d_\pm}{d_+} \xi)} + \frac{d_\pm}{d_+} \xi}}{2 \xi \sin^2{(\frac{d_\pm}{d_+} \xi)}}  \right),  \\ 
\lambda(\xi) &:= \beta(\xi)^2 + \xi \sum_{\pm} \frac{\rho_\pm}{\rho_-} \coth{(\frac{d_+}{d_\pm} \xi)},
\end{split}
 \ee
the spectrum in $\mathscr{T}$ consists of two real eigenvalues with multiplicity $2$.  In the region bounded by $\Gamma_1$ and $\Gamma_2$, there are two pairs of oppositely signed simple real eigenvalues.  

Nilsson's approach is to fix $\beta > \beta_0$ and treat $\lambda$ as a bifurcation parameter with $0 < \lambda-\lambda_0 \ll 1$.  This ensures that $(\beta,\lambda)$ remains in the Region~A depicted in Figure~\ref{dispersion figure}, which is the narrow open set bounded below by  $\Gamma_1$ and lying beneath $\Gamma_2$.  Because he opts to non-dimensionalize the system at the outset, translating his result to our setting involves introducing some heavy notation.  Thankfully, this will be pared down soon.

\begin{theorem}[Nilsson \cite{nilsson2017internal}] \label{region A existence theorem} 
Let $\{\Pi_\varepsilon = (\rho_{\pm\varepsilon}, d_{\pm\varepsilon},\sigma_\varepsilon,c_\varepsilon) : 0 < \varepsilon \ll 1\}$ be a smooth curve in the dimensional parameter space such that the corresponding Bond number is fixed to $\beta > \beta_0$ and $\lambda = \lambda_0 + \varepsilon^2$.
 \begin{enumerate}[label=\rm(\alph*)]
 \item \label{region A existence KdV part} Suppose that $\varrho_\varepsilon - 1/h_\varepsilon^2 = O(1)$ as $\varepsilon \searrow 0$.  Then for any $k > 1/2$, there exists a smooth curve 
 \[ 
 \mathscr{C}_\beta^{\mathrm{A}} = \{ u_{\varepsilon;\,\beta}^{\mathrm{A}} : 0 < \varepsilon \ll 1 \}\subset \Xspace^k
 \]
 so that $u_{\varepsilon;\,\beta}^{\mathrm{A}}$ is a traveling internal wave for the parameter values $\Pi_\varepsilon$.  Along this curve, the free surface profile has leading-order form 
 \be \label{region A KdV scaling}
  \eta_{\varepsilon;\, \beta}^{\mathrm{A}} = \frac{\varepsilon^2 d_{+} }{\varrho-1/h^2} \sechsq{\left(\frac{\varepsilon \placeholder}{2 d_{+} \sqrt{\beta-\beta_0}}\right) }+ O(\varepsilon^3) \qquad \textrm{in $\Xspace_1^k$ as $\varepsilon\searrow 0$.} 
  \ee 
  \item \label{region A existence Gardner part} Suppose instead that $\varrho_\varepsilon - 1/h_\varepsilon^2 = \kappa \varepsilon$ for a fixed $\kappa \neq 0$.  Then for any $k > 1/2$ there exists two smooth curves 
  \[ \mathscr{C}_{\beta,\kappa,\pm}^{\mathrm{A}} = \{ u_{\varepsilon;\,\beta,\kappa,\pm}^{\mathrm{A}} : 0 < \varepsilon \ll 1\}  \subset \Xspace^k\]
  so that $u_{\varepsilon;\,\beta,\kappa,\pm}^{\mathrm{A}}$ is a traveling internal wave for the parameter values $\Pi_\varepsilon$. Along $\mathscr{C}_{\beta,\kappa,\pm}^{\mathrm{A}}$, the free surface profile has leading-order form
   \be \label{region A Gardner scaling}
 \eta_{\varepsilon;\, \beta,\kappa,\pm}^{\mathrm{A}} = \frac{2 \varepsilon d_{+} }{\kappa \pm \sqrt{\kappa^2 + 4(\varrho + 1/h^3)} \cosh\left( \dfrac{\varepsilon \placeholder}{d_{+} \sqrt{\beta-\beta_0}} \right)} + O(\varepsilon^3) \qquad \textrm{in $\Xspace_1^k$ as $\varepsilon \searrow 0$.}  
 \ee
 \end{enumerate}
\end{theorem}

\begin{remark}
The above solutions are obtained using a center manifold reduction at the point $(\beta,\lambda_0) \in \Gamma_1$.  For the scaling regime of part~\ref{region A existence KdV part}, the reduced equation is a perturbation of steady KdV.  This gives rise to waves with the classical $\sechsq{}$ asymptotics in \eqref{region A KdV scaling}.  However, when $\varrho-1/h^2 \ll 1$, cubic terms enter at leading order, and so one instead obtains an equation of Gardner or mKdV-KdV type.  An important consequence of this construction is that the $O(\varepsilon^3)$ remainder terms in \eqref{region A KdV scaling} and \eqref{region A Gardner scaling} are exponentially decaying and exhibit the same scaling of the spatial variable as the leading-order part.  Note also that the regularity of the solutions is not stated by Nilsson, but follows from a standard bootstrapping argument.
\end{remark}

Theorem~\ref{region A existence theorem} fixes $\beta$ but allows the dimensional parameters to vary.  While convenient for proving existence, this choice is not ideal for stability analysis:  two waves on one of these curves may not necessarily solve the same physical problem. The general theory in \cite{grillakis1987stability1,varholm2020stability} instead asks for a family of bound states parameterized by $c$, with the remaining dimensional parameters held constant.     Given a choice of parameters $(\rho_{\pm*},d_{\pm*},\sigma_*,c_*)$, we therefore let
\be \label{definition beta_c^A}
 (\beta_c,\lambda_c) := \left( \frac{\sigma_*}{d_{+*} \rho_{-*} c^2}, \, -\frac{g \jump{\rho_*}d_{+*}}{\rho_{-*} c^2} \right), \quad \varepsilon_c^{\mathrm{A}} := \sqrt{\lambda_c - \lambda_0} \qquad  \textrm{for } |c-c_*| \ll 1.
 \ee
 The first of these parameterizes a segment of the straight line joining $(\beta_*,\lambda_*)$ to the origin in the $(\beta,\lambda)$-plane, while the second expresses the bifurcation parameter $\varepsilon$ from Theorem~\ref{region A existence theorem} in terms of $c$.
 
The next two corollaries convert Theorem~\ref{region A existence theorem} to statements on bound states indexed by $c$.  In particular, they prove that \cite[Assumption 5]{varholm2020stability} is satisfied.  

\begin{corollary}[KdV bound states] \label{KdV bound state corollary} 
Let $(\rho_{\pm*},d_{\pm*},\sigma_*,c_*)$ be given so that $\varrho_* - 1/h_*^2 \neq 0$ and the corresponding non-dimensional parameters $(\beta_*,\lambda_*)$ lies in Region~A.  
There exists an open interval $\mathscr{I} \ni c_*$ and a family of bound states $\{ U_c^{\mathrm{A}} \}_{c \in \mathscr{I}} \subset \mathcal{O} \cap \Wspace$  having the non-dimensional parameter values $(\beta_c,\lambda_c)$ given by \eqref{definition beta_c^A}.  The free surface profile is 
\[ 
\eta_c^{\mathrm{A}} := \eta_{\varepsilon_c^{\mathrm{A}};\, \beta_c}^{\mathrm{A}} \qquad \textrm{for } c \in \mathscr{I}. 
\]
Moreover, $\{U_c^{\mathrm{A}}\}$ satisfies \cite[Assumption 5]{varholm2020stability} in that the following holds.
\begin{enumerate}[label=\rm(\alph*)] 
\item \label{bound state smooth part} The mapping $c \in \mathscr{I} \mapsto U_c^{\mathrm{A}} \in \mathcal{O} \cap \Wspace$ is $C^1$.
\item \label{bound state technical part} For all $c \in \mathscr{I}$, 
\[ 
U_c^{\mathrm{A}} \in \Dom{T^{\prime}(0)^3} \cap \Dom{JIT^\prime(0)}, \qquad
 U_c^{\mathrm{A}}, \, JIT^\prime(0) U_c \in \Dom{T^\prime(0)|_{\Wspace}}.
\]
\item \label{bound state nontrivial part} Each $U_c^{\mathrm{A}}$ is nontrivial in that $T^\prime(0) U_c \not\equiv 0$ for $c \in \mathscr{I}$.  
\item \label{bound state localized} The waves are localized in that $\lim\inf_{|s| \to \infty} \| T(s) U_c^{\mathrm{A}} - U_c^{\mathrm{A}} \|_{\Xspace} > 0$.
\end{enumerate}
\end{corollary}
\begin{proof}
Let $(\rho_{\pm*},d_{\pm*},\sigma_*,c_*)$ be given as above and assume that the corresponding $(\beta_*,\lambda_*)$ satisfy $\beta_* > \beta_0$ and $0 < \lambda_*-\lambda_0 \ll 1$.  Then for all $|c-c_*| \ll 1$, the dimensional parameters meet the hypotheses of Theorem~\ref{region A existence theorem}\ref{region A existence KdV part}, and so may simply take $U_c^{\mathrm{A}} := u_{\varepsilon_c^{\mathrm{A}};\,\beta_c}$ for $\beta_c$ and $\varepsilon_c^{\mathrm{A}}$ defined according to \eqref{definition beta_c^A}.

The free surface profile from \eqref{region A KdV scaling} is constructed as a solution to a second-order ODE that is a homoclinic to $0$.  It can be verified directly that the origin is a saddle point, and hence $\eta_{\varepsilon; \beta}^{\mathrm{A}}$ is exponentially localized, with uniform decay rate on compact subsets of parameter space.   Moreover, due to the translation invariance, the profile is of class $C^\infty$.  In particular, it is clearly an element of $\Xspace_1^k$ for all $k \geq 1/2$.  Solving the kinematic condition, we see that the corresponding $\psi = \psi_{\varepsilon;\beta}^{\mathrm{A}}$ is likewise smooth and an element of $\Xspace_2^k$ for all $k \geq 1/2$.  Part~\ref{bound state smooth part} now follows from the smooth dependence of $u_{\varepsilon; \beta}$ on $(\varepsilon,\beta)$.  Part~\ref{bound state technical part} certainly holds in view of the (arbitrarily high) regularity of the bound states.  Finally, parts~\ref{bound state nontrivial part} and \ref{bound state localized} are obvious given the form of $\eta_c^{\mathrm{A}}$.
\end{proof}

An identical argument applied to the family of waves in Theorem~\ref{region A existence theorem}\ref{region A existence Gardner part} yields the following.

\begin{corollary}[Gardner bound states]  \label{Gardner bound state corollary} 
Let $(\rho_{\pm*},d_{\pm*},\sigma_*,c_*)$ be given so that the corresponding $(\beta_*,\lambda_*)$ lies in Region~A and $|\varrho_*-1/h_*^2| \eqsim |\lambda_*-\lambda_0|^{1/2}$.   There exists an open interval $\mathscr{I} \ni c_*$ and two families of bound states $\{ {U}_c^{\mathrm{A}\pm} \}_{c \in \mathscr{I}} \subset \mathcal{O} \cap \Wspace$  having the non-dimensional parameter values $(\beta_c^\mathrm{A},\lambda_c^{\mathrm{A}})$ given by \eqref{definition beta_c^A} and with the remaining parameters fixed.  They satisfy \cite[Assumption 5]{varholm2020stability} and the corresponding free surface is given by 
\[ 
\eta_c^{\mathrm{A} \pm} := \eta_{\varepsilon_c^{\mathrm{A}};\, \beta_c,\kappa_c^{\mathrm{A}},\pm}^{\mathrm{A}} \qquad \textrm{for } \kappa_c^{\mathrm{A}} := \frac{1}{\varepsilon_c^{\mathrm{A}}} \left( \varrho_* - \frac{1}{h_*^2} \right), \quad c \in \mathscr{I}. 
\]
\end{corollary}

Consider now the situation where $(\beta,\lambda)$ is contained in Region~C, which is a neighborhood of the curve $\Gamma_2$.  Nilsson uses a center manifold reduction method to construct traveling waves, this time bifurcating from the point $(\beta_0,\lambda_0)$.  Setting $\gammaC := (\varrho+h)/45$, one can show using the parameterization of $\Gamma_2$ that for all $\delta \in \mathbb{R}$, the point
 \be \label{region C lambda beta}
 \beta = \beta_0 + 2(1+\delta) \gammaC\varepsilon^2, \qquad \lambda = \lambda_0 + \gammaC\varepsilon^4
 \ee
 is contained in Region~C for all $0 < \varepsilon \ll 1$.  When $\delta > 0$, it lies below $\Gamma_2$ and for $\delta < 0$, it lies above.  
 
 At $\varepsilon = 0$, this gives the critical parameter value $(\beta_0,\lambda_0)$ where we recall that $0$ is an eigenvalue of multiplicity $4$. The resulting reduced equation on the center manifold thus has a four-dimensional phase space.   When $\varrho -1/h^2 = O(1)$ as $\varepsilon \searrow 0$, after performing a rescaling and truncation, we obtain the ODE
 \be \label{Z1 ODE}
 Z^{\prime\prime\prime\prime} - 2(1+ \delta)  Z^{\prime\prime} + Z - \frac32 \gammaC^{-3/2} \left( \varrho - \frac{1}{h^2} \right) Z^2 = 0.
 \ee
 This equation arises in the study of capillary-gravity waves beneath vacuum in the critical surface tension regime as well as a modeling the buckling of elastic struts \cite{amick1992homoclinic}.   Analysis in \cite{champneys1993bifurcation,buffoni1996bifurcation} shows that, at $\delta = 0$, there is a \emph{primary homoclinic} solution $Z_{0}$ to \eqref{Z1 ODE} that is unimodal, even, and exponentially localized.  Moreover, there is a smooth one-parameter family of homoclinic orbits $\{Z_{\delta}\}_\delta$ defined for $\delta \geq 0$ and $-1 \ll \delta < 0$ that bifurcates from $Z_{0}$.  These solutions are {\it transversely constructed}, in that the stable and unstable manifolds of the zero equilibrium of \eqref{intro kawahara} intersect transversely at $Z = Z_\delta(0)$ at the zero level set of the Hamiltonian energy.  For $\delta \ge 0$, we have that $Z_\delta$ is the unique (up to translation) homoclinic solution to \eqref{intro kawahara} that is positive, even and monotone for $x > 0$ (see \cite{amick1992homoclinic}).     When $-1 \ll \delta < 0$, uniqueness is not known and $Z_\delta$ has exponentially decaying oscillatory tails.   In addition to the primary homoclinic orbits, there exists a ``plethora'' of other solutions to \eqref{Z1 ODE} that take the form of multisolitons; see \cite{buffoni1996bifurcation,devaney1976homoclinic}.   Because these are multimodal, they are unlikely to be amenable to analysis through the general theory in \cite{varholm2020stability} and so we will not consider them here.   
 
On the other hand, if $\varrho-1/h^2 = \kappa \varepsilon^2$, for some $\kappa \neq 0$, then upon rescaling and truncating to leading order, the reduced equation on the center manifold takes the form 
\be \label{Z2 ODE}
 Z^{\prime\prime\prime\prime} - 2(1+ \delta)  Z^{\prime\prime} + Z - \frac32 \gammaC^{-3/2} \kappa Z^2 - 4 \gammaC^{-2} \left( \varrho +\frac{1}{h^3} + \frac{2(\varrho-1)^2}{225 \gammaC} \right) Z^3  = 0.
\ee
In \cite[Appendix B]{nilsson2017internal}, it is shown that, at $\delta = 0$, this ODE has both a positive and negative primary homoclinic solution, which we denote by $Z_{0;\kappa,\pm}$.  As in the non-resonant case, these are exponentially localized, unique up to translation (for the fixed sign), and because they are transversely constructed, they persists for $| \delta| \ll 1$.  Let the corresponding families be denoted $\{Z_{\delta;\kappa,\pm} \}$.

We now state Nilsson's results for this case reformulated in the style of Corollaries~\ref{KdV bound state corollary} and \ref{Gardner bound state corollary}.  Let $(\rho_{\pm*},d_{\pm*},\sigma_*,c_*)$ be given so that the corresponding $(\beta_*,\lambda_*)$ lies in Region~C. In view of \eqref{region C lambda beta}, we define  
\be \label{definition epsilon C region}
\varepsilon_c^{\mathrm{C}} := \left(\frac{\lambda_c - \lambda_0}{\gamma_*}\right)^{1/4}, \qquad \delta_c := \frac{\beta_c - \beta_0}{2 \gamma_* (\varepsilon_c^{\mathrm{C}})^2} -1  \qquad \textrm{for } |c - c_*| \ll 1,
\ee
with $(\beta_c,\lambda_c)$ given in \eqref{definition beta_c^A}.  The existence of bound states is then summarized in the following lemma.  

\begin{lemma}[Region C bound state] \label{region C bound states lemma} 
Let $(\rho_{\pm*},d_{\pm*},\sigma_*,c_*)$ be given so that $\varrho_*-1/h_*^2 \neq 0$ and the corresponding non-dimensional parameters $(\beta_*,\lambda_*)$ lie in Region~C with $0 < \varepsilon_{c_*}^\mathrm{C} \ll 1$.
\begin{enumerate}[label=\rm(\alph*)] 
\item \label{region C kawahara part}  There exists an open interval $\mathscr{I} \ni c_*$ and a family of bound states $\{ U_c^{\mathrm{C}} \}_{c \in \mathscr{I}} \subset \mathcal{O} \cap \Wspace$ having the non-dimensional parameter values $(\beta_c,\lambda_c)$ and satisfying \cite[Assumption 5]{varholm2020stability}.  The corresponding free surface profile takes the form 
 \be \label{region C Z1 scaling}
 \eta_c^{\mathrm{C}} = \varepsilon^4 d_{+} \sqrt{\gamma} Z_{\delta}\left( \frac{\varepsilon \placeholder}{d_{+}} \right) + O\left(\varepsilon^5\right) \qquad \textrm{in } \Xspace_1^k
 \ee
 with $\varepsilon = \varepsilon_c^{\mathrm{C}}$ and $\delta = \delta_c^{\mathrm{C}}$ given by \eqref{definition epsilon C region}, $d_+ = d_{+*}$, and $\gamma = \gamma_*$.
\item Suppose that $|\varrho_* - 1/h_*^2| \eqsim (\epsilon_{c_*}^{\mathrm{C}})^2 \ll 1$.  Then there exists an open interval $\mathscr{I} \ni c_*$ and two families of bound states $\{ U_c^{\mathrm{C}\pm} \}_{c \in \mathscr{I}} \subset \mathcal{O} \cap \Wspace$ having the non-dimensional parameter values $(\beta_c,\lambda_c)$ and satisfying \cite[Assumption 5]{varholm2020stability}.  The corresponding free surface profile takes the form 
 \be \label{region C Z2 scaling}
 \eta_{c}^{\mathrm{C}\pm} = \varepsilon^2 d_{+} \sqrt{\gammaC}  Z_{\delta; \,\kappa,\pm} \left( \frac{\varepsilon \placeholder}{d_{+}} \right) + O\left(\varepsilon^3\right) \quad \textrm{in } \Xspace_1^k\qquad \textrm{for } \kappa = \kappa_c^{\mathrm{C}} := \frac{1}{\varepsilon_c^2} \left( \varrho_*-\frac{1}{h_*^2} \right),
 \ee
 with $\varepsilon = \varepsilon_c^{\mathrm{C}}$ and $\delta = \delta_c^{\mathrm{C}}$ given by \eqref{definition epsilon C region}, $d_+ = d_{+*}$, and $\gamma = \gamma_*$. 
\end{enumerate}
\end{lemma}
\begin{remark}
While we will carry out many of the calculations for both families $\{U_c^{\mathrm{C}}\}$ and $\{ U_c^{\mathrm{C}\pm}\}$, we only obtain a stability result for the former.  In the latter case, we find that the rescaled linearized augmented potential does not converge precisely to the linearization of \eqref{Z1 ODE}, which obstructs the spectral analysis in the next section; see Lemma~\ref{lem limit rescaled}\ref{scaled Q in region C}.  
\end{remark}
 
We conclude this section by noting that Nilsson also proves the existence of many types of traveling waves with $(\beta,\lambda)$ in a neighborhood of the bifurcation curve 
\be \label{Gamma 3 parameterization}
\Gamma_3 = \{ (\beta(i\xi), \lambda(i\xi)) : \xi \in [0,\infty) \},
\ee
with $(\beta(\xi), \lambda(\xi))$ given by \eqref{parameterization Gamma}.  The stability of these solutions will be the subject of a forthcoming work.
 
 \section{Spectral analysis} \label{spectrum section}

Observe that if $u(t) = T(ct)U$ is a traveling wave for the bound state $U \in \mathcal{O} \cap \Wspace$ and wave speed $c \in \mathbb{R}$, then necessarily by \eqref{Hamiltonian equation} and Lemma~\ref{symmetry lemma}\ref{energy conserved part} we have 
\[ 
\frac{\diff u}{\diff t} = c T^\prime(0) U = J \Diff\eng(U).
\]
Combining this with \eqref{symmetry generates waves}, we obtain  the steady equation   
\[  
\Diff\eng(U) = c \Diff\mom(U),
\]
where we have used that $J$ is injective.  This motivates us to consider the \emph{augmented Hamiltonian}, which for a fixed $c$ is the functional $\augHam \in C^\infty(\nbhdO \cap \Vspace; \mathbb{R})$ given by
\[
\augHam(u) := E(u) - c P(u).
\]
The above calculation shows that bound states are critical points of $\augHam$.  It also suggests that one can construct such solutions as constrained extrema of the energy on level sets of the momentum, with the wave speed a Lagrange multiplier.   In order to exploit this connection, we must first understand better the second derivative of $\augHam$. 

With that in mind, this section is devoted to the quite difficult task of computing the spectrum of the linearized augmented Hamiltonian at either a shear flow or small-amplitude internal capillary-gravity wave.   Here we will follow the general approach of Mielke \cite{mielke2002energetic}, which was also the basis for the calculation in \cite{varholm2020stability}.  The strategy has two steps.  First, via the kinematic condition $\psi$ is eliminated in favor of $\eta$.  Making this substitution in the definition of $\augHam$ gives the so-called augmented potential $\augV = \augV(\eta)$, which proves to be much more amenable to analysis.  In particular, we show in Section~\ref{subsec aug potential} that its second variation at a critical point is characterized by a certain second-order nonlocal differential operator $Q_c(\eta)$.  As one might predict, $Q_c(0)$ is a Fourier multiplier whose symbol is related directly to the dispersion relation \eqref{dispersion relation}.

This is enough to characterize the continuous spectrum of $\Diff^2\augV(\eta)$ when $\eta$ is sufficiently small amplitude; see Lemma~\ref{cont spec lemma}.  Determining the discrete spectrum, however, requires considerably more effort.  Following Mielke, the second step is to conjugate $Q_c(\eta)$ with a rescaling $S_\varepsilon$ informed by the asymptotics of $\eta$ discussed in Section~\ref{traveling wave section}.  Briefly put, the idea here is to show that linearization and scaling almost commute.  It is well known that in the shallow water regime, the internal wave system can be modeled by nonlinear dispersive PDEs such as KdV or Gardner.   We seek to prove that imposing this scaling on the linearized operator $Q_c(\eta)$ via conjugation by $S_\varepsilon$ will, to leading order, coincide with the linearization of the corresponding model equation.  After a delicate calculation, we do indeed find that in the long-wave limit $\varepsilon \searrow 0$, the rescaled operator $S_\varepsilon^{-1} Q_c(\eta) S_\varepsilon$ converges (in an appropriate sense) to the linearized steady KdV or Gardner equation in the case of Region~A, and to the linearization of \eqref{Z1 ODE} or \eqref{Z2 ODE} in the case of Region~C.  This is the subject of Section~\ref{sec rescaling}.  In Section~\ref{spectrum linearized augV section}, we prove that the spectrum of $\Diff^2\augV$ is qualitatively the same as that of this limiting rescaled operator.

Lastly, in Section~\ref{spectrum linearized augHam section} we take the hard-won information about the spectrum of $\Diff^2\augV$ and translate it back to that of $\Diff^2\augHam$.  For $U_c$ one of the family of bound states described in Section~\ref{traveling wave section}, we confirm that $\Diff^2\augHam(U_c)$ extends to a self-adjoint operator on $\Xspace$ that has Morse index $1$.  This is the final hypothesis in the general theory \cite{varholm2020stability}.

\subsection{The augmented potential and its derivatives}\label{subsec aug potential}

If $u_* = (\eta_*, \psi_*)$ is a critical point of $\augHam$, then in particular $\Diff_\psi \eng(u_*) = c \Diff_\psi \mom(u_*)$.  Because $\potE$ is independent of $\psi$ and $A(\eta)$ is self-adjoint, we see that 
\[ \Diff_\psi \eng(u) \dot \psi = \int_{\mathbb{R}} \dot \psi A(\eta) \psi \,\diffx.\]
 Combining this with \eqref{DP} we find that $\psi_*$ can be uniquely determined from $\eta_*$ via
\be  \psi_*(\eta) := -cA(\eta)^{-1} \eta^\prime. \label{definition psistar} \ee
Note that $\psi_*$ also depends on $c$, but in this section the wave speed will be fixed, so there is no harm in suppressing it. In fact it will turn out to be easier to work with $\varphi_*$ rather than $\psi_*$. So we recall from \eqref{def psi} and \eqref{kinematic phipm} that
\be \label{relation psistar phi}
\psi_* = \rho_- \varphi_{*-} - \rho_+ \varphi_{*+}, \qquad \varphi_{*\pm} = \pm c \DN_{\pm}(\eta)^{-1} \eta'.
\ee
When there is no risk of confusion, we will drop the $*$ subscripts to declutter the notation.  

Recall from Remark~\ref{a1 a2 remark} that the coefficients $a_1^\pm$ and $a_2^\pm$ that arise in the first derivative formula \eqref{first derivative DN formula} for $\DN_\pm(\eta)$ can be alternatively be expressed as
\[
a_1^\pm(\eta, \phi) := \mp(\partial_x \HE_\pm(\eta) \phi )|_{\mathscr{S}}, \qquad a_2^\pm := -(\partial_y \HE_\pm(\eta) \phi)|_{\mathscr{S}}.
\]
Therefore, when they are evaluated at $\phi = \varphi_\pm$, they give (up to a sign) the trace of the velocity field on ${\mathscr{S}}$.  Following \cite[Section 6]{varholm2020stability}, we introduce the related functions
\be\label{def b}
b_1^\pm := \mp a_1^\pm(\eta, \varphi_\pm) - c, \qquad b_2^\pm := - a_2^\pm(\eta, \varphi_\pm).
\ee
This way, $(b_1^\pm, b_2^\pm)$ represents the \emph{relative velocity} in $\Omega_\pm$ restricted to the interface. Consequently, for $\eta \in \Wspace_1$, we have from \eqref{relation psistar phi} that $b_1^\pm, b_2^\pm \in H^{2+}(\mathbb{R})$.  Notice also that, because $u$ represents a traveling wave, the kinematic condition \eqref{Hamiltonian kinematic} gives 
\be\label{kinematic for b}
b_2^\pm = \eta' b_1^\pm.
\ee

Differentiating \eqref{relation psistar phi}, we find that
\[
\Diff\psi_*(\eta) \dot\eta = \rho_- \Diff \varphi_-(\eta) \dot\eta - \rho_+ \Diff \varphi_+(\eta) \dot\eta, \qquad \Diff\varphi_\pm(\eta) \dot\eta = \pm \langle c \Diff(\DN_\pm(\eta)^{-1}) \dot\eta, \, \eta^\prime\rangle \pm c \DN_{\pm}(\eta)^{-1} \dot\eta^\prime.
\]
On the other hand, 
\[
\langle \Diff(\DN_\pm(\eta)^{-1}) \dot\eta,\,  \eta^\prime \rangle = -\DN_\pm(\eta)^{-1} \langle \Diff\DN_\pm(\eta)\dot\eta, \,  \DN_\pm(\eta)^{-1}  \eta^\prime \rangle,
\]
and so we may infer from Lemma \ref{DG formula lemma} and \eqref{relation psistar phi} that
\begin{align*}
\DN_\pm(\eta) \langle \Diff(\DN_\pm(\eta)^{-1}) \dot\eta,\, c \eta^\prime \rangle 
& = \parn{a_1^\pm(\eta, \pm \varphi_\pm) \dot\eta}^\prime - \DN_\pm(\eta) \parn{a_2^\pm(\eta, \pm \varphi_\pm) \dot\eta} \\
& = \pm \parn{a_1^{\pm}(\eta,\varphi_\pm) \dot\eta}^\prime \mp \DN_\pm(\eta)\parn{a_2^\pm(\eta,\varphi_\pm) \dot\eta}.
\end{align*}
Thus, 
\begin{align*}
\Diff\varphi_\pm(\eta) \dot\eta & = \DN_\pm(\eta)^{-1} \parn{a_1^{\pm}(\eta,\varphi_\pm) \dot\eta}^\prime - a_2^\pm(\eta,\varphi_\pm) \dot\eta \pm c \DN_{\pm}(\eta)^{-1} \dot\eta^\prime \\
& = \mp \DN_\pm(\eta)^{-1} \parn{b_1^{\pm} \dot\eta}^\prime + b_2^\pm \dot\eta,
\end{align*}
and hence
\begin{equation}\label{compute Dpsistar}
\begin{split}
\Diff\psi_*(\eta) \dot\eta & = \underbrace{\sum_{\pm} \rho_\pm \DN_\pm(\eta)^{-1} \parn{b_1^\pm \dot\eta}^\prime}_{\textstyle{=: \DpsiS \dot\eta}} - \underbrace{\sum_{\pm} \pm \rho_\pm b_2^\pm \dot\eta}_{\textstyle{=: \DpsiT \dot\eta}}.
\end{split}
\end{equation}

Now, let the \emph{augmented potential} be the functional $\augV \in C^\infty(\mathcal{O} \cap \Vspace; \mathbb{R})$ given by
\be \augV(\eta) := \augHam(\eta, \psi_*(\eta)) = \min_{\psi} \augHam(\eta, \psi).\label{definition augV} \ee
While it is not immediately obvious, for small-amplitude waves the spectrum of $\Diff^2 \augHam$ can be determined from that of $\Diff^2\augV$.  We therefore devote the remainder of this subsection to studying the second variation of $\augV$. In particular, we will derive an analytically tractable quadratic form representation defined in terms of physical quantities.  

An essential ingredient in all of these calculations is having access to concise formulas for the variations of the many nonlocal operators.  First, we need the following elementary second derivative formula for the Dirichlet--Neumann operators $\DN_\pm$.  Here we use notation similar to that in \cite{mielke2002energetic,varholm2020stability}.

\begin{lemma}[Second derivative of $\DN_\pm$]\label{D^2G formula lemma}
For all $u = (\eta, \psi) \in \mathcal{O} \cap \Vspace$ and $\dot\eta \in \Vspace_1$, it holds that
\be \label{second derivative DN formula} 
\begin{split}
\int_{\mathbb{R}} \psi \left\langle \Diff^2\DN_\pm(\eta)[\dot\eta,\dot\eta], \, \psi\right\rangle \,\diffx & =  \int_{\mathbb{R}} \left( a_4^\pm(u) \dot\eta^2 + 2a_2^\pm(u) \dot\eta \DN_\pm(\eta) \left( a_2^\pm(u) \dot\eta \right) \right) \,\diffx,
\end{split}
\ee
where
\be \label{def a3 a4}
\begin{split}
a_4^\pm(u) &:= -2 a_1^\pm(u)^\prime a_2^\pm(u),
\end{split}
\ee
and $a_1^\pm$, $a_2^\pm$ are given by \eqref{def a1 a2}.
\end{lemma}
\begin{proof}
This is a straightforward though quite tedious calculation.  
\end{proof}

Far more involved is the second derivative of $A(\eta)$, a formula for which is given in the next lemma.  As the proof is rather long but not especially deep, we delay it to Appendix~\ref{identities appendix}.

\begin{lemma}[Second derivative of $A$]  \label{second variation A lemma} 
For all $u = (\eta, \psi) \in \mathcal{O} \cap \Vspace$ and $\dot\eta \in \Vspace_1$, it holds that
\be \begin{split} 
\int_{\mathbb{R}} \psi \left\langle \Diff^2 A(\eta)[\dot \eta, \dot \eta], \,  \psi \right\rangle \,\diffx & =  \int_{\mathbb{R}}  \Big( a_4(u) \dot\eta + 2 \sum_\pm \rho_\pm a_2^\pm(\eta, \theta_\pm) \DN_\pm(\eta) \left( a_2^\pm(\eta,\theta_\pm) \dot\eta \right) \\ 
& \qquad \qquad- 2 \mathscr{M}(u)\dot\eta  + 2\mathscr{N}(u)\dot\eta\Big) \dot\eta \,\diffx, 
\end{split} \label{second derivative A formula} \ee
where we define the functions
\be
 \theta_\pm(u) := \DN_\pm(\eta)^{-1} A(\eta) \psi, \qquad  a_4(u) := \sum_{\pm} \rho_\pm a_4^\pm(\eta, \theta_\pm), \label{def theta and a4}
 \ee
 and linear operators 
\begin{align}
 \mathscr{L}_\pm(u) \dot\eta & :=   - \DN_\pm(\eta)^{-1} \left( a_1^\pm(\eta, \theta_\pm) \dot \eta \right)^\prime + a_2^\pm(\eta, \theta_\pm) \dot \eta, \qquad \mathscr{L}(u) := \sum_{\pm} \rho_\pm \mathscr{L}_\pm(u) \label{def script L} \\
\mathscr{M}(u) \dot\eta &:= \sum_\pm \rho_\pm \left( a_1^\pm(\eta, \theta_\pm) (\mathscr{L}_\pm(u)\dot\eta)^\prime + a_2^\pm(\eta,\theta_\pm) \DN_\pm(\eta) \mathscr{L}_\pm(u)\dot\eta \right)  \label{def script M} \\
\mathscr{N}(u)\dot\eta & := \sum_{\pm}  \rho_\pm \left( a_1^\pm(\eta,\theta_\pm) \left( A(\eta) \DN_\pm(\eta)^{-1} \mathscr{L}(u) \dot\eta \right)^\prime + a_2^\pm(\eta,\theta_\pm) A(\eta) \mathscr{L}(u)\dot\eta \right). \label{def script N}
\end{align}
\end{lemma}
\begin{remark} \label{second derivative A remark}
Formally setting $\rho_+ = 0$ and $\rho_- = 1$ recovers the standard one-fluid model with normalized density. We can see from \eqref{A inverse formula} that this would imply $A(\eta) = \DN_-(\eta)$, and so \eqref{second derivative A formula} must agree with the second variation formula \eqref{second derivative DN formula}.  Indeed, one can verify directly that $\theta_- = \psi$, so that
\[  \mathscr{L}_-(u) = -\DN_-(\eta)^{-1} \partial_x a_1^-(u) + a_2^-(u), \qquad a_4(u) = a_4^-(u),\]
and hence
\[ \mathscr{N}(u)  = a_1^-(u) \partial_x \mathscr{L}_-(u)  + a_2^-(u)\DN_-(\eta) \mathscr{L}_-(u) = \mathscr{M}(u),\]
giving back the one-fluid formula in \cite[Proposition 2.1]{mielke2002energetic}.
\end{remark}

\begin{lemma}[Second derivative of $\augV$] \label{variations augV lemma}
For all $(\eta,\psi_*(\eta)) \in \mathcal{O} \cap \Vspace$ and $\dot \eta \in \Vspace_1$, it holds that
\be \begin{split} 
\Diff^2 \augV(\eta)[\dot \eta, \dot \eta] & = \Diff_\eta^2 \augHam(\eta, \psi_*(\eta))[ \dot\eta, \dot\eta]  - \int_{\mathbb{R}} (\DpsiS -\DpsiT) \dot \eta A(\eta) (\DpsiS - \DpsiT) \dot \eta\,\diffx  
\end{split} \label{D2augV formula} \ee
where $\DpsiS$ and $\DpsiT$ are defined in \eqref{compute Dpsistar}.
\end{lemma}
\begin{proof}
Starting from the definition of $\augV$ in \eqref{definition augV}, we see that 
\[ \Diff\augV(\eta) \dot\eta = \Diff_\eta \augHam(u_*) \dot \eta + \Diff_\psi \augHam(u_*) \Diff \psi_*(\eta) \dot \eta = \Diff_\eta \augHam(u_*) \dot\eta,\]
where $u_* = u_*(\eta) :=  (\eta, \psi_*(\eta))$.  Note that the last equality follows from the fact that $u_*$ is a critical point of $\augHam$ for all $\eta$.  Differentiating again in $\eta$ gives
\begin{align*}
\Diff^2\augV(\eta)[\dot \eta, \dot\eta] & = \Diff_\eta^2 \augHam(u_*)[ \dot\eta, \dot \eta] + \Diff_\psi \Diff_\eta \augHam(u_*)[ D \psi_*(\eta) \dot \eta, \dot \eta] \\
& = \Diff_\eta^2 \augHam(u_*)[ \dot\eta, \dot \eta] - \Diff_\psi^2 \augHam(u_*)[ \Diff\psi_*(\eta) \dot \eta, \, \Diff\psi_*(\eta) \dot \eta].
\end{align*}
The potential energy is independent of $\psi$ and the momentum is linear in $\psi$.  Thus, 
\begin{align*}
\Diff_\psi^2 \augHam(u_*)[ \Diff\psi_*(\eta) \dot \eta, \, \Diff\psi_*(\eta) \dot \eta] & = \Diff_\psi^2 \kinE(u_*)[ \Diff\psi_*(\eta) \dot \eta, \, \Diff\psi_*(\eta) \dot \eta] = \int_{\mathbb{R}} \Diff\psi_*(\eta)\dot \eta A(\eta) \Diff\psi_*(\eta)\dot \eta \,\diffx,
\end{align*}
which, from \eqref{compute Dpsistar}, implies \eqref{D2augV formula}.
\end{proof}

\begin{lemma}[Quadratic form]\label{quadratic form lemma}
For all $(\eta,\psi_*(\eta)) \in \mathcal{O} \cap \Vspace$ and $c \in \mathbb{R}$, there is a self-adjoint linear operator $\Qform(\eta) \in \Lin(\Xspace_1; \Xspace_1^*)$ such that
\be  \Diff^2 \augV(\eta)[\dot\eta,\dot\zeta] = \left\langle  \Qform(\eta) \dot\eta,  \, \dot\zeta \right\rangle_{\Xspace_1^* \times \Xspace_1} \label{quadratic form formula} \ee
for all $\dot\eta, \dot\zeta \in \Vspace_1$.  It is given explicitly by 
\be \label{def Q}
\begin{split}
\Qform(\eta) \dot\eta & = -\left( \sigma \frac{\dot\eta^\prime}{\jbracket{\eta^\prime}^3} \right)^\prime - \Big( g \jump{\rho} + \sum_{\pm} \pm \rho_\pm b_1^\pm (b_2^\pm)^\prime \Big) \dot\eta + \sum_\pm \rho_\pm b_1^\pm \left( \DN_\pm(\eta)^{-1} (b_1^\pm \dot\eta)^\prime \right)^\prime.
\end{split}
\ee
\end{lemma}
\begin{remark}\label{quadratic form remark}
Taking $\rho_+ = 0$ and $\rho_- = 1$ recovers the one-fluid problem, and it is straightforward to see that formula \eqref{def Q} agrees with computation in \cite[Theorem 3.3]{mielke2002energetic}.
\end{remark}

\begin{proof}
We continue to write $u_* := (\eta, \psi_*(\eta))$.  Since the momentum is linear in $\eta$, we see that 
\be \label{D2 augHam prelim}
\begin{split}
\Diff_\eta^2 \augHam(u_*)[\dot\eta,\dot\eta] & = \Diff_\eta^2 \kinE(u_*)[\dot\eta,\dot\eta] + \Diff_\eta^2 \potE(u_*)[\dot\eta,\dot\eta] \\
& = \frac12 \int_{\mathbb{R}}  \psi_* \langle \Diff^2A(\eta)[\dot\eta,\dot\eta],\, \psi_*\rangle \,\diffx - \int_{\mathbb{R}} g \jump{\rho} \dot\eta^2 \,\diffx +  \int_{\mathbb{R}}  \sigma\frac{(\dot\eta^\prime)^2}{\jbracket{\eta^\prime}^3}  \,\diffx.
\end{split}
\ee
The latter two terms on the right-hand side above are already in the desired form.  But, to understand the first requires the formula for the second variation of $A(\eta)$ derived in Lemma~\ref{second variation A lemma}.

In particular, notice that when $\theta_\pm$ defined in \eqref{def theta and a4} is evaluated at $u_*$, it simplifies to
\[ \theta_\pm(u_*) = - c \DN_\pm(\eta)^{-1} \eta^\prime = \mp \varphi_\pm,\]
and $a_1^\pm(\eta, \theta_\pm) = b_1^\pm + c$, $a_2^\pm(\eta, \theta_\pm) = \pm b_2^\pm$. We further define
\[
 \DpsiS_\pm(\eta) \xi  := \DN_\pm(\eta)^{-1} \parn{b_1^\pm \xi}^\prime, \quad  \DpsiT_\pm(\eta) \xi := \pm b_2^\pm \xi,
\]
so that $\DpsiS(\eta)  = \sum_{\pm} \rho_\pm \DpsiS_\pm(\eta)$ and $\DpsiT(\eta)  = \sum_{\pm} \rho_\pm \DpsiT_\pm(\eta)$.
Making these substitution, we find from the second derivative formula \eqref{second derivative A formula} that
\begin{align*}
\frac{1}{2} \int_{\mathbb{R}} \psi_* \langle \Diff^2 A(\eta)[\dot\eta,\dot\eta], \, \psi_* \rangle \, dx & = \sum_{\pm} \rho_\pm \int_{\mathbb{R}} \left(\mp(b_1^\pm)^\prime b_2^\pm \dot\eta^2 + \mathcal{T}_\pm \dot\eta \DN_\pm(\eta) \mathcal{T}_\pm \dot \eta \right) \, \diffx \\
& \qquad + \int_{\mathbb{R}} \left( - \dot\eta \mathscr{M}(u_*) \dot\eta + \dot\eta  \mathscr{N}(u_*) \dot\eta \right)  \, \diffx.
\end{align*}
Let us next look more closely at the two terms on the second line above.  Observe first that  
\begin{equation}\label{L formula}
\begin{split} 
\mathscr{L}_\pm(u_*) \dot\eta & = -\DN_\pm(\eta)^{-1} \left( (b_1^\pm + c) \dot\eta \right)^\prime \pm b_2^\pm \dot\eta  = -\left(\DpsiT_\pm  -\DpsiS_\pm + c \DN_\pm(\eta)^{-1} \partial_x \right) \dot\eta \\
\mathscr{L}(u_*) & =  \mathcal{T} - \mathcal{S} - c A(\eta)^{-1} \partial_x,
\end{split}
\end{equation}
where the second line follows from the first and \eqref{A inverse formula}.  Because $\mathscr{L}_\pm$ and $\mathscr{L}$ will be evaluated at $u_*$ throughout the calculation, we will suppress their arguments in the interests of readability.  Using \eqref{L formula}, we see that the operator $\mathscr{M}$ defined in \eqref{def script M} at the critical point satisfies
\begin{equation*}
\begin{split} 
\int_{\mathbb R} \dot\eta \mathscr{M} \dot\eta \,\diffx & = \sum_\pm \rho_\pm \int_{\mathbb R}   \Big( \parn{b_1^\pm + c} (\mathscr{L}_\pm \dot\eta)^\prime \pm b_2^\pm \DN_\pm(\eta) \mathscr{L}_\pm \dot\eta \Big) \dot\eta \,\diffx \\
& = \sum_\pm \rho_\pm \int_{\mathbb R}  \left( -\left( \parn{b_1^\pm + c} \dot\eta \right)^\prime \pm \DN_\pm(\eta) b_2^\pm \dot\eta \right) \mathscr{L}_\pm \dot\eta \,\diffx \\
& = \sum_\pm \rho_\pm \int_{\mathbb R}  \mathscr{L}_\pm \dot\eta \DN_\pm(\eta)\mathscr{L}_\pm \dot\eta \,\diffx, \\
\end{split}
\end{equation*}
where again we are abbreviating $\mathscr{M} = \mathscr{M}(u_*)$.  Substituting in the expression \eqref{L formula} and expanding yields
\begin{equation*}
\begin{split} 
\int_{\mathbb R} \dot\eta \mathscr{M} \dot\eta \,\diffx & =
 \sum_\pm \rho_\pm \int_{\mathbb R} \Big( \DpsiS_\pm \dot\eta \DN_\pm(\eta) \DpsiS_\pm \dot\eta  - 2 \DpsiS_\pm \dot\eta \DN_\pm(\eta) \DpsiT_\pm \dot\eta +  \DpsiT_\pm \dot\eta \DN_\pm(\eta) \DpsiT_\pm \dot\eta \Big) \,\diffx \\
& \qquad + \int_{\mathbb R} \Big( c^2 \dot\eta^\prime A(\eta)^{-1} \dot\eta^\prime  + 2c \dot\eta^\prime (\DpsiS - \DpsiT) \dot\eta \Big) \,\diffx.
\end{split}
\end{equation*}
For later use, we compute
\begin{equation*}
\begin{split} 
\int_{\mathbb R} \DpsiS_\pm \dot\eta \DN_\pm(\eta) \DpsiT_\pm \dot\eta \,\diffx & = \pm \int_{\mathbb R} \DN_\pm(\eta)^{-1} \big( b_1^\pm \dot\eta \big)^\prime \DN_\pm (\eta) (b_2^\pm \dot\eta) \,\diffx
= \pm \int_{\mathbb R}\big( b_1^\pm \dot\eta \big)^\prime (b_2^\pm \dot\eta) \,\diffx \\
& = \pm \frac12 \int_{\mathbb R} \Big( (b_1^\pm)^\prime b_2^\pm - b_1^\pm (b_2^\pm)^\prime \Big) \dot\eta^2 \,\diffx.
\end{split}
\end{equation*}

Finally, in view of \eqref{def script N} and the formula for $\mathscr{L}$ in \eqref{L formula}, we have that $\mathscr{N} = \mathscr{N}(u_*)$ satisfies
\begin{equation*}
\begin{split}
\int_{\mathbb R} \dot\eta \mathscr{N}\dot\eta \,\diffx & = \sum_{\pm} \rho_\pm \int_{\mathbb R}   \Big( (b_1^\pm + c) \parn{A(\eta) \DN_\pm(\eta)^{-1} \mathscr{L} \dot\eta}^\prime \pm b_2^\pm A(\eta) \mathscr{L} \dot\eta \Big) \dot\eta \,\diffx \\
& = \sum_{\pm} \rho_\pm \int_{\mathbb R}   \Big( (b_1^\pm + c) (A(\eta) \DN_\pm(\eta)^{-1} \big(\DpsiT - \DpsiS - c A(\eta)^{-1} \partial_x) \dot\eta \big)^\prime \\
& \qquad\qquad\qquad  \pm b_2^\pm A(\eta) (\DpsiT - \DpsiS - c A(\eta)^{-1} \partial_x) \dot\eta \Big) \dot \eta \,\diffx.
\end{split}
\end{equation*}
Recalling that $A(\eta)$ and $\DN_\pm(\eta)^{-1}$ commute, continuing to simplify the right-hand side we obtain
\begin{equation*}
\begin{split}
\int_{\mathbb R} \dot\eta \mathscr{N} \dot\eta \,\diffx & = - \sum_{\pm} \rho_\pm \int_{\mathbb R}  \DN_\pm(\eta)^{-1} \big( (b_1^\pm + c) \dot\eta \big)^\prime A(\eta) (\DpsiT - \DpsiS - c A(\eta)^{-1} \partial_x) \dot\eta \,\diffx \\
& \qquad + \int_{\mathbb R} \DpsiT \dot\eta A(\eta) (\DpsiT - \DpsiS - c A(\eta)^{-1} \partial_x) \dot\eta \,\diffx \\
& = \int_{\mathbb R} (\DpsiT - \DpsiS - c A(\eta)^{-1} \partial_x) \dot\eta A(\eta) (\DpsiT - \DpsiS - c A(\eta)^{-1} \partial_x) \dot\eta \,\diffx \\
& = \int_{\mathbb R} \Big(\Diff\psi_*(\eta) \dot\eta A(\eta) \Diff\psi_*(\eta) \dot\eta + 2c \dot\eta^\prime (\DpsiS - \DpsiT) \dot\eta + c^2 \dot\eta^\prime A^{-1} \dot\eta^\prime \Big) \,\diffx.
\end{split}
\end{equation*}

Putting the above together and using Lemma \ref{variations augV lemma}, \eqref{D2 augHam prelim}  and Lemma \ref{second variation A lemma} we obtain
\begin{align*}
D^2 \augV(\eta)[\dot \eta, \dot \eta] & = \Diff_\eta^2 \augHam(u_*)[ \dot\eta, \dot\eta] - \int_{\mathbb R} \Diff\psi_*(\eta) \dot\eta A(\eta) \Diff\psi_*(\eta) \dot\eta \,\diffx \\
& = \int_{\mathbb{R}} \left( \frac12 \psi_* \langle\Diff^2 A(\eta)[\dot\eta,\dot\eta],\, \psi_*\rangle - g \jump{\rho} \dot\eta^2 + \sigma\frac{(\dot\eta^\prime)^2}{\jbracket{\eta^\prime}^3} \right)  \,\diffx \\
& \qquad - \int_{\mathbb R} \Diff\psi_*(\eta) \dot\eta A(\eta) \Diff\psi_*(\eta) \dot\eta \,\diffx \\
& = \int_{\mathbb{R}} \left( \sigma\frac{(\dot\eta^\prime)^2}{\jbracket{\eta^\prime}^3} - \Big( g \jump{\rho} + \sum_\pm \pm \rho_\pm b_1^\pm (b_2^\pm)^\prime  \Big) \dot\eta^2 - \sum_\pm \rho_\pm \DpsiS_\pm \dot\eta \DN_\pm(\eta) \DpsiS_\pm \dot\eta \right) \,\diffx,
\end{align*}
which leads to the formula $\Qform(\eta)$ claimed in \eqref{def Q}.
\end{proof}

Following \cite[Theorem 3.5]{mielke2002energetic}, we can determine the continuous spectrum of $\Qform(\eta)$ as follows. 
\begin{lemma}[Continuous spectrum] \label{cont spec lemma}
Let $u = (\eta, \psi) \in \mathcal{O} \cap \Vspace$ be given. Then the operator $\Qform(\eta)$ defined in \eqref{def Q} is self-adjoint on $L^2(\R)$ with domain $H^2(\R)$. The continuous spectrum of $\Qform(\eta)$ is the same as the one of $\Qform(0)$, which is $[\nu_*, +\infty)$, where
\be\label{def nu star}
\nu_* := \left\{\begin{array}{ll} 
-g\jump{\rho} \left(1 - \dfrac{\lambda_0^2}{\lambda^2}\right), & \ \textup{ for } \ \beta \ge \beta_0,\\\\
\displaystyle -g\jump{\rho} \left[ 1 - \frac{1}{\lambda^2} \max_{\xi\in \R}\left( \sum_{\pm} \frac{\rho_\pm}{\rho_-} d_+ \xi \coth{(d_\pm \xi)} - \beta d_+^2 \xi^2 \right) \right], \ & \ \textup{ for }\ \beta < \beta_0.
\end{array}\right.
\ee
\end{lemma}
\begin{proof}
The domain and the self-adjointness of $\Qform(\eta)$ follows from the regularity of $\eta$. The  continuous spectrum of $\Qform(\eta)$ coincides with that of $\Qform(0)$ because $\eta(x) \to 0$ as $|x| \to \infty$. A direct computation yields that the Fourier symbol of $\Qform(0)$ is given by
\[
\symbq_c(\xi) :=  -g\jump{\rho} \left[ 1 - \frac{1}{\lambda^2}\left(  \sum_{\pm} \frac{\rho_\pm}{\rho_-} d_+ \xi \coth{(d_\pm \xi)}   - \beta d_+^2 \xi^2 \right) \right],
\]
which leads to the conclusion of the lemma. 
\end{proof}

\begin{remark}
Observe that the symbol $\symbq_c$ above recovers the dispersion relation in that $d_+\xi$ is a root of \eqref{dispersion relation} if and only if $\symbq_c(\xi) = 0$.
\end{remark}

\subsection{Rescaled operator}\label{sec rescaling}

We now execute the second step in the plan outlined at the start of the section, namely using a long-wave rescaling to discern the leading-order form of the operator $\Qform(\eta)$ in the small-amplitude limit along the families of waves discussed in Section~\ref{traveling wave section}.  Because we wish to exploit the fact that  $(\beta,\lambda)$ is close to the curve $\Gamma_1$ or $\Gamma_2$, it is more convenient to perform these calculations working with the parameterization in \cite{nilsson2017internal}.  With that in mind, let $\{\Pi_\varepsilon\}$ be a smooth curve in the dimensional parameter space.  For Region~A, we assume that the corresponding $\beta > \beta_0$ is fixed and $\lambda = \lambda_0 + \varepsilon^2$, whereas for Region~C, $(\beta,\lambda)$ are given by \eqref{region C lambda beta} with $\delta$ fixed. To avoid cluttered notation, the dependence of $(\rho_\pm,d_\pm,\sigma,c)$ on $\varepsilon$ will be suppressed when there is no risk of confusion.  Recall that the corresponding curves of traveling waves are denoted $\mathscr{C}_{\beta}^{\mathrm{A}}$, $\mathscr{C}_{\beta,\kappa,\pm}^{\mathrm{A}}$, $\mathscr{C}_{\beta,\delta}^{\mathrm{C}}$, and $\mathscr{C}_{\beta,\delta,\kappa,\pm}^{\mathrm{C}}$.

The main character in this analysis is the scaling operator
\[
S_\varepsilon f := f\left( \frac{\varepsilon \placeholder}{d_+} \right).
\]
Clearly $S_\varepsilon$ is a bounded isomorphism on $H^k(\mathbb{R})$ for all $k \geq 0$ with $\| S_\varepsilon\|_{\Lin(H^k)} = O(\varepsilon^{-1})$.  Note that $\partial_x$ and $S_\varepsilon$ satisfy the following commutation identities.
\[
\partial_x S_\varepsilon = \frac{\varepsilon}{d_+} S_\varepsilon \partial_x, \qquad \partial_x S^{-1}_\varepsilon = \frac{d_+}{\varepsilon} S^{-1}_\varepsilon \partial_x,
\]
In particular, this shows that $\partial_x S_\varepsilon$ and $\partial_x S_\varepsilon^{-1}$ are uniformly bounded in $\Lin(H^{k+1},H^k)$ for any $k$. 

From the existence theory in Section~\ref{traveling wave section}, the traveling wave profiles can be written
\begin{equation} \label{definition rescaled profile}
\eta_\varepsilon =:  \varepsilon^m d_+ S_\varepsilon\left( \widetilde\eta + \widetilde\etaerror_\varepsilon\right), \qquad \widetilde{\etaerror}_\varepsilon = O(\varepsilon) \quad \textrm{in } \Wspace_1 \textrm{ as } \varepsilon \searrow 0,
\end{equation}
with 
\[
m := \left\{ \begin{aligned}
2 & \qquad \textrm{for $\mathscr{C}_{\beta}^{\mathrm{A}}$ and $\mathscr{C}_{\beta, \delta,\kappa,\pm}^{\mathrm{C}}$,} \\
1 & \qquad \textrm{for $\mathscr{C}_{\beta,\kappa,\pm}^{\mathrm{A}}$,} \\
4 & \qquad \textrm{for $\mathscr{C}_{\beta,\delta}^{\mathrm{C}}$.}
\end{aligned} \right.
\]
Note that in \eqref{definition rescaled profile} we are continuing the practice of omitting superscripts and subscripts when they can be inferred from context.
Thus, from \eqref{region A KdV scaling} and \eqref{region A Gardner scaling} it follows that in Region~A, $\tilde\eta$ is a scaled KdV or Gardner soliton, while in Region~C it is given by $Z_\delta$ or $Z_{\delta,\kappa,\pm}$.  
 From the commutation identities, we then have that $\eta_\varepsilon^\prime =  \varepsilon^{m+1} S_\varepsilon \left( \widetilde{\eta}^\prime + \widetilde{\etaerror}_\varepsilon^\prime \right).$

Abusing notation somewhat, let $Q_\varepsilon$ be the operator resulting from evaluating $Q_c$ at the parameter values $\Pi_\varepsilon$: 
\begin{equation} \label{definition Q epsilon}
Q_{\varepsilon}(\eta_\varepsilon) := -\partial_x\left( \frac{\sigma}{\jbracket{\eta^\prime_\varepsilon}^3} \partial_x  \right) - \Big( g \jump{\rho} + \sum_{\pm} \pm \rho_\pm b_{1\varepsilon}^\pm (b_{2\varepsilon}^\pm)^\prime \Big)  + \sum_\pm \rho_\pm b_{1\varepsilon}^\pm \partial_x  \DN_\pm(\eta_\varepsilon)^{-1} \partial_x b_{1\varepsilon}^\pm 
\end{equation}
where $b_{i\varepsilon}^\pm = b_i^\pm(\eta_\varepsilon)$ is a multiplication operator and $\eta_\varepsilon$ is from one of the families $\mathscr{C}_\beta^{\mathrm{A}}$, $\mathscr{C}_{\beta,\kappa,\pm}^{\mathrm{A}}$, $\mathscr{C}_{\beta,\delta}^{\mathrm{C}}$, or $\mathscr{C}_{\beta,\delta,\kappa,\pm}^{\mathrm{C}}$. Note that again the dependence of many quantities on $\varepsilon$  is being suppressed.  Our interest is the rescaled operator:
\begin{equation} \label{definition rescaled operator}
\widetilde{Q}_{\varepsilon}(\eta_\varepsilon) := \frac{1}{\varepsilon^n} \frac{d_+}{c^2 \rho_-} S^{-1}_\varepsilon Q_{\varepsilon}(\eta_\varepsilon) S_\varepsilon,
\end{equation}
where $n = 2$ in Region~A and $n= 4$ in Region~C.  Conjugating by $S_\varepsilon$ imposes a long-wave scaling that will, in the limit $\varepsilon \searrow 0$, converge to the linearized operator for the corresponding dispersive model equation.  We are also non-dimensionalizing the problem in order to simplify the resulting expressions.

\begin{lemma}[Expansion of $\widetilde{Q}_\varepsilon$] \label{R asymptotics lemma} The operator $\widetilde{Q}_\varepsilon$ defined in \eqref{definition rescaled operator} admits the expansion 
\[
\widetilde{Q}_{\varepsilon}(\eta_\varepsilon) = \widetilde{Q}_{\varepsilon}(0) + \widetilde{R}_\varepsilon,
\]
where in Region~A
\be \label{expansion R region A}
\widetilde{R}_\varepsilon = \left\{
\begin{aligned}
& -3\left(\varrho-\frac{1}{h^2}\right)\widetilde{\eta} + O(\varepsilon^2) & \qquad &  \textrm{for } \mathscr{C}_\beta^{\mathrm{A}}  \\
& -3\kappa \widetilde\eta - 6\left(\varrho+\frac{1}{h^3}\right) \widetilde{\eta}^2 + O(\varepsilon) & \qquad & \textrm{for } \mathscr{C}_{\beta,\kappa,\pm}^{\mathrm{A}},
\end{aligned}
\right.
\ee
in $\Lin(H^{k+2},H^k)$, and in Region~C
\be \label{expansion R region C}
\widetilde{R}_\varepsilon = \left\{
\begin{aligned}
& -3\left(\varrho-\frac{1}{h^2}\right)\widetilde\eta + O(\varepsilon^2) & \qquad &  \textrm{for } \mathscr{C}_{\beta,\delta}^{\mathrm{C}}  \\
& -3\kappa \widetilde\eta - 6\left( \varrho+\frac{1}{h^3} \right) \widetilde\eta^2 + (1-\varrho) \left( \partial_x (\widetilde\eta \partial_x) + \widetilde\eta^{\prime\prime}\right) + O(\varepsilon^2) & \qquad & \textrm{for } \mathscr{C}_{\beta,\delta,\kappa,\pm}^{\mathrm{C}},
\end{aligned}
\right.
\ee
in $\Lin(H^{k+2},H^k)$.  
\end{lemma}

\begin{proof}
Looking at its definition in \eqref{definition Q epsilon}, we see that $Q_\varepsilon(\eta_\varepsilon)$ is the sum of a second-order differential operator (call it the surface tension term), a multiplication operator (the potential term), and a first-order nonlocal operator (the nonlocal term).  Rescaling the surface tension term yields 
\be \label{surface tension term rescaled}
- \frac{1}{\varepsilon^n} \frac{d_+}{c^2 \rho_-} S^{-1}_\varepsilon \partial_x\left( \frac{\sigma}{\jbracket{\eta^\prime_\varepsilon}^3} \partial_x \right) S_\varepsilon  = - \varepsilon^{2 - n} \partial_x\left( \frac{\beta}{\jbracket{ \varepsilon^{m + 1}(\widetilde\eta^\prime + \widetilde\etaerror_\varepsilon^\prime)}^3} \partial_x \right). 
\ee
To understand the contribution of the potential term to $\widetilde{Q}_{\varepsilon}(\eta_\varepsilon)$, we first denote the non-dimensionalized and rescaled relative velocity field
\be\label{def b tilde}
b_{1\varepsilon}^\pm =: c S_\varepsilon \widetilde{b}_1^\pm, \qquad b_{2\varepsilon}^\pm =: c S_\varepsilon \widetilde{b}_2^\pm.
\ee
From the kinematic boundary condition \eqref{kinematic for b} we then have that $\widetilde{b}_2^\pm = \varepsilon^{k+1} \widetilde{\eta}^\prime \widetilde b_1^\pm$.  Hence
\begin{align*}
-\frac{1}{\varepsilon^n} \frac{d_+}{c^2 \rho_-} S^{-1}_\varepsilon \left( g \jump{\rho} + \sum_{\pm} \pm \rho_\pm b_{1\varepsilon}^\pm (b_{2\varepsilon}^\pm)^\prime \right) S_\varepsilon = \frac{\lambda}{\varepsilon^n} - {\varepsilon^{m-n+2}} \sum_\pm \pm \frac{\rho_\pm}{\rho_-}  \widetilde{b}_1^\pm (\widetilde\eta^\prime \widetilde{b}_1^\pm)^\prime.
\end{align*}
The rescaling of the nonlocal term in $Q_\varepsilon(\eta_\varepsilon)$ will require the most effort to expand.  Towards that end, we define the operator $\NLop(\eta_\varepsilon) \in \Lin(H^{k+2},H^{k+1})$ by 
\be\label{def NLop}
\NLop(\eta_\varepsilon) := \frac{d_+}{\varepsilon^n} S_\varepsilon^{-1} \partial_x \DN_\pm(\eta_\varepsilon)^{-1} \partial_x S_\varepsilon.
\ee
In particular, this means that
\begin{equation}\label{def Q0}
\widetilde{Q}_{\varepsilon}(0) = \frac{1}{\varepsilon^n} \left(-\varepsilon^2 \beta \partial_x^2 + \lambda + \sum_\pm \frac{\rho_\pm}{\rho_-} \varepsilon^n \NLop(0) \right).
\end{equation}

Now, using the above calculations, we will analyze the difference operator
\begin{equation}\label{diff op}
\begin{split}
\widetilde{R}_\varepsilon & := \widetilde{Q}_{\varepsilon}(\eta_\varepsilon) - \widetilde{Q}_{\varepsilon}(0) \\
& = -\beta \varepsilon^{2 - n} \partial_x\left[ \left( \frac{1}{\jbracket{ \varepsilon^{m + 1}( \widetilde\eta^\prime+ \widetilde\etaerror_\varepsilon^\prime)}^3} - 1 \right)\partial_x \right] - {\varepsilon^{m -n+2}} \sum_\pm \pm \frac{\rho_\pm}{\rho_-}  \widetilde{b}_1^\pm (\tilde\eta^\prime \widetilde{b}_1^\pm)^\prime \\
& \qquad + \sum_\pm \frac{\rho_\pm}{\rho_-} \left( \widetilde{b}_1^\pm \NLop(\eta_\varepsilon) \widetilde{b}_1^\pm - \NLop(0) \right).
\end{split}
\end{equation}
In view of \eqref{definition rescaled profile} and \eqref{surface tension term rescaled}, the first term on the right-hand side above is higher order: 
\be\label{1st diff op}
-\beta \varepsilon^{2 - n} \partial_x\left[ \left( \frac{1}{\jbracket{ \varepsilon^{m + 1} (\widetilde\eta^\prime+ \widetilde\etaerror_\varepsilon^\prime)}^3} - 1 \right)\partial_x \right] = O(\varepsilon^{2m - n + 4}) \qquad \textrm{in } \Lin(H^{k+2},H^k).
\ee
Consider the remaining two terms in \eqref{diff op}.  Notice that for any $f \in H^{k+2}$ we have
\begin{equation*}
\begin{split}
\mathcal{F} \left(\NLop(0) f \right) (\xi) & = \frac{d_+}{\varepsilon^n} \frac{\varepsilon}{d_+} \mathcal{F} \left( \partial_x \DN_\pm(0)^{-1} \partial_x S_\varepsilon f \right) \left( \frac{\varepsilon}{d_+} \xi \right) 
 = \frac{d_+}{\varepsilon^n} \symbm_\pm \left( \frac{\varepsilon}{d_+} \xi \right) \widehat{f}(\xi)
\end{split}
\end{equation*}
where $\symbm_\pm(\xi) := - \xi \coth(d_\pm \xi)$
is the symbol for $\partial_x \DN_\pm(0)^{-1} \partial_x$.  Thus $\NLop(0)$ is indeed a Fourier multiplier and its symbol is given by
\be\label{symbol K0}
\widetilde{\symbm}_\varepsilon^\pm(\xi)  := - \frac{1}{\varepsilon^n} \frac{\varepsilon \xi}{\tanh(d_\pm \varepsilon \xi/ d_+)}.
\ee
As an immediate consequence, it follows that 
\begin{equation}\label{est tilde K(0)}
\begin{split}
\left\| \varepsilon^n \NLop(0) + \frac{d_+}{d_\pm} \right\|_{\Lin(H^{k+2}, H^k)} &\leq \left\|  \frac{1}{\jbracket{\placeholder}^{2}} \left( \varepsilon^n \widetilde{\symbm}_\varepsilon^\pm + \frac{d_+}{d_\pm} \right) \right\|_{L^\infty} \lesssim \varepsilon^2.
\end{split}
\end{equation}
In other words, $\varepsilon^n \NLop(0)$ is to leading order the multiplication operator $-d_+/d_\pm$ in $\Lin(H^{k+2},H^k)$.

To estimate the scaled relative velocity, we observe that by \eqref{relation psistar phi}--\eqref{def b} and \eqref{def b tilde}, it holds that
\begin{align*}
\widetilde{b}_1^\pm & = \frac1c S_\varepsilon^{-1} \left( \partial_x \Phi_{\varepsilon \pm}|_{\mathscr{S}} - c \right),
\end{align*}
where, as usual, $\Phi_{\varepsilon\pm}$ denotes the velocity potential.  But expanding the Dirichlet--Neumann operator, we find that
\[
\begin{aligned}
\varphi_\pm^\prime & = \pm c \partial_x \left( \DN_\pm(\eta_\varepsilon)^{-1} \partial_x \eta_\varepsilon \right) \\
& = \pm c \partial_x \left[ \DN_\pm(0)^{-1} \eta_\varepsilon^\prime + \left\langle \Diff \DN_\pm(0)^{-1} \eta_\varepsilon, \eta_\varepsilon^\prime \right\rangle \right] + O(\varepsilon^{3m}) &\qquad& \textrm{in } H^{k},\\
(\partial_x \Phi_{\varepsilon \pm})|_{\mathscr{S}} & = \frac{1}{1 + (\eta_\varepsilon^\prime)^2}\left( \varphi_\pm^\prime \pm \eta_\varepsilon^\prime \DN_\pm(\eta_\varepsilon)\varphi_\pm \right)  = \frac{1}{1 + (\eta_\varepsilon^\prime)^2}\left( \varphi_\pm^\prime \pm (\eta_\varepsilon^\prime)^2 \right) \\
& = \pm c \partial_x \left[ \DN_\pm(0)^{-1} \eta_\varepsilon^\prime + \left\langle \Diff \DN_\pm(0)^{-1} \eta_\varepsilon, \eta_\varepsilon^\prime \right\rangle \right] + O(\varepsilon^{3m}) &\qquad &\textrm{in } H^{k}.
\end{aligned}
\]
We can compute $\Diff \DN_\pm(0)^{-1}$ as
\[
\left\langle \Diff \DN_\pm(0)^{-1} \eta_\varepsilon, f \right\rangle = - \DN_\pm(0)^{-1} \left\langle \Diff\DN_\pm(0) \eta_\varepsilon, \DN_\pm(0)^{-1} f \right\rangle,
\]
and from Lemma~\ref{DG formula lemma}, we see that 
\begin{equation}\label{DG compute}
\begin{split}
\left\langle \Diff\DN_\pm(0) \eta_\varepsilon, \DN_\pm(0)^{-1} \partial_x S_\varepsilon f \right\rangle & = \pm \partial_x S_\varepsilon \left[ \left( S_\varepsilon^{-1} \partial_x \DN_\pm(0)^{-1} \partial_x S_\varepsilon f \right) d_+ \varepsilon^m \widetilde\eta \right]   \\
& \qquad \pm \DN_\pm(0) S_\varepsilon \left( \frac{\varepsilon}{d_+}(\partial_x f) d_+ \varepsilon^m \widetilde \eta \right) \\
& = \pm \partial_x S_\varepsilon \varepsilon^{m+n} \left( \NLop(0) f \right) \widetilde\eta \pm \varepsilon^{m+1} \DN_\pm(0) S_\varepsilon(\widetilde\eta \partial_x f).
\end{split}
\end{equation}
Therefore
\be \label{expansion b_1 tilde}
\begin{split}
\widetilde{b}_1^\pm & = S_\varepsilon^{-1} \big[ \pm \partial_x \left( \DN_\pm(0)^{-1} \eta_\varepsilon^\prime + \left\langle \Diff \DN_\pm(0)^{-1} \eta_\varepsilon, \eta_\varepsilon^\prime \right\rangle \right) - 1 \big] + O(\varepsilon^{3m}) \\
& = \pm \varepsilon^m d_+ \left[ S_\varepsilon^{-1} \partial_x \DN_\pm(0)^{-1} \partial_x (S_\varepsilon \widetilde\eta) - S_\varepsilon^{-1} \partial_x \DN_\pm(0)^{-1} \left\langle \Diff\DN_\pm(0) \eta_\varepsilon, \DN_\pm(0)^{-1} \partial_x S_\varepsilon \widetilde\eta \right\rangle \right] \\
& \qquad - 1 + O(\varepsilon^{3m}) \\
& = \pm \varepsilon^{m + n} \NLop(0) \widetilde\eta - 1+ O(\varepsilon^{3m}) \\
& \qquad \mp \varepsilon^m d_+ S_\varepsilon^{-1} \partial_x \DN_\pm(0)^{-1} \left[ \pm \varepsilon^{m+n} \partial_x S_\varepsilon \left( \NLop(0) \widetilde\eta \right) \widetilde\eta \pm \varepsilon^{m+1} \DN_\pm(0) S_\varepsilon(\widetilde\eta \widetilde\eta^\prime) \right] \\
& = \pm \varepsilon^{m + n} \NLop(0) \widetilde\eta - 1 - \varepsilon^{2m+2n} \NLop(0) \left( \left( \NLop(0) \widetilde\eta \right) \widetilde\eta \right) - \varepsilon^{2m+2} \partial_x(\widetilde\eta \widetilde\eta^\prime) + O(\varepsilon^{3m}) \\
& = -1 \mp \varepsilon^m \frac{d_+}{d_\pm} \widetilde\eta - \varepsilon^{2m} \frac{d_+^2}{d_\pm^2} \widetilde\eta^2 + O(\varepsilon^{m+2}) \qquad \textrm{in } H^k.
\end{split}
\ee
Hence for the second term on the right-hand side of \eqref{diff op} we have
\begin{equation}\label{2nd diff op}
\begin{split}
- {\varepsilon^{m -n+2}} \sum_\pm \pm \frac{\rho_\pm}{\rho_-}  \widetilde{b}_1^\pm (\widetilde\eta^\prime \widetilde{b}_1^\pm)^\prime & = {\varepsilon^{m -n+2}}  (1 - \varrho) \widetilde\eta^{\prime\prime} + \varepsilon^{2m -n+2} \left( \varrho + \frac{1}{h} \right) \left[ 2 \widetilde\eta \widetilde\eta^{\prime\prime} + (\widetilde\eta^\prime)^2 \right] \\
& \quad + O(\varepsilon^{3m - n +2}) \qquad \qquad \textrm{in } \Lin{(H^{k+2},H^k)}.
\end{split}
\end{equation}

Using the expansion \eqref{expansion b_1 tilde} for $\tilde{b}_1^\pm$ also furnishes the estimate
\begin{equation} \label{bMb estimate}
\begin{split}
\widetilde{b}_1^\pm \NLop(\eta_\varepsilon) \widetilde{b}_1^\pm &= \NLop(\eta_\varepsilon) \pm \varepsilon^m \frac{d_+}{d_\pm} \left[ \widetilde\eta \NLop(\eta_\varepsilon) + \NLop(\eta_\varepsilon) \widetilde\eta \right]   \\
& \qquad  + \varepsilon^{2m} \frac{d_+^2}{d_\pm^2} \left[ \widetilde\eta \NLop(\eta_\varepsilon) \widetilde\eta + \widetilde\eta^2 \NLop(\eta_\varepsilon) + \NLop(\eta_\varepsilon) \widetilde\eta^2 \right]  + O(\varepsilon^{3m - n})
\end{split}
\end{equation}
in $\Lin(H^{k+2},H^k)$.
On the other hand, from the definition of $\NLop$ in \eqref{def NLop} it follows that for all $f \in H^{k+2}$ with $\| f \|_{H^{k+2}} = 1$, 
\begin{equation} \label{difference of Ms} 
\begin{split}
\left( \NLop(\eta_\varepsilon) - \NLop(0) \right) f & = \frac{d_+}{\varepsilon^n}  S_\varepsilon^{-1} \partial_x \left( \DN_\pm(\eta_\varepsilon)^{-1} - \DN_\pm(0)^{-1} \right) \partial_x S_\varepsilon f \\
& = \frac{d_+}{\varepsilon^n}  S_\varepsilon^{-1} \partial_x  \left\langle \Diff \DN_\pm(0)^{-1} \eta_\varepsilon, \partial_x S_\varepsilon f \right\rangle  \\
& \qquad + \frac{d_+}{2\varepsilon^n}  S_\varepsilon^{-1} \partial_x  \left\langle \Diff^2 \DN_\pm(0)^{-1} [\eta_\varepsilon, \eta_\varepsilon], \partial_x S_\varepsilon f \right\rangle  + O(\varepsilon^{3m - n}) \quad \textrm{in } H^k.
\end{split}
\end{equation}
Explicit calculation yields
\begin{align*}
\left\langle \Diff^2 \DN_\pm(0)^{-1} [\eta_\varepsilon, \eta_\varepsilon], f \right\rangle & = - \DN_\pm(0)^{-1} \left\langle \Diff^2\DN_\pm(0) [\eta_\varepsilon, \eta_\varepsilon], \DN_\pm(0)^{-1} f \right\rangle \\
& \qquad + 2 \DN_\pm(0)^{-1} \left\langle \Diff\DN_\pm(0) \eta_\varepsilon, \DN_\pm(0)^{-1} \left\langle \Diff\DN_\pm(0) \eta_\varepsilon, \DN_\pm(0)^{-1} f \right\rangle \right\rangle.
\end{align*}
From \eqref{DG compute} we have 
\begin{equation*}
\begin{split}
\left\langle \Diff\DN_\pm(0) \eta_\varepsilon, \right. & \left. \DN_\pm(0)^{-1} \left\langle \Diff\DN_\pm(0) \eta_\varepsilon, \DN_\pm(0)^{-1} \partial_x S_\varepsilon f \right\rangle \right\rangle \\
& = \partial_x S_\varepsilon \left[ \left( S_\varepsilon^{-1} \partial_x \DN_\pm(0)^{-1} \partial_x S_\varepsilon \varepsilon^{m+n} \left( \NLop(0) f \right) \widetilde\eta \right) d_+ \varepsilon^m \widetilde\eta \right] + O(\varepsilon^{2m + 1}) \\
& = \varepsilon^{2(m + n)} \partial_x S_\varepsilon \NLop(0) \left( \left( \NLop(0) f \right) \widetilde\eta \right) \widetilde\eta + O(\varepsilon^{2m + 1}) \qquad \textrm{in } H^k.
\end{split}
\end{equation*}
Likewise, Lemma \ref{D^2G formula lemma} allows us to estimate
\[
\left\langle \Diff^2\DN_\pm(0) [\eta_\varepsilon, \eta_\varepsilon],\, \DN_\pm(0)^{-1} f \right\rangle = O(\varepsilon^{2m + 1}) \qquad \textrm{in } H^k.
\]
Substituting the above into \eqref{difference of Ms} yields
\begin{equation}\label{difference NLop}
\begin{split}
\left( \NLop(\eta_\varepsilon) - \NLop(0) \right) f & = \mp \frac{d_+}{\varepsilon^n} S_\varepsilon^{-1} \partial_x \DN_\pm(0)^{-1} \partial_x S_\varepsilon \left( \varepsilon^{m+n} \NLop(0) f \right) \widetilde\eta \\
& \qquad + \frac{d_+}{\varepsilon^n} S_\varepsilon^{-1} \partial_x \DN_\pm(0)^{-1} \partial_x S_\varepsilon \left( \varepsilon^{2(m + n)} \NLop(0) \left( \NLop(0) f \right) \widetilde\eta \right) \widetilde\eta \\
& \qquad \mp \frac{d_+}{\varepsilon^n} S_\varepsilon^{-1} \partial_x S_\varepsilon \left( \varepsilon^{m+1} \widetilde\eta \partial_x f \right)  + O(\varepsilon^{2m-n + 1}) \\
& = \mp \varepsilon^{m+n} \NLop(0) \left( \left( \NLop(0) f \right) \widetilde\eta \right) \mp \varepsilon^{m-n + 2} \partial_x \left( \widetilde\eta \partial_x f \right) \\
& \qquad + \varepsilon^{2(m+n)} \NLop(0) \left( \NLop(0) \left( \NLop(0) f \right) \widetilde\eta \right) \widetilde\eta + O(\varepsilon^{2m-n+1}) \\ 
& = \mp \varepsilon^{m - n} \frac{d_+^2}{d_\pm^2} \widetilde\eta f \mp \varepsilon^{m-n+ 2} \partial_x \left( \widetilde\eta \partial_x f \right) - \varepsilon^{2m - n} \frac{d_+^3}{d_\pm^3} \widetilde\eta^2 f + O(\varepsilon^{2m-n+1}), 
\end{split}
\end{equation}
in $H^{k}$.  Using this, the previous estimate  \eqref{bMb estimate} becomes
\begin{equation*}
\begin{split}
\widetilde{b}_1^\pm \NLop(\eta_\varepsilon) \widetilde{b}_1^\pm f & = \NLop(0) f \pm \varepsilon^m \frac{d_+}{d_\pm} \widetilde\eta \NLop(0) f \pm \varepsilon^m \frac{d_+}{d_\pm} \NLop(0) \widetilde\eta f \\
& \qquad + \varepsilon^{2m} \frac{d_+^2}{d_\pm^2} \left[ \widetilde\eta \NLop(0) \widetilde\eta f + \widetilde\eta^2 \NLop(0) f + \NLop(0)\widetilde\eta^2 f  \right] \\
& \qquad \mp \varepsilon^{m-n} \frac{d_+^2}{d_\pm^2} \widetilde\eta f \mp \varepsilon^{m-n + 2} \partial_x \left( \widetilde\eta \partial_x f \right) - 3 \varepsilon^{2m-n} \frac{d_+^3}{d_\pm^3} \widetilde\eta^2 f  \\
& \qquad - \varepsilon^{2m-n + 2} \frac{d_+}{d_\pm} \left[\widetilde\eta \partial_x(\widetilde\eta \partial_x f) + \partial_x (\widetilde\eta \partial_x(\widetilde\eta f)) \right] + O(\varepsilon^{2m - n + 1}),
\end{split}
\end{equation*}
in $H^k$.  We can simplify further by applying \eqref{est tilde K(0)}, which results in
\begin{equation*}
\begin{split}
\widetilde{b}_1^\pm \NLop(\eta_\varepsilon) \widetilde{b}_1^\pm f &= \NLop(0) f \mp 3 \varepsilon^{m - n} \frac{d_+^2}{d_\pm^2} \widetilde\eta f - 6 \varepsilon^{2m - n} \frac{d_+^3}{d_\pm^3} \widetilde\eta^2 f \mp \varepsilon^{m-n + 2} \partial_x \left( \widetilde\eta \partial_x f \right)  \\
& \qquad + O(\varepsilon^{2m - n + 1}) \qquad \textrm{in } H^k.
\end{split}
\end{equation*}

Therefore in computing the third term on the right-hand side of \eqref{diff op} we find
\begin{equation*}
\begin{split}
\sum_\pm \frac{\rho_\pm}{\rho_-} \left( \widetilde{b}_1^\pm \NLop(\eta_\varepsilon) \widetilde{b}_1^\pm -  \NLop(0) \right) f & = 3\varepsilon^{m-n} \sum_\pm \frac{\mp \rho_\pm}{\rho_-} \frac{d_+^2}{d_\pm^2} \widetilde\eta f - 6 \varepsilon^{2m - n} \sum_\pm \frac{\rho_\pm}{\rho_-} \frac{d_+^3}{d_\pm^3} \widetilde\eta^2 f \\
& \qquad + \varepsilon^{m-n + 2} \sum_\pm \frac{\mp \rho_\pm}{\rho_-} \partial_x \left( \widetilde\eta \partial_x f \right) + O(\varepsilon^{2m - n + 1}) \\
& = - 3\varepsilon^{m-n} \left( \varrho - \frac{1}{h^2} \right) \widetilde\eta f - 6 \varepsilon^{2m-n} \left( \varrho + \frac{1}{h^3} \right) \widetilde\eta^2 f \\
& \qquad + \varepsilon^{m-n + 2} (1 - \varrho) \partial_x \left( \widetilde\eta \partial_x f \right) + O(\varepsilon^{2m - n + 1}),
\end{split}
\end{equation*}
in $H^k$.  Taken together with \eqref{1st diff op} and \eqref{2nd diff op}, this yields the claimed expansion for $\widetilde{R}_\varepsilon$.
\end{proof}

Let us now look more closely at the leading-order part of $\widetilde{Q}_\varepsilon(\eta_\varepsilon)$, which by the above lemma is the Fourier multiplier $\widetilde{Q}_\varepsilon(0)$.  Analyzing its symbol will allow us to infer that it has a point-wise limit as $\varepsilon \searrow 0$.  Near the critical Bond number, however, there is a degeneracy that causes the limiting operator to be fourth order.  Combining this with the previous result, we obtain the following.

\begin{lemma}[Limiting rescaled operator]\label{lem limit rescaled}
Consider the rescaled operator $\widetilde{Q}_{\varepsilon}(\eta_\varepsilon)$ given by \eqref{definition rescaled operator}.  
\begin{enumerate}[label=\rm(\alph*)]
\item \label{scaled Q in region A} Suppose that $\beta > \beta_0$ and $\lambda = \lambda_0 + \varepsilon^2$ lies in Region~A.  Then for any $k > 1/2$ and $\zeta \in H^{k+2}$, 
\[
\| \widetilde{Q}_\varepsilon(\eta_\varepsilon) \zeta - \widetilde{Q}_0 \zeta \|_{H^k} \longrightarrow 0 \qquad \textrm{as } \varepsilon \searrow 0,
\]
where the operator $\tilde{Q}_0 \in \Lin(H^{k+2},H^k)$ is given by 
\begin{equation*} 
\widetilde{Q}_{0} =
\left\{
\begin{aligned}
& \displaystyle  - (\beta - \beta_0) \partial_x^2 +1 - 3 \left( \varrho - \frac{1}{h^2} \right) \widetilde\eta & \qquad & \textrm{for } \mathscr{C}_\beta^{\mathrm{A}} \\
& \displaystyle  - (\beta - \beta_0) \partial_x^2+1 - 3 \kappa \widetilde\eta - 6 \left( \varrho + \frac{1}{h^3} \right) \widetilde\eta^2 & \qquad & \textrm{for } \mathscr{C}_{\beta;\kappa,\pm}^{\mathrm{A}}.
\end{aligned}
\right.
\end{equation*}
 
\item \label{scaled Q in region C} Suppose that $(\beta,\lambda)$ lie in Region~C and are given by \eqref{region C lambda beta} for a fixed $\delta < 0$.  Then for any $k > 1/2$ and $\zeta \in H^{k+4}$,
\[
\| \widetilde{Q}_\varepsilon(\eta_\varepsilon) \zeta - \widetilde{Q}_0 \zeta \|_{H^k} \longrightarrow 0 \qquad \textrm{as } \varepsilon \searrow 0,
\]
where the operator $\widetilde{Q}_0 \in \Lin(H^{k+4},H^k)$ is given by
\begin{equation*}
\widetilde{Q}_0 = 
\left\{
\begin{aligned}
& \displaystyle \gammaC \partial_x^4 - 2(1 + \delta) \gammaC \partial_x^2 + \gammaC - 3 \left( \varrho - \frac{1}{h^2} \right) \widetilde\eta &  & \textrm{for } \mathscr{C}_{\beta,\delta}^{\mathrm{C}} \\
&
\begin{aligned}
& \gammaC \partial_x^4 - 2(1 + \delta) \gammaC \partial_x^2 + \gammaC - 3 \kappa \widetilde\eta - 6 \left( \varrho + \frac{1}{h^3} \right) \widetilde\eta^2 + (1 - \varrho) \left( \partial_x(\widetilde\eta \partial_x) + \widetilde\eta^{\prime\prime} \right)
\end{aligned}
&  & \textrm{for } \mathscr{C}_{\beta;\kappa,\delta,\pm}^{\mathrm{C}}.
\end{aligned}\right.
\end{equation*}
\end{enumerate}
\end{lemma}

\begin{proof}
Fix $k > 1/2$.  Recall that $\widetilde{Q}_{\varepsilon}(\eta_\varepsilon) = \widetilde{Q}_{\varepsilon}(0) + \widetilde{R}_\varepsilon$, where $\widetilde{Q}_{\varepsilon}(0)$ is given in \eqref{def Q0}. We have already seen in Lemma~\ref{R asymptotics lemma} that $\widetilde{R}_\varepsilon$ has a uniform limit in $\Lin(H^{k+2},H^k)$ as $\varepsilon \searrow 0$.  From \eqref{symbol K0}, it is clear that  $\widetilde{Q}_{\varepsilon}(0)$ is a Fourier multiplier: for all $f \in H^{k+2}$, 
\begin{equation*}
\mathcal{F}\left( \widetilde{Q}_{\varepsilon}(0) f \right) (\xi) = \frac{1}{\varepsilon^n} \left( \varepsilon^{2} \beta \xi^2 + \lambda - \sum_\pm \frac{\rho_\pm}{\rho_-} {\varepsilon \xi}\coth{\left(\frac{d_\pm}{d_+} \varepsilon \xi\right)} \right) \widehat{f}(\xi) =: \widetilde{\symbq}_{\varepsilon}(\xi) \widehat{f}(\xi).
\end{equation*}
Consider the point-wise limit of the symbol $\widetilde{\symbq}_{\varepsilon}$ as $\varepsilon \searrow 0$.  Here it is important to keep in mind that the dimensional parameters are moving along the curve $\{\Pi_\varepsilon\}$ and $\lambda \searrow \lambda_0$ in this limit.  Therefore, we write
\begin{equation*}
\begin{split}
\widetilde{\symbq}_{\varepsilon}(\xi) & = \varepsilon^{2-n} (\beta - \beta_0) \xi^2 + \frac{\lambda - \lambda_0}{\varepsilon^n} + \frac{1}{\varepsilon^n} \left( \beta_0 (\varepsilon \xi)^2 + \lambda_0 - \sum_\pm \frac{\rho_\pm}{\rho_-} \varepsilon \xi \coth{\left(\frac{d_\pm}{d_+} \varepsilon \xi\right)} \right) \\
& =: \varepsilon^{2-n} (\beta - \beta_0) \xi^2 + \frac{\lambda - \lambda_0}{\varepsilon^n} + \frac{\symbr(\varepsilon \xi)}{\varepsilon^n}.
\end{split}
\end{equation*}
Taylor expanding $\symbr$ near $\widetilde \xi := \varepsilon \xi = 0$ yields that
\begin{equation}\label{expansion r}
\begin{split}
\symbr(\widetilde\xi) = \beta_0 \widetilde\xi^2 + \lambda_0 - \sum_\pm \frac{\rho_\pm}{\rho_-}  \widetilde\xi \coth{\left(\frac{d_\pm}{d_+} \widetilde \xi \right)} = \gammaC \widetilde{\xi}^4 + O(\widetilde{\xi}^6) \quad \text{as }\ \widetilde\xi \to 0.
\end{split}
\end{equation}

For Region~A, we have $n = 2$ and $\lambda = \lambda_0 + \varepsilon^2$, and hence for each fixed $\xi \in \mathbb{R}$, 
\[ \widetilde{\symbq}_\varepsilon(\xi) \longrightarrow (\beta-\beta_0) \xi^2 + 1 \qquad \textrm{as } \varepsilon \searrow 0.\]
On the other hand, in Region~C we have $n = 4$ with $(\beta,\lambda)$ given by \eqref{region C lambda beta}.  Again, fixing $\xi$ we then have that the limiting symbol is
\[
\widetilde{\symbq}_\varepsilon(\xi) \longrightarrow \gammaC \xi^4 +  2(1+\delta) \gamma \xi^2 + \gamma \qquad \textrm{as } \varepsilon \searrow 0.
\]
Combining these expressions for the limiting symbol with the asymptotics of $\widetilde{R}_\varepsilon$ from \eqref{expansion R region A} and \eqref{expansion R region C}, the formulas for $\widetilde{Q}_0$ in \ref{scaled Q in region A} and \ref{scaled Q in region C} now follow.  
\end{proof}

\subsection{Spectrum of the linearized augmented potential} \label{spectrum linearized augV section}
Using the limiting behavior derived above, we will now characterize the spectrum of the $Q_\varepsilon(\eta_\varepsilon)$.  It is worth reiterating that an essential challenge in this analysis is that the operator converges point-wise to $Q_0(0)$ whose essential spectrum is $[0, \infty)$.  It is for this reason that we introduced the rescaled operator $\widetilde Q_\varepsilon(\eta_\varepsilon)$, since by Lemma \ref{lem limit rescaled} converges (again only point-wise) to $\widetilde Q_0$, which has a gap between the positive essential spectrum and $0$. 

\subsubsection*{Spectral analysis in Region~A}\label{subsubsec spectral A}
We start by deriving the spectral properties of $Q_\varepsilon(\eta_\varepsilon)$ for the strong surface tension waves with parameters $(\beta, \lambda)$ in Region~A.

\begin{lemma}\label{lem kdvspec}
In the setting of Lemma \ref{lem limit rescaled} \ref{scaled Q in region A}, the limiting rescaled operator $\widetilde{Q}_0$ satisfies
\begin{equation}\label{spec assump}
\essspectrum{\widetilde Q_0} = [1, \infty), \qquad \spectrum{\widetilde Q_0}= \{ -\widetilde \nu^2,\, 0 \} \cup \widetilde \Lambda
\end{equation}
where the first two eigenvalues $-\widetilde \nu^2 < 0$ and $0$ are both simple with corresponding eigenfuctions $\widetilde \phi_1$ and $\widetilde \phi_2 = \widetilde \eta^\prime$, respectively; and there exists $\nu_* > 0$ such that $\widetilde \Lambda \subset [\nu_*, \infty)$.
\end{lemma}
\begin{proof}
This is a classical result on linear Schr\"odinger operators, and can be found, for example, in \cite{AnguloPavabook}. The fact that $-\widetilde \nu^2$ and $0$ are all simple follows from the theory of ODEs: the Wronskian of two $L^2$ solutions to the eigenvalue problem $\widetilde Q_0 f = \widetilde\nu f$ is necessarily $0$. 
\end{proof}

Using a similar argument as \cite[Theorem 4.3]{mielke2002energetic}, we then have the following result. 
\begin{theorem}[Spectrum in Region A] \label{thm Qspectrum}
Let the assumptions of Lemma \ref{lem limit rescaled} \ref{scaled Q in region A} hold. For each $a \in (0, \nu_*)$ there exists some $\varepsilon_0 > 0$ such that for all $\varepsilon \in (0, \varepsilon_0)$ the operator $Q_\varepsilon(\eta_\varepsilon)$ satisfies
\[
\essspectrum{Q_\varepsilon(\eta_\varepsilon)} \subset [\varepsilon^2 c^2 \rho_-/d_+, \infty),  \qquad   \spectrum{Q_\varepsilon(\eta_\varepsilon)} = \{ - \nu^2, \, 0\} \cup \Lambda,
\]
where $\Lambda \subset [a\varepsilon^2 c^2 \rho_-/d_+, \infty)$, and 
\[
\nu^2 = \frac{\varepsilon^2 c^2 \rho_-}{d_+}\widetilde \nu^2 + o(\varepsilon^2) \qquad \textrm{as } \ \varepsilon \searrow 0.
\]
The first two eigenvalues $\nu_1 := -\nu^2 < 0$ and $\nu_2 := 0$ are simple with the associated eigenfunctions taking the form $\phi_i = S_\varepsilon \widetilde \phi_i + o(1)$ in $H^{k}$ as $\varepsilon \searrow 0$. 
\end{theorem}
\begin{proof}
From Lemma \ref{R asymptotics lemma} we see that it suffices to prove that the operator
\[
\mathcal{Q}_\varepsilon := \widetilde Q_\varepsilon(0) + \widetilde R_0,
\]
with $\widetilde{R}_0$ defined by \eqref{expansion R region A} with $\varepsilon = 0$, has exactly two simple eigenvalues lying in $(-\infty, a)$ that converge to $\widetilde \nu_i$ respectively for $i = 1, 2$. It is clear that $\mathcal{Q}_\varepsilon$ is self-adjoint. Note that $\mathcal{Q}_\varepsilon$ may not have 0 as an exact eigenvalue, but this does hold for $\widetilde Q_\varepsilon(\eta_\varepsilon)$.

Firstly, from Lemma \ref{lem limit rescaled} \ref{scaled Q in region A} it follows that
\begin{equation}\label{evalue closeness}
\| \left( \mathcal{Q}_\varepsilon - \widetilde\nu_i  \right) \widetilde \phi_i \|_{H^k} \le C \varepsilon^2 \|\widetilde \phi_i \|_{H^k}.
\end{equation}
Therefore $\mathcal{Q}_\varepsilon$ admits spectral values close to $\widetilde \nu_i$ with $O(\varepsilon^2)$ distance.

Now we consider a sequence $\{ (\nu_{\varepsilon_j}, \phi_{\varepsilon_j}) \}$ of eigenpairs of $\mathcal{Q}_{\varepsilon_j}$ with $\nu_{\varepsilon_j} \in (-\infty, a)$ and $\varepsilon_j \searrow 0$ as $j \to \infty$. Our goal is to prove the compactness of the eigenpair sequence and confirm that the limit must be an eigenpair of $\widetilde Q_0$.

We normalize so that $\| \phi_{\varepsilon_j} \|_{H^k} = 1$. Note that $\|\widetilde\eta\|_{W^{N, \infty}} \le C_N$ for any $N \geq 0$. Moreover from the proof of Lemma \ref{R asymptotics lemma} we see that $\widetilde Q_\varepsilon(0) - 1$ is positive semi-definite. From this we know that the spectrum of $\mathcal{Q}_\varepsilon$ is bounded below: $\spectrum{\mathcal{Q}_\varepsilon} \subset [1 - C_k, \infty)$. Since $\widetilde \eta$ decays exponentially, we have that $\essspectrum{\mathcal{Q}_\varepsilon} = [1, \infty)$. Thus $\spectrum{\mathcal{Q}_\varepsilon} \cap [1 - C_k, a]$ consists of discrete eigenvalues of finite multiplicity. By definition,
\begin{equation}\label{evalue prob}
\left( \widetilde Q_{\varepsilon_j}(0) - \nu_{\varepsilon_j} \right) \phi_{\varepsilon_j} = -\widetilde R_0 \phi_{\varepsilon_j}.
\end{equation}
Since $\nu_{\varepsilon_j} \in [1 - C_k, a]$, from the proof of Lemma \ref{R asymptotics lemma}, the Fourier symbol of operator on the left-hand side is
\[
\widetilde{\symbq}_{\varepsilon_j}(\xi) - \nu_{\varepsilon_j} \ge \  \widetilde{\symbq}_0(\xi) - a = (\beta - \beta_0) \xi^2 + 1 - a \ge  \delta_* (1 + \xi^2)
\]
for some $\delta_* > 0$ independent of $\varepsilon_j$. This uniform ellipticity property allows us via bootstrapping to obtain the bound $\| \phi_{\varepsilon_j} \|_{H^{k+4}} \le C_*$ from some universal constant $C_* > 0$.

To obtain compactness of the sequence $\{\phi_{\varepsilon_j}\}$ in $H^{k+2}$, we proceed to prove a uniform decay estimate. Given an exponential weight $w := \cosh(\alpha \placeholder)$ for some $\alpha > 0$, we see that for any Schwartz function $f$, 
\[
\mathcal{F} \left[ w \left( \widetilde Q_{\varepsilon_j}(0) - \nu_{\varepsilon_j}  \right)^{-1} f \right](\xi) = \frac12 \left[ \frac{\widehat f(\xi + i \alpha)}{\widetilde{\symbq}_{\varepsilon_j}(\xi + i \alpha) - \nu_{\varepsilon_j}} + \frac{\widehat f(\xi - i \alpha)}{\widetilde{\symbq}_{\varepsilon_j}(\xi - i \alpha) - \nu_{\varepsilon_j}} \right].
\]
Taking $\alpha^2 < (1 - a)/(\beta - \beta_0)$ it follows that
\[
\sup_{|\imagpart{\xi}| \le \alpha} \left| \frac{1}{\widetilde{\symbq}_{\varepsilon_j}(\xi \pm i \alpha) - \nu_{\varepsilon_j}} \right| \le C^*,
\]
for some $C^* > 0$. Therefore 
\[
\left\| \left( \widetilde Q_{\varepsilon_j}(0) - \nu_{\varepsilon_j} \right)^{-1} \right\|_{\Lin(L^2_w)} \le C^*,
\]
where $L^2_w := \{ f\in L^2: \ w f \in L^2 \}$ is the weighted $L^2$ space corresponding to $w$. Hence from \eqref{evalue prob},
\[
\| \phi_{\varepsilon_j} \|_{L^2_w} \le C^* \| \widetilde R_0 \phi_{\varepsilon_j} \|_{L^2_w} \le C^* \| w \widetilde R_0 \|_{L^\infty} \| \phi_{\varepsilon_j} \|_{L^2} \le C^* \| w \widetilde R_0 \|_{L^\infty} \lesssim 1.
\]
Thus $\{ \phi_{\varepsilon_j} \}$ is bounded in $H^{k+4} \cap L^2_w$, which is compactly embedded in $H^{k+2}$. Hence up to a subsequence, as $j \to \infty$, $\nu_{\varepsilon_j} \to \nu_* \in (-\infty, a]$ and $\phi_{\varepsilon_j} \to \phi_*$ in $H^{k+2}$ with $\|\phi_*\|_{H^k} = 1$. Moreover, $\widetilde Q_0 \phi_* = \nu_* \phi_*$, which indicates that $\phi_* = \widetilde\phi_i$ for some $i = 1, 2$.

Finally we check the convergence of the corresponding spectral projections. Set $\mathcal P_\varepsilon$ to be the spectral projection for $\mathcal{Q}_\varepsilon$ associated with the interval $[1 - C_k, a]$. From \eqref{evalue closeness}, there exists $\varepsilon_0 > 0$ such that $\text{dim}\Rng{\mathcal P_\varepsilon} \ge 2$ for $\varepsilon \in (0, \varepsilon_0)$. Also $\mathcal P_\varepsilon = \sum^{N_\varepsilon}_{i = 1} \left\langle \placeholder, \phi_{i,\varepsilon} \right\rangle \phi_{i, \varepsilon}$ for some finite integer $N_\varepsilon$ and orthonormal eigenbasis $\{ \phi_{i, \varepsilon} \}_{i=1}^{N_\varepsilon}$. Were there a sequence $\varepsilon_j \searrow 0$ such that $N_{\varepsilon_j} \ge 3$, then it would contradict the above convergence result. Therefore, for all $\varepsilon$ sufficiently small, it must be that $N_\varepsilon = 2$.  We can then conclude that $\phi_{i, \varepsilon} \to \widetilde\phi_i$ in $H^k$.
\end{proof}

\subsubsection*{Spectral analysis in Region~C}\label{subsubsec spectral C}
The same argument can also be applied to the near critical surface tension waves with $(\beta,\lambda)$ in Region~C. On the solution curve $\mathscr{C}_{\beta,\delta}^{\mathrm{C}}$, we have that $\widetilde \eta$ satisfies
\be\label{near critical eqn}
\gamma \partial_x^4 \widetilde \eta - 2(1 + \delta) \gamma \partial_x^2 \widetilde \eta + \gamma \widetilde \eta - \frac32 \left( \varrho - \frac{1}{h^2} \right) \widetilde \eta^2 = 0.
\ee
Direct computation shows that the Green's function of $[\partial_x^4 - 2(1+\delta)\partial_x^2 +1]^{-1}$ decays like $e^{-s|x|}$ as $x \to \pm\infty$, where
\be\label{decay green}
s := \left\{\begin{array}{ll}
\displaystyle \sqrt{1+\delta - \sqrt{\delta(2+\delta)}}, \quad & \delta \ge 0, \\
\sqrt{|\delta| / 2}, \quad & \delta < 0,
\end{array}\right.
\ee
indicating that $\widetilde\eta$ is exponentially localized. Therefore, invoking the Weyl theorem on continuous spectrum, we know that 
\begin{equation}\label{ess nu}
\essspectrum{\widetilde Q_0} = [\gamma, \infty) \quad \text{when } \ \delta > -2.
\end{equation} 
Note that the operator $\widetilde Q_0$ is self-adjoint in $L^2(\R)$ with domain $H^{k+4}(\R)$. Therefore, its spectrum is confined to the real line. Standard ODE theory shows that any eigenvalue of $\widetilde Q_0$ has geometric multiplicity $\le 2$.

By setting $Z := \frac{3}{2\gamma}\left( \varrho - \frac{1}{h^2} \right) \widetilde \eta$, equation \eqref{near critical eqn} becomes \eqref{intro kawahara} 
\[
Z^{\prime\prime\prime\prime} - 2(1+\delta) Z^{\prime\prime} + Z - Z^2 = 0,
\]
which leads us to study
\be \label{linearized operator kawahara}
Q_\delta := \partial_x^4 - 2(1+\delta) \partial_x^2 + 1 - 2 Z_\delta
\ee
viewed as an unbounded operator on $L^2(\mathbb{R})$ with domain $H^4(\mathbb{R})$.  While more exotic than the Schr\"odinger operator encountered in Region~A, the spectral properties of this $Q_\delta$ for $\delta > -2$ have been studied by Sandstede \cite{sandstede1997}.  We quote an important results of his below.  
\begin{lemma}[Sandstede \cite{sandstede1997}]\label{lem spec kawahara}
Let $\delta > -2$ and $Z_\delta$ be a homoclinic solution of \eqref{intro kawahara}, and consider the linearized operator $Q_\delta$ given by \eqref{linearized operator kawahara}.
\begin{enumerate}[label=\rm(\roman*)]
\item \label{negative evalue} $Q_\delta$ has at least one negative eigenvalue.
\item \label{simple evalue} If $Z_\delta$ is transversely constructed, then zero is a simple eigenvalue of $Q_\delta$. Moreover, when $\delta$ is varied, the number of negative eigenvalues remains constant until $Z_\delta$ ceases to be transversely constructed.
\item \label{spec at transversed soln} In particular, for $\delta \ge 0$ or $-1 \ll \delta < 0$ and consider $Z_\delta$ being a transversely constructed primary homoclinic orbit. Then $Q_\delta$ has exactly one negative eigenvalue. That is, the spectrum of $Q_\delta$ takes the form
\begin{equation*}
\essspectrum{Q_\delta} = [1, \infty), \qquad \spectrum{Q_\delta}= \{ -\widetilde \nu^2,\, 0 \} \cup \widetilde \Lambda
\end{equation*}
where $-\widetilde \nu^2 < 0$ and $0$ are both simple with corresponding eigenfuctions $\widetilde \phi_1$ and $\widetilde \phi_2 = Z_\delta^\prime$, respectively; and there exists $\nu_* > 0$ such that $\widetilde \Lambda \subset [\nu_*, \infty)$.
\end{enumerate}
\end{lemma}

With these provisions, we obtain the following theorem of the spectrum of the augmented potential in Region~C. The proof is very similar to the one for Theorem \ref{thm Qspectrum}, and hence we omit it. 
\begin{theorem}[Spectrum in Region C] \label{thm Qspectrum conditional}
Let the assumptions of Lemma \ref{lem limit rescaled} \ref{scaled Q in region C} hold. 
Let $\widetilde \nu, \nu_*$ and $\widetilde \phi_{1,2}$ given as in Lemma \ref{lem spec kawahara} \ref{spec at transversed soln}. Then 
for each $a \in (0, \nu_*)$ there exists some $\varepsilon_0 > 0$ such that for all $\varepsilon \in (0, \varepsilon_0)$ the operator $Q_\varepsilon(\eta_\varepsilon)$ satisfies
\[
\essspectrum{Q_\varepsilon(\eta_\varepsilon)} \subset [\gamma \varepsilon^4 c^2 \rho_-/d_+, \infty),  \qquad   \spectrum{Q_\varepsilon(\eta_\varepsilon)} = \{ - \nu^2, \, 0\} \cup \Lambda,
\]
where $\Lambda \subset [a\varepsilon^4 c^2 \rho_-/d_+, \infty)$, and 
\[
\nu^2 = \frac{\varepsilon^4 c^2 \rho_-}{d_+}\widetilde \nu^2 + o(\varepsilon^4) \qquad \textrm{as } \ \varepsilon \searrow 0.
\]
The first two eigenvalues $\nu_1 := -\nu^2$ and $\nu_2 := 0$ are simple with the associated eigenfunctions taking the form $\phi_i = S_\varepsilon \widetilde \phi_i + o(1)$ in $H^{k}$ as $\varepsilon \searrow 0$.
\end{theorem}

\subsection{Spectrum of the linearized augmented Hamiltonian} \label{spectrum linearized augHam section}

\begin{lemma}[Extension of $\Diff^2\augHam$] \label{extension of D^2 augHam lemma}
Let $\{ U_c \}$ be one of the family of bound states $\{ U_c^{\mathrm{A}} \}$, $\{ U_c^{\mathrm{A}\pm} \}$, or $\{ U_c^{\mathrm{C}} \}$ given by Corollaries~\ref{KdV bound state corollary}, \ref{Gardner bound state corollary} or Lemma~\ref{region C bound states lemma}\ref{region C kawahara part}, respectively.   Then  $\Diff^2 \augHam(U_c)$ extends uniquely to a bounded linear operator $\Hc : \Xspace \to \Xspace^*$ such that
\[   \Diff^2 \augHam(U_c)[\dot u,\, \dot v]  = \langle \Hc \dot u,\, \dot v\rangle_{\Xspace^* \times \Xspace} \qquad \textrm{for all } \dot u,\, \dot v \in \Vspace, \]
and $I^{-1} \Hc$ is self-adjoint on $\Xspace$.
\end{lemma}

\begin{proof} It suffices to consider the diagonal, so let a bound state $U_c = (\eta_c,\psi_c)$ and $\dot u = (\dot\eta,\dot\psi) \in \Vspace$ be given. By Lemmas~\ref{variations augV lemma} and \ref{quadratic form lemma}, we have that 
\begin{align*}
\Diff^2 \augHam(U_c)[\dot u,\dot u] & = \Diff^2 \augV(\eta_c)[\dot \eta, \dot \eta] + \int_{\mathbb{R}} (\DpsiT_c-\DpsiS_c) \dot\eta A(\eta_c) (\DpsiT_c-\DpsiS_c) \dot\eta \, \diffx \\
& \qquad +2 \Diff_\psi \Diff_\eta  \augHam(U_c)[\dot\eta,\dot\psi] + \Diff_\psi^2 \augHam(U_c)[\dot\psi,\dot\psi] \\
& = \langle \Qform(\eta_c) \dot\eta, \dot\eta \rangle_{\Xspace^* \times \Xspace} + \int_{\mathbb{R}} (\DpsiT_c-\DpsiS_c) \dot\eta A(\eta_c) (\DpsiT_c-\DpsiS_c) \dot\eta \, \diffx \\ 
& \qquad + 2 \int_{\mathbb{R}} \dot\psi \langle \Diff A(\eta_c) \dot\eta, \psi_c \rangle \, \diffx + 2c \int_{\mathbb{R}} \dot\eta^\prime \dot\psi \, \diffx + \int_{\mathbb{R}} \dot\psi A(\eta_c) \dot\psi \, \diffx ,
\end{align*} 
where we write $\DpsiS_c$ and $\DpsiT_c$ to indicate that these operators are being evaluated at $U_c$. The first derivative formula in Lemma~\ref{DG formula lemma} then gives
\begin{align*}
\Diff^2 \augHam(U_c)[\dot u,\dot u] & = \langle \Qform(\eta_c) \dot\eta, \dot\eta \rangle_{\Xspace^* \times \Xspace} + \int_{\mathbb{R}}  (\DpsiT_c-\DpsiS_c) \dot\eta  A(\eta_c) (\DpsiT_c-\DpsiS_c) \dot\eta \, \diffx + \int_{\mathbb{R}} \dot\psi A(\eta_c) \dot\psi \, \diffx  \\ 
& \quad + 2c \int_{\mathbb{R}} \dot\eta^\prime \dot\psi \, \diffx + 2 \sum_\pm \rho_\pm \int_{\mathbb{R}}\left(  (b_{1c}^\pm+c)  \left( \DN_\pm(\eta_c)^{-1} A(\eta_c) \dot \psi \right)^{\prime} \pm b_{2c}^\pm A(\eta_c) \dot\psi \right)  \dot\eta \, \diffx,
\end{align*}  
where $(b_{1c}^\pm, b_{2c}^\pm)$ is the relative velocity determined by $U_c$ via \eqref{def b}.  Recalling the definitions of $\DpsiS$ and $\DpsiT$ in \eqref{compute Dpsistar}, this can be expressed quite concisely as:
\be \label{Hc formula}
\Diff^2 \augHam(U_c)[\dot u,\dot u] = \langle \Qform(\eta_c) \dot\eta, \dot\eta \rangle_{\Xspace^* \times \Xspace} + \int_{\mathbb{R}}   \left( (\DpsiT_c-\DpsiS_c) \dot\eta + \dot\psi \right) A(\eta_c) \left( (\DpsiT_c-\DpsiS_c) \dot\eta + \dot\psi \right)  \, \diffx.
\ee
It is then clear that $\Diff^2 \augHam(U_c)$ extend to an element of $\Xspace^*$.
\end{proof}

We can now state and prove the main result of this section, which characterizes the spectrum of $\augHam$.  It corresponds to \cite[Assumption 6]{varholm2020stability}.

\begin{theorem}[Spectrum] \label{spectrum theorem}
Let $\{ U_c \}$ be one of the family of bound states $\{ U_c^{\mathrm{A}} \}$, $\{ U_c^{\mathrm{A}\pm} \}$, or $\{ U_c^{\mathrm{C}} \}$ given by Corollaries~\ref{KdV bound state corollary}, \ref{Gardner bound state corollary} or Lemma~\ref{region C bound states lemma}\ref{region C kawahara part}, respectively.  Then 
\[ \spectrum{I^{-1} \Hc} = \{ -\mu_c^2, \, 0\} \cup \Sigma_c, \]
where $-\mu_c^2 < 0$ is a simple eigenvalue corresponding to a unique eigenvector $\chi_c$; $0$ is a simple eigenvalue generated by $T$; and $\Sigma_c \subset (0,\infty)$ is bounded uniformly away from $0$.  
\end{theorem}

\begin{proof}
This follows from the structure of $I^{-1} \Hc$ and a soft analysis argument as in \cite[Proposition 5.3]{mielke2002energetic}.  Due either to Theorem~\ref{thm Qspectrum} or Theorem~\ref{thm Qspectrum conditional}, the operator 
 \[
 \Qform(\eta_c) + (\alpha - \nu_c^2) \langle \placeholder, \phi_{1c} \rangle  \phi_{1c} + \alpha \langle \placeholder, \eta_c^\prime\rangle \eta_c^\prime
 \]
 is positive definite for all $\alpha > 0$, where $-\nu_c^2$ is the negative eigenvalue of $Q_c(\eta_c)$ and $\phi_{1c}$ is the corresponding eigenfunction.  As $A(\eta_c)$ is itself positive definite,  from \eqref{Hc formula} we obtain the estimate
\[ \langle \Hc u, \, u \rangle_{\Xspace^* \times \Xspace} + (\alpha - \nu_c^2) \langle I^{-1} (\phi_{1c},0), u \rangle_{\Xspace^* \times \Xspace}^2 + \alpha \langle I^{-1} (\eta_c^\prime,0), \, u \rangle_{\Xspace^* \times \Xspace}^2 \gtrsim_c \| u \|_{\Xspace}^2,\]
for all $u \in \Xspace$.  Thus $I^{-1}\Hc$ is positive definite on a codimension $2$ subspace.  

On the other hand, we know that $T^\prime(0) U_c$ is in the kernel of $\Hc$, and by \eqref{Hc formula} we have that 
\[ \langle \Hc u, \, u  \rangle_{\Xspace^* \times \Xspace} = \langle \Qform(\eta_c) \phi_{1c}, \phi_{1c} \rangle_{\Xspace^* \times \Xspace} = -\nu_c^2 < 0 \qquad \textrm{for } u = (\phi_{1c}, (\DpsiS_c - \DpsiT_c) \phi_{1c}).\]
Thus $I^{-1} \Hc$ has a one-dimensional kernel generated by $T^\prime(0) U_c$, a one-dimensional negative definite subspace, and it is positive definite in the orthogonal complement.  The claimed spectral properties of $I^{-1} \Hc$ are now easily confirmed.   
\end{proof}

\section{Proof of the main results} \label{proof section}

Finally, in this section we will give the proof of the stability theorems discussed in Section~\ref{introduction section}.  In order to state them more concisely, we introduce the following notation.  For a fixed bound state $U_c$ and radius $r > 0$, we define the tubular neighborhoods
\begin{align*}
\tube_r^\Xspace & := \{ u \in \mathcal{O} : \inf_{s \in \mathbb{R}} \| u - T(s) U_c \|_{\Xspace} < r \}, \\
\tube_r^\Wspace & := \{ u \in \mathcal{O} \cap \Wspace : \inf_{s \in \mathbb{R}} \| u - T(s) U_c \|_{\Wspace} < r \}.
\end{align*}
Similarly, for any $R > 0$, let $\mathcal{B}_R^\Wspace$ denote the intersection of $\mathcal{O}$ with the ball of radius $R$ centered at the origin in $\Wspace$.  Then $U_c$ is said to be \emph{conditionally orbitally stable} provided that for all $r > 0$ and $R > 0$, there exists $r_0 > 0$ such that if $u : [0,t_0) \to \mathcal{B}_R^\Wspace$ is a solution to \eqref{Hamiltonian equation} with $u(0) \in \tube_{r_0}^\Xspace$, then $u(t) \in \tube_r^\Xspace$ for all $t \in [0,t_0)$.  On the other hand, we say that $U_c$ is \emph{orbitally unstable} provided that there exists $\nu_0 > 0$ such that, for all $0 < \nu < \nu_0$ there exists initial data in $\tube_\nu^\Wspace$ for which the corresponding solution exits $\tube_{\nu_0}^\Wspace$ in finite time.  

\subsection{Stability of uniform flows}

We begin with the simpler case of the trivial solution $U_c = (0,0)$, corresponding to a laminar flow with (the same) constant purely horizontal velocity in each layer.   For $(\beta,\lambda)$ in Region~B, we then have by Lemma~\ref{cont spec lemma} that $I^{-1} \Hc$ is positive definite.  Let us now state and prove a rigorous version of Theorem~\ref{parallel flow theorem}.   Because $U_c = 0$, the tubular neighborhoods above simply become balls in the appropriate spaces, and hence conditional orbital stability is equivalent to conditional stability.  
\begin{theorem}[Stability of uniform flows] \label{precise parallel theorem}
Let $U_c = (0,0)$ be the trivial bound state for the internal wave problem \eqref{Hamiltonian equation} with wave speed $c \in \mathbb{R}$.  Then $U_c$ is conditionally stable if the corresponding $(\beta,\lambda)$ lies in Region~B.
\end{theorem}
\begin{proof}
Because $\augHam$ is $C^\infty(\Vspace; \mathbb{R})$, $\augHam(U_c) = 0$, and $\Diff\augHam(U_c) = 0$, Taylor expanding it at $U_c$ gives
\[
\augHam(u) = \frac{1}{2} \langle \Hc u, \, u \rangle + O( \| u \|_{\Vspace}^3).
\]
For $(\beta,\lambda)$ in Region~B, we have by Lemmas~\ref{cont spec lemma} and \ref{extension of D^2 augHam lemma}  that $I^{-1} \Hc$ is positive definite on $\Xspace$.  On the other hand, the cubic term above can be controlled via Lemma~\ref{spaces lemma}:
\[ 
\| u \|_{\Vspace}^3 \lesssim  \| u \|_{\Wspace}^{1-\theta} \| u \|_{\Xspace}^{2+\theta} \leq r^\theta R^{1-\theta}  \| u \|_{\Xspace}^{2} \qquad \textrm{for all } u \in \tube_r^\Xspace \cap \mathcal{B}_R^\Wspace.
\]
Thus, for $r > 0$ sufficiently small, it holds that
\be \label{lower bound shear flow}
\augHam(u)  \geq \alpha \| u \|_{\Xspace}^2 \qquad \textrm{for all } u \in \tube_r^\Xspace \cap \mathcal{B}_R^\Wspace,
\ee
for some $\alpha = \alpha(r,R)> 0$.

Now, seeking a contradiction, suppose that $U_c$ is not conditionally stable.  Thus there exists $R > 0$, $r > 0$, and a sequence of initial data $\{ u_0^n\} \subset \mathcal{O} \cap \Wspace$ with $u_0^n \to 0$ in $\Xspace$ but for which the corresponding solution $u_n : [0,t_0^n) \to \mathcal{B}_R^\Wspace$ exits $\tube_r^\Xspace$ in finite time:
\[
\| u_n(\tau_n) \|_\Xspace = r \qquad \textrm{for some $\tau_n \in (0,t_0^n)$.}
\]
Let $\tau_n$ be the first such time and, if necessary, shrink $r$ so that \eqref{lower bound shear flow} holds.  
Together with the conservation of energy and momentum, this ensures that
\be \label{impossible inequality}
 \augHam(u_0^n) = \augHam(u_n(\tau_n)) \geq \alpha r^2  \qquad \textrm{for all } n \geq 1.
\ee

Because $\{ u_0^n \} \subset \mathcal{B}_R^\Wspace$ and $u_0^n \to 0$ in $\Xspace$, Lemma~\ref{spaces lemma} forces $u_0^n \to 0$ in $\Vspace$.  But then, the continuity of $\eng$ and $\mom$ would imply that $\augHam(u_0^n)$  also vanishes in the limit.  As this is in obvious contradiction with \eqref{impossible inequality}, the proof is complete.
\end{proof}

\subsection{Stability for strong surface tension} 

Next, we turn to the more complicated situation where the wave in question is small-amplitude but nontrivial.  Consider first the strong surface tension case corresponding to the waves in Region~A.  In Theorem~\ref{spectrum theorem}, it was shown that $I^{-1} \Hc$ has a negative direction in this regime, and so we will use the energy-momentum approach to show stability.  Having laid the groundwork for this argument in the previous sections, we are prepared to state and prove a precise version of Theorem~\ref{large beta theorem}.  

\begin{theorem}[Stability for strong surface tension] \label{precise strong surface tension theorem}
For all $c$ such that $0 < \lambda_c - \lambda_0 \ll 1$, the bound states $U_c^{\mathrm{A}}$ and $U_c^{\mathrm{A}\pm}$ given by Corollaries~\ref{KdV bound state corollary} and \ref{Gardner bound state corollary} are conditionally orbitally stable.  
\end{theorem}

\begin{proof}
Let $U_c$ stand for both $U_c^{\mathrm{A}}$ and $U_c^{\mathrm{A}\pm}$, as the first stage of the proof is identical in either case.  In Section~\ref{formulation section}, we confirmed that Assumptions $1$--$5$ of \cite{varholm2020stability} hold, and Assumption~$6$ was verified in Theorem~\ref{spectrum theorem}.  By \cite[Theorem 2.4]{varholm2020stability}, to prove that $U_{c_*}$ is conditionally orbitally stable we need only show that $\mi^{\prime\prime}(c_*) > 0$, where $\mi$ is \emph{moment of instability} defined by
\be\label{def moment}
\mi(c) := \augHam(U_c) = \eng(U_c) - c \mom(U_c).
\ee
Because $U_c$ is a critical point of $\augHam$, differentiating the above equation gives
\be \label{first derivative d formula}
\mi^\prime(c) = -\mom(U_c).
\ee
Thus we must confirm that $c \mapsto -\mom(U_c)$ is strictly increasing at $c = c_*$.  

The definition of the momentum \eqref{definition momentum} and kinematic condition \eqref{definition psistar} yield the explicit formula
\begin{equation*}
\begin{split}
\mi^{\prime}(c) & = \int_\R \eta_c' \psi_c \,\diffx =  c \int_\R \eta_c \partial_x  A(\eta_c)^{-1} \eta_c^\prime \,\diffx.
\end{split}
\end{equation*}
 As in Section~\ref{sec rescaling}, we will exploit a long-wave rescaling to analyze this quantity.  Recycling notation, let us redefine the scaling operator to be
\be\label{rescale T}
S_c f := f\left( \frac{\varepsilon_c \placeholder}{d_+ \sqrt{\beta_c - \beta_0}} \right),
\ee
where $\varepsilon_c = \varepsilon_c^{\mathrm{A}}$ and $\beta_c$ are given by \eqref{definition beta_c^A}.  Likewise, the asymptotics for the free surface profile established in \eqref{region A KdV scaling} and \eqref{region A Gardner scaling} permits us to write
\[ 
\eta_c =: \varepsilon_c^m d_+ S_c \left(  \widetilde\eta + \widetilde{r}_c \right) \qquad \textrm{for } \widetilde{r}_c = O(\varepsilon_c) \quad \textrm{in } H^k,
\]
with $m = 2$ for $U_c^{\mathrm{A}}$ and $m = 1$ for $U_c^{\mathrm{A}\pm}$. 
Using the rescaling, we compute that
\begin{equation*}
\begin{split}
\mi^{\prime}(c) & = c \varepsilon_c^{2m} d_+^2 \int_\R   (S_c (\widetilde\eta+\widetilde{r}_c)) \partial_x  A(\eta_c)^{-1} \partial_x  S_c (\widetilde\eta + \widetilde{r}_c ) \,\diffx \\
& = c \varepsilon_c^{2m-1} d_+^3 \sqrt{\beta_c - \beta_0}  \int_\R (\widetilde\eta+\widetilde{r}_c) S_c^{-1} \partial_x A(\eta_c)^{-1} \partial_x S_c (\widetilde\eta+\widetilde{r}_c)   \,\diffx \\
& = c \varepsilon_c^{2m-1} d_+^3 \sqrt{\beta_c - \beta_0}\sum_\pm \rho_\pm \int_\R (\widetilde\eta+\widetilde{r}_c)  S_c^{-1} \partial_x  \DN_\pm(\eta_c)^{-1} \partial_x S_c   (\widetilde\eta+\widetilde{r}_c) \,\diffx,
\end{split}
\end{equation*}
where the last line follows from \eqref{A inverse formula}.  Similar to \eqref{def NLop}, let us define
\[
\NLopc(\eta_c) := d_+ S_c^{-1} \partial_x \DN_\pm(\eta_c)^{-1} \partial_x S_c.
\]
Arguing as in Lemma~\ref{R asymptotics lemma}, we then find that 
\[
\left\| \NLopc(0) + \frac{d_+}{d_\pm} \right\|_{\Lin(H^2, L^2)} \lesssim \varepsilon_c^2, \qquad 
\left\| \NLopc(\eta_c) - \NLopc(0) \right\|_{\Lin(H^2, L^2)} \lesssim \varepsilon_c^{m},
\]
and hence
\begin{equation} \label{final deriv d equation}
\begin{aligned}
\mi^{\prime}(c) & = c \varepsilon_c^{2m-1} d_+^2 \sqrt{\beta_c - \beta_0} \sum_\pm \rho_\pm   \int_\R  \widetilde \eta \NLopc(0) \widetilde\eta \,\diffx + O(\varepsilon_c^{3m - 1}) &\qquad& \textrm{in } C^1(\mathscr{I}) \\
& = - c \varepsilon_c^{2m-1} d_+^2 \sqrt{\beta_c-\beta_0}  \sum_\pm \rho_\pm \frac{d_+}{d_\pm}  \int_\R  \widetilde{\eta}^2 \,\diffx + O(\varepsilon_c^{3m - 1}) &\qquad& \textrm{in } C^1(\mathscr{I}),
\end{aligned}
\end{equation}
where recall that $\mathscr{I}$ is a sufficiently small interval containing $c_*$.

Now, observe that $\varepsilon_c, \beta_c > 0$, and from \eqref{definition beta_c^A}, 
\[
c\mapsto \varepsilon_c \textrm{ and } c \mapsto c \beta_c \signum{c} \textrm{ are both positive and } \left\{ \begin{aligned} \textrm{strictly decreasing for } & c > 0, \textrm{ and }  \\ \textrm{strictly increasing for } & c < 0, \end{aligned} \right.
\]
Therefore $c \mapsto -c \varepsilon_c^{2m-1} (\beta_c-\beta_0)^{1/2}$ is strictly increasing. This completes the proof for the family $\{ U_c^{\mathrm{A}} \}$, as $\widetilde\eta$ is independent of $c$ in that case.  

The argument for $\{U_c^{\mathrm{A}\pm}\}$ is only slightly more complicated.  Recall that by \eqref{region A Gardner scaling},
\be\label{etatilde form}
\widetilde\eta = \widetilde\eta_c^{\mathrm{A}\pm}(x) = \frac{1}{\kappa_c \pm \sqrt{\kappa_c^2 + 4(\varrho + h)} \cosh x},
\ee
with $\kappa_c = \kappa_c^{\mathrm{A}}$ defined as in Corollary~\ref{Gardner bound state corollary}.  Since we are in fact computing $\int_\R \widetilde\eta^2 \,\diffx$, it is sufficient to assume that $\kappa_c > 0$.  Then, clearly $\widetilde\eta_c^{\mathrm{A}+} > 0$ and $c \mapsto \kappa_c \signum{c}$ is increasing, so we again have by \eqref{final deriv d equation} and the argument in the previous paragraph that $\mi^{\prime}$ is strictly increasing.   
Finally, $\widetilde\eta_c^{\mathrm{A}-}$ is a wave of depression and 
an explicit computation using \eqref{etatilde form} gives
\[
\int_\R (\widetilde\eta_c^{\mathrm{A}-})^2 \,\diffx = \frac{\kappa_c \tan^{-1}\left( \frac{\kappa_c + \sqrt{\kappa_c^2 + 4(\varrho + h)}}{4(\varrho + h)} \right)}{2(\varrho + h)^{3/2}} - \frac{1}{2(\varrho + h)}.
\]
It is easily seen that the right-hand side above is strictly increasing in $c$ for $c > 0$ and strictly decreasing for $c < 0$.   The proof is therefore complete.
\end{proof}

\subsection{Stability for near critical surface tension}

Consider now the families of bound states $\{U_c^{\mathrm{C}}\}$ that correspond to traveling waves in Region~C.  Recall from Section~\ref{traveling wave section}, that to leading order, the corresponding free surface profiles are rescalings of the family of primary homoclinic orbits $\{ Z_\delta \}$ of the ODEs \eqref{Z1 ODE}.  To unify the presentation, we will write $Z_c$ as shorthand for  $Z_{\delta_c^{\mathrm{C}}}$.

The next theorem shows that under the hypothesis of Theorem~\ref{thm Qspectrum conditional}, the orbital stability/instability of these waves can be inferred purely from properties of the primary homoclinic orbits.  

\begin{theorem}[Stability for critical surface tension]
Consider the family of traveling waves $\{U_c^{\mathrm{C}}\}$ given in Lemma~\ref{region C bound states lemma}~\ref{region C kawahara part} and assume that the hypothesis of Theorem \ref{thm Qspectrum conditional} holds.  For all $c_*$ with $0 < \lambda_{c_*} - \lambda_0 \ll 1$, the corresponding wave is conditionally orbitally stable provided that the function
\be \label{d prime region c} 
c \mapsto \signum{c} \int_\R  Z_{c}^2 \,\diffx \quad \textrm{is strictly increasing at } c_*,
\ee
and it is orbitally unstable if this function is strictly decreasing there.  
\end{theorem}
\begin{proof}
Throughout the argument, we abbreviate $\{U_c\}$ for $\{U_c^{\mathrm{C}}\}$ and $\varepsilon_c = \varepsilon_c^{\mathrm{C}}$.  We have already proved in Theorem~\ref{spectrum theorem} that the spectral hypothesis on $I^{-1} \Hc$ in \cite[Assumption 6]{varholm2020stability} holds.  As in the previous subsection, we may therefore apply \cite[Theorem~2.4]{varholm2020stability} to conclude that $U_{c_*}$ is conditionally orbitally stable provided that $d^{\prime\prime}(c_*) > 0$, where $d$ is the moment of instability \eqref{def moment}.  On the other hand, because the Cauchy problem is locally well-posed, \cite[Assumption 7]{varholm2020stability} is satisfied, and so \cite[Theorem 2.6]{varholm2020stability} tells us that $U_{c_*}$ is orbitally unstable if $d^{\prime\prime}(c_*) < 0$.

From Lemma~\ref{region C bound states lemma}, we know that free surface profile takes the form
\[
\eta_c = \varepsilon_c^4 d_+ S_c\left( Z_c + \widetilde{r}_c \right) \qquad \textrm{with } \widetilde{r}_c = O(\varepsilon_c) \textrm{ in $H^k$ as } \varepsilon \searrow 0,
\]
where we have redefined the scaling operator to be $S_c f := f(\varepsilon \placeholder/d_+)$.
The same argument as in the proof of Theorem~\ref{precise strong surface tension theorem} reveals that
\[ 
d^\prime(c) = - c \varepsilon_c^{7} d_+^2 \gamma \sum_\pm \rho_\pm \frac{d_+}{d_\pm}  \int_\R  Z_c^2 \,\diffx + O(\varepsilon_c^{11}) \quad \textrm{in } C^1(\mathscr{I}).
\]
From the definition of $\varepsilon_c$ in \eqref{definition epsilon C region}, 
\[
c\mapsto \varepsilon_c \textrm{ and } c \mapsto c \varepsilon^2_c \signum{c} \textrm{ are both positive and } \left\{ \begin{aligned} \textrm{strictly decreasing for } & c > 0, \textrm{ and }  \\ \textrm{strictly increasing for } & c < 0, \end{aligned} \right.
\]
Thus $c \mapsto -c \varepsilon_c^{7}$ is strictly increasing. Therefore $d^{\prime}(c)$ is strictly increasing at $c_*$ when \eqref{d prime region c} is satisfied. 
\end{proof}

We remark that \eqref{d prime region c} is stated in terms of the wave speed $c$, but to compare it to results on dispersive model equations of Kawahara type \eqref{intro kawahara} it is natural to consider the related function $\delta \mapsto \int Z_\delta^2 \, \diffx$.  Looking carefully at its definition in \eqref{definition epsilon C region}, we see that  $c \mapsto \delta_c^{\mathrm{C}}$ can be both increasing or decreasing depending on the various physical parameters.  

\section*{Acknowledgments}
The research of RMC is supported in part by the NSF through DMS-1907584.  The research of SW is supported in part by the NSF through DMS-1812436.  The authors would also like to thanks Dag Nilsson for enlightening communications regarding the existence theory in Section~\ref{traveling wave section}, and Daniel Sinambela for close readings of earlier versions of the manuscript.


\appendix

\section{Elementary identities} \label{identities appendix}

\begin{proof}[Proof of Lemma~\ref{second variation A lemma}]
As in the proof of Lemma~\ref{DG formula lemma}, we start by considering the corresponding formula for $A(\eta)^{-1}$.  Recalling \eqref{A inverse formula}, we see that 
\[ \begin{split} 
 D^2 (A(\eta)^{-1})[ \dot \eta, \dot \eta] & = \sum_{\pm} \rho_\pm \Diff^2 (\DN_\pm(\eta)^{-1})[ \dot\eta, \dot\eta] \\
 & = -\sum_{\pm} \rho_\pm \DN_\pm(\eta)^{-1} \left( \Diff^2\DN_\pm(\eta)[\dot\eta,\dot\eta] - 2 \Diff\DN_\pm(\eta)[\dot\eta] \DN_\pm(\eta)^{-1} \Diff\DN_\pm(\eta)[\dot\eta]\right) \DN_\pm(\eta)^{-1}.
 \end{split}
 \]
 On the other hand, we have the elementary identity 
\be \label{D^2 A identity} 
\begin{split}
 \Diff^2 A(\eta)[ \dot\eta,\dot\eta] 
& = -A(\eta) \Diff^2(A(\eta)^{-1})[\dot\eta,\dot\eta] A(\eta) + 2\Diff A(\eta)[\dot\eta] A(\eta)^{-1} \Diff A(\eta)[\dot\eta].
\end{split}
\ee
Together, these will furnish a representation formula for the second variation of $A(\eta)^{-1}$ once we have fully expanded these expressions using \eqref{first derivative DN formula} and \eqref{second derivative DN formula}.  

Consider each of the terms on the right-hand side of \eqref{D^2 A identity}.  For the first, we have
\begin{align*}
-\int_{\mathbb{R}} \psi A(\eta) \Diff^2 (A(\eta)^{-1})[\dot\eta,\dot\eta] A(\eta) \psi \,\diffx & = \sum_{\pm} \rho_\pm \int_{\mathbb{R}}  \theta_\pm     \Diff^2 \DN_\pm(\eta)[\dot\eta,\dot\eta] \theta_\pm \,\diffx    \\
& \qquad -2 \sum_{\pm} \rho_\pm \int_{\mathbb{R}}  \theta_\pm \Diff\DN_\pm(\eta)[\dot\eta] \DN_\pm(\eta)^{-1} \Diff\DN_\pm(\eta)[\dot\eta]  \theta_\pm \,\diffx,
\end{align*}
where  recall that $\theta_\pm = \theta_\pm(\eta, \psi)$ is given by \eqref{def theta and a4}.  Throughout the remainder of the proof, $a_i^\pm$ will always be evaluated at $(\eta, \theta_\pm)$, so we suppress the arguments for readability.  By the first variation \eqref{first derivative DN formula} and second variation \eqref{second derivative DN formula}  formulas for $\DN_\pm(\eta)$, this becomes
\begin{align*}
-\int_{\mathbb{R}} \psi A(\eta) \Diff^2 (A(\eta)^{-1})[\dot\eta,\dot\eta] A(\eta) \psi \,\diffx & = \sum_\pm \rho_\pm \int_{\mathbb{R}}  \left( a_4^\pm \dot\eta^2 + 2a_2^\pm \dot\eta \DN_\pm(\eta) \left( a_2^\pm \dot\eta \right) \right) \,\diffx \\
& \qquad -2 \sum_{\pm} \rho_\pm \int_{\mathbb{R}} a_1^\pm \left(  \DN_\pm(\eta)^{-1} \Diff\DN_\pm(\eta)[\dot\eta]\theta_\pm \right )^\prime \dot\eta \,\diffx  \\
&  \qquad -2 \sum_\pm \rho_\pm \int_{\mathbb{R}}  a_2^\pm \left( \Diff\DN_\pm(\eta)[\dot\eta] \theta_\pm \right) \dot\eta \,\diffx \\
& = \sum_\pm \rho_\pm \int_{\mathbb{R}}  \left( a_4^\pm \dot\eta^2 + 2a_2^\pm \dot\eta \DN_\pm(\eta) \left( a_2^\pm\dot\eta \right) \right) \,\diffx \\
& \qquad -2 \sum_\pm \rho_\pm \int_{\mathbb{R}} \mathscr{L}_\pm[\dot\eta] \Diff\DN_\pm(\eta)[\dot \eta] \theta_\pm  \,\diffx,
\end{align*}
for the linear operator $\mathscr{L}_\pm$ given by \eqref{def script L}.
Using \eqref{first derivative DN formula} once more allows us to simplify this to
\begin{align*}
-\int_{\mathbb{R}} \psi A(\eta) \Diff^2 (A(\eta)^{-1})[\dot\eta,\dot\eta] A(\eta) \psi \,\diffx & = \sum_\pm \rho_\pm \int_{\mathbb{R}}  \left( a_4^\pm \dot\eta + 2a_2^\pm \DN_\pm(\eta) \left( a_2^\pm \dot\eta \right) \right) \dot\eta \,\diffx \\
& \qquad -2 \sum_{\pm} \rho_\pm \int_{\mathbb{R}} \left( a_1^\pm \mathscr{L}_\pm[\dot \eta]^\prime + a_2^\pm \DN_\pm(\eta) \mathscr{L}_\pm[\dot\eta] \right) \dot \eta \,\diffx.
\end{align*} 
So finally we have
\be \label{D^2A calculation first term}
-\int_{\mathbb{R}} \psi A(\eta) \Diff^2 (A(\eta)^{-1})[\dot\eta,\dot\eta] A(\eta) \psi \,\diffx  =  \int_{\mathbb{R}}  \Big( a_4 \dot\eta + 2\sum_\pm \rho_\pm a_2^\pm \DN_\pm(\eta) \left( a_2^\pm \dot \eta \right) - 2 \mathscr{M}\dot\eta  \Big) \dot\eta \,\diffx 
\ee
where recall $a_4 = a_4(\eta,\psi)$ and $\mathscr{M} = \mathscr{M}(\eta,\psi)$ were defined in \eqref{def theta and a4} and \eqref{def script M}, respectively.   

Likewise, the second in term on the right-hand side of \eqref{D^2 A identity} can be treated as follows.  Using \eqref{first derivative A formula}, we calculate that
\begin{align*}
\int_{\mathbb{R}} \psi \Diff A(\eta)[ \dot\eta] A(\eta)^{-1} \Diff A(\eta)[\dot\eta] \psi \,\diffx & = \sum_{\pm} \rho_\pm\int_{\mathbb{R}}  \Big(  a_1^\pm \left( \DN_\pm(\eta)^{-1} \Diff A(\eta)[\dot\eta]\psi \right)^\prime \\
& \qquad\qquad\qquad +a_2^\pm A(\eta) \Diff A(\eta)[\dot\eta] \psi \Big) \dot\eta \,\diffx \\
& = \sum_\pm \rho_\pm \int_{\mathbb{R}} \mathscr{L}_\pm[\dot\eta] \Diff A(\eta)[\dot\eta] \psi \,\diffx = \int_{\mathbb{R}} \mathscr{L}[\dot\eta] \Diff A(\eta)[\dot\eta] \psi \,\diffx.
\end{align*}
Applying \eqref{first derivative A formula} once more then yields
\be \label{script N identity} 
\begin{split}
\int_{\mathbb{R}} \psi \Diff A(\eta)[ \dot\eta] A(\eta)^{-1} \Diff A(\eta)[\dot\eta] \psi \,\diffx & = \sum_{\pm}  \rho_{\pm} \int_{\mathbb{R}} \Big( a_1^{\pm} \left( A(\eta) \DN_{\pm}(\eta)^{-1} \mathscr{L}[\dot\eta] \right)^\prime  \\
& \qquad \qquad\qquad +a_2^{\pm} A(\eta) \mathscr{L}[\dot\eta] \Big) \dot\eta \,\diffx \\
& = \int_{\mathbb{R}} \dot\eta\mathscr{N} \dot\eta \,\diffx,
\end{split}
\ee
with $\mathscr{N} = \mathscr{N}(\eta,\psi)$ defined in \eqref{def script N}.  Combining this with \eqref{D^2 A identity} and \eqref{D^2A calculation first term} gives the formula \eqref{second derivative A formula}, completing the proof.  
\end{proof}

\bibliographystyle{siam}
\bibliography{internalwaves}

\end{document}